\documentclass[11pt,reqno,english]{article}

\usepackage{amsmath,amsfonts,amssymb,graphicx,amsthm,url}
\usepackage[noadjust]{cite}
\usepackage{stmaryrd}
\usepackage{mathrsfs,booktabs,tabularx}
\usepackage{xifthen,xcolor,tikz,setspace}
\usetikzlibrary{decorations.pathmorphing,patterns,shapes,calc,decorations}
\usetikzlibrary{decorations.pathreplacing}
\usepackage{mathtools,upgreek,paralist}
\usepackage[final]{showkeys} 

\definecolor{refkey}{gray}{.75}
\definecolor{labelkey}{gray}{.5}
\usepackage[algo2e,boxed,vlined,algoruled]{algorithm2e}
\usepackage{comment}
\usepackage[shortlabels]{enumitem}

\usepackage[letterpaper]{geometry}
\usepackage[colorinlistoftodos]{todonotes}
\presetkeys{todonotes}{inline, color=green}{}

\usepackage{accents}
\usepackage[algo2e,boxed,vlined,algoruled]{algorithm2e}

\usepackage{authblk}

\usepackage[small]{caption}
\usepackage[colorlinks=true, citecolor=blue]{hyperref}
\colorlet{DarkGreen}{green!50!black}
\colorlet{DarkGray}{gray!60!black}

\numberwithin{equation}{section}

\renewcommand{\epsilon}{\varepsilon}

\newcommand{\one}{\mathbf{1}}

\usepackage{color}
 \definecolor{refkey}{gray}{.5}
 \definecolor{labelkey}{gray}{.5}
\definecolor{light}{gray}{.9}

\usepackage{soul}

\usepackage[capitalise]{cleveref}

\newtheorem{thm}{Theorem}

\newtheorem{theorem}{Theorem}[section]
\newtheorem*{theorem*}{Theorem}
\newtheorem{lemma}[theorem]{Lemma}

\crefname{claim}{Claim}{Claims}
\newtheorem{proposition}[theorem]{Proposition}

\newtheorem{observation}[theorem]{Observation}
\newtheorem{fact}[theorem]{Fact}
\newtheorem{corollary}[theorem]{Corollary}
\newtheorem{cor}[theorem]{Corollary}

\theoremstyle{definition}{
\newtheorem{example}[theorem]{Example}
\newtheorem{definition}[theorem]{Definition}

\newtheorem*{definition*}{Definition}

\newtheorem{remark}[theorem]{Remark}
\newtheorem*{remark*}{Remark}

}

\newcommand{\E}{\mathbb E}

\renewcommand{\P}{\mathbb P}

\newcommand{\R}{\mathbb R}
\newcommand{\Z}{\mathbb Z}

\newcommand{\cC}{\ensuremath{\mathcal C}}
\newcommand{\cD}{\ensuremath{\mathcal D}}
\newcommand{\cE}{\ensuremath{\mathcal E}}
\newcommand{\cF}{\ensuremath{\mathcal F}}
\newcommand{\cG}{\ensuremath{\mathcal G}}

\newcommand{\cL}{\ensuremath{\mathcal L}}

\newcommand{\cR}{\ensuremath{\mathcal R}}
\newcommand{\cS}{\ensuremath{\mathcal S}}
\newcommand{\cT}{\ensuremath{\mathcal T}}

\newcommand{\cY}{\ensuremath{\mathcal Y}}

\newcommand{\sT}{{\ensuremath{\mathscr T}}}

 \renewcommand{\epsilon}{\varepsilon}

\DeclareMathOperator{\diam}{diam}

\newcommand{\tv}{{\textsc{tv}}}

\def\sT{{\mathsf T}}
\newcommand{\op}{\mathrm{op}}
\newcommand{\bone}{\mathbf{1}}
\newcommand{\eps}{\varepsilon}
\newcommand{\wR}{\widetilde{\mathcal{R}}}
\newcommand{\wC}{\widetilde{\mathcal{C}}}
\newcommand{\wE}{\widetilde{\mathcal{E}}}
\newcommand{\wF}{\widetilde{\mathcal{F}}}
\newcommand{\wEE}{\wE_{\vec{u},\vec{t},\vec{C}}}
\newcommand{\wFF}{\widetilde{\mathcal{F}}_{\vec{u}, \vec{t}}}
\newcommand{\wJ}{\widetilde{\mathcal{J}}}
\newcommand{\wX}{\widetilde X}

\newcommand{\betap}{\beta_{\textnormal{p}}}

\title{Rapid phase ordering for Ising and Potts dynamics \\ on random regular graphs}
\date{}

\author{Reza Gheissari\thanks{Department of Mathematics, Northwestern University. Evanston, IL. \url{gheissari@northwestern.edu}}, Allan Sly\thanks{Department of Mathematics, Princeton University. Princeton, NJ. \url{allansly@princeton.edu}}, Youngtak Sohn\thanks{Division of Applied Mathematics, Brown University. Providence, RI. \url{youngtak_sohn@brown.edu}}}

\begin{document}

\maketitle

\vspace{-1.25cm}
\begin{abstract}
    We consider the Ising, and more generally, $q$-state Potts Glauber dynamics on random $d$-regular graphs on $n$ vertices at low temperatures $\beta \gtrsim \frac{\log d}{d}$. The mixing time is exponential in $n$ due to a bottleneck between $q$ dominant phases consisting of configurations in which the majority of vertices are in the same state. We prove that for any $d\ge 7$, from biased initializations with $\epsilon_d n$ more vertices in state-$1$ than in other states, the Glauber dynamics quasi-equilibrates to the stationary distribution conditioned on having plurality in state-$1$ in optimal $O(\log n)$ time. Moreover, the requisite initial bias $\epsilon_d$ can be taken to zero as $d \to \infty$. 
    Even for the $q=2$ Ising case, where the states are naturally identified with $\pm 1$, proving such a result requires a new approach in order to control negative information spread in spacetime despite the model being in low temperature and exhibiting strong local correlations. For this purpose, we introduce a coupled non-Markovian \emph{rigid dynamics} for which a delicate temporal recursion on probability mass functions of minus spacetime cluster sizes establishes their subcriticality.  
\end{abstract}

\section{Introduction}
The out-of-equilibrium dynamics of spin systems like the Ising and Potts models are very well-studied in a variety of fields. One of their notable features is that at low-temperatures, the dynamics are slow to equilibrate due to a bottleneck between configurations that are mostly in one state, versus mostly in another. A central question is in what sense this is the ``only" obstruction to fast relaxation of the low-temperature dynamics. This question can be posed in the following form: if you initialize with a small bias towards one ground state, does the low-temperature dynamics converge rapidly to the restriction of the Gibbs distribution to the corresponding phase? 

In the mathematical physics literature, such questions have been studied numerically for a long time, with extensive predictions for the convergence of the Ising dynamics from random initializations, in which it is expected that the magnetization diffuses away from $0$ and picks one  phase to dominate, at which point quasi-equilibration within that phase is rather fast: see the famous survey of~\cite{bray1994theory} on the theory of phase ordering kinetics. The question is also relevant to theoretical computer science as, when performing low-temperature sampling with Markov Chain Monte Carlo, it is common to initialize randomly, either from a product measure, or from some high-temperature initialization (as happens in simulated annealing~\cite{Kirkpatrick-simulated-annealing}). 
The Ising and Potts models on non-trivial geometries serve as the natural baselines with which to put these predictions on a rigorous footing.

Mathematically, a major obstruction to answering these kinds of questions is the fact that many of the tools for bounding mixing times (e.g., functional inequalities like spectral gap and bottleneck sets~\cite{LP}, and the more modern tools of spectral independence~\cite{Anari-spectral-independence} and localization schemes~\cite{ChenEldan}) are useful for controlling  mixing from worst-case initialization using local decay of correlation inputs, but in contexts where the worst-case mixing time is exponentially slow due to strong local correlations, showing fast mixing from classes of ``nice" initializations can be very challenging. 

Indeed, at positive temperatures, despite the very extensive literature on mixing times for Ising Glauber dynamics in the last thirty or more years (see e.g., the survey of~\cite{Martinelli-notes}, and Chapter 15 of~\cite{LP}), the set of results of the above form can be summarized as follows. Firstly, on the complete graph, the dynamics is essentially one-dimensional, being characterized by the birth-and-death chain of the magnetization: it was shown in~\cite{LLP} that at all low-temperatures, the mixing time of continuous-time Glauber dynamics restricted to positive magnetization is $O(\log n)$. On the infinite tree,~\cite{CaMaTree} showed that from i.i.d.\ $\text{Rad}(1-\varepsilon)$ initializations, at all sufficiently low temperatures the dynamics converges locally to the plus Gibbs measure. 
For graphs with more complicated geometry, like the random $d$-regular graphs and the integer lattice, it was shown a few years ago in~\cite{GhSi22} that at low-temperatures, the quasi-equilibration to the plus phase is fast from the \emph{all}-plus initialization. This specific initialization gives access to monotonicity tools, since it stochastically dominates the stationary measure, whereas even a very biased product initialization does not. In the Potts case, apart from the complete graph case where mixing from random initializations was recently studied in~\cite{BlGhZh24}, the literature is even more limited due to the absence of monotonicity. Indeed, proving polynomial quasi-equilibration from just the identically-state-$1$ initialization on random graphs was posed as an open problem in~\cite{Fast-polymer-dynamics}. 

The one regime where there is more known as pertains to convergence times from biased, but non-monochromatic initializations on graphs with non-trivial geometry is in the zero-temperature limit of the Ising dynamics. The zero-temperature Ising Glauber dynamics is also known as the majority dynamics, or the voter model, an interacting particle system of significant interest in its own right: see e.g., the book~\cite{Liggett-book}. In that context, the analogue of rapid quasi-equilibration to the plus phase would be rapid fixation at the all-plus configuration (which is absorbing). On the lattice $\mathbb Z^d$,~\cite{FoScSi02} showed that when initialized from a sufficiently biased i.i.d.\ initialization, the zero-temperature Ising dynamics converges to the all-plus configuration. The bias was shown in~\cite{Morris-zero-temp} to be allowed to go to zero as $d\to\infty$. In the past several years, there has been a lot of activity around the same question for zero-temperature dynamics on dense random graphs, with particular focus on how small the bias can be as a function of the average degree~\cite{GNS-zero-temp-Ising-RG,Tran-Vu-majority-dynamics,Fountoulakis-majority-dynamics,SahSawhney-majority-dynamics}. Interestingly, as the random graph gets sparse ($O(1)$-average degree), near-zero magnetization initializations become delicate as the energy landscape of the Ising model on random graphs near zero-magnetization is expected to be ``glassy", and  there exist trapping configurations for the zero-temperature dynamics~\cite{Benjamini-majority-dynamics}. 

We finally mention some works that have studied mixing of low temperature dynamics when forced to stay in narrow subsets of the state space. One such chain is the Kawasaki dynamics where the Ising dynamics is forced to have a fixed magnetization, and updates are made by swapping adjacent $+1$ and $-1$ spins. While sampling from this distribution is hard on general maximum-degree $d$ graphs at low-temperatures~\cite{kuchukova_et_al:LIPIcs.APPROX/RANDOM.2024.56}, the recent paper~\cite{BaBoDa-Kawasaki} showed that for every magnetization, it mixes quickly as long as $\beta \lesssim \frac{1}{\sqrt{d}}$ (well into the low-temperature regime on random regular graphs). In another direction,~\cite{Fast-polymer-dynamics} showed that Ising and Potts dynamics are fast to mix at sufficiently low temperatures on random $d$-regular graphs if constrained to remain close to the monochromatic configuration not just in the sense of total spin count, but in the much stronger sense of never allowing connected components of the non-dominant state bigger than $c\log n$. These constrained chains can be used to sample from the Ising and Potts distributions quickly, but do not imply anything for the unconstrained Glauber dynamics because with such strong constraints, the standard Glauber chain leaves the subset of the state space more quickly than the restricted chain mixes.

In this paper, we study the low (but positive) temperature Ising and Potts Glauber dynamics on random $d$-regular graphs initialized from biased initializations and demonstrate that they quasi-equilibrate in optimal $O(\log n)$ time to the corresponding metastable distribution, i.e., the Potts distribution conditioned on having plurality in state-$1$.  

\subsection{Main results}
We begin by presenting our results for the Ising Glauber dynamics. The Ising model on a graph $G = (V,E)$ on $|V| = n$ vertices, at inverse temperature $\beta>0$ is the following probability distribution over assignments $\Omega = \{- 1, +1\}^n$: 
\begin{align}\label{eq:Ising-measure}
	\pi_{G,\beta}(\sigma) \propto \exp\Big(  \beta \sum_{v\sim w} \sigma_v \sigma_w\Big)\,.
\end{align}
The (continuous-time) Glauber dynamics $(X_t)_{t\ge 0}$ for the Ising model on $G$ is the Markov chain that is initialized from $X_0\in \Omega$, and
\begin{enumerate}
\item Assigns each vertex $v\in V$ a rate-$1$ Poisson clock;
\item If the clock at vertex $v$ rings at time $t$, then it generates $X_t$ from $X_{t^-}$ by resampling   
\begin{align}\label{eq:Ising-transition-rate}
    X_t(v) \sim \pi_{G,\beta}\big(\sigma_v\in \cdot  \mid (\sigma_{w})_{w\ne v} = (X_{t^-}(w))_{w\ne v}\big) \,,\end{align}
and leaving $X_t(w) = X_{t^-}(w)$ for all $w\ne v$. 
\end{enumerate} 
We denote the law of this Markov chain as $\mathbb P_{X_0}(X_t \in \cdot)$.

We consider the Glauber dynamics of the Ising model on graphs $G$ drawn from the uniform distribution over $d$-regular graphs on $n$ vertices,  which we denote by $G\sim \mathcal G_d(n)$. 
The Ising model on $G\sim \mathcal G_d(n)$ undergoes a phase transition at $\beta_c(d) := \tanh^{-1}(1/(d-1))$ (which goes to zero as $1/d$ as $d\to\infty$). While for high temperatures $\beta<\beta_c(d)$, the worst-case mixing time (the time to be at total-variation distance at most $1/4$ to the stationary distribution from a worst-case initialization) is an optimal $O(\log n)$~\cite{MS}, for $\beta>\beta_c(d)$, the worst-case mixing time is slow, $\exp(\Theta(n))$, because of a bottleneck between configurations with a majority plus,
\begin{align}\label{eq:plus-phase}
    \Omega^+ = \{\sigma: m(\sigma) \ge 0\}\,, \qquad \text{where} \qquad m(\sigma) = \frac{1}{n} \sum_{v\in V} \sigma_v\,,\end{align}
 and ones with a majority minus, $\Omega^-$~\cite{DeMo10}. This in particular means that if we initialize from the \emph{plus phase}, by which we mean the Gibbs distribution conditioned on being in $\Omega^+$, i.e., $\pi^+= \pi (\cdot \mid \Omega^+)$, the time for the Markov chain to hit $\Omega^-$ is exponentially long. 
 Exactly at the critical point $\beta_c(d)$, polynomial mixing time was recently shown by combining~\cite{BaDa-LSI} with~\cite{Weitz} (see Example 6.19 of~\cite{BaBoDa-notes}).

Our first main theorem is that at all sufficiently low temperatures, the Ising Glauber dynamics initialized from a configuration $X_0$ with a bias to the plus phase, i.e., $m(X_0) \ge \epsilon$, (quasi-)equilibrates to $\pi^+$ in $O(n\log n)$ time steps. When we write $o_d(1), \asymp_d$, we mean as $d\to\infty$, while we use little-$o$, big-$O$ notations without subscripts to mean as $n\to\infty$.

\begin{thm}
    \label{thm:main}
    For every $d\ge 7$, there exists constants  $C(\epsilon, \beta, d)>0$, $\epsilon_0(d)\in (0, 1)$ and $\beta_0<\infty$ with $\epsilon_{0}(d) \asymp_d \frac{1}{\log d}$ and $\beta_0 \asymp_d \frac{\log d}{d}$, such that for every $\epsilon\in (\epsilon_0, 1]$ and every $\beta>\beta_0$,  if $G \sim \mathcal G_{d}(n)$, the following holds with  probability $1-o(1)$. If $X_0$ has $m(X_0) \ge \epsilon$, then 
    \begin{align*}
        \|\mathbb P_{X_0}(X_t \in \cdot) - \pi_G(\,\cdot \mid \Omega^+)\|_\tv \le n^{-10}\,, \qquad \text{for all} \qquad \text{$C \log n \le t\le e^{n/C}$}\,.
    \end{align*}
    By symmetry, we have the same bound on the distance to $\pi_G^- = \pi_G(\cdot \mid \Omega^-)$ if $X_0$ has $m(X_0)\le - \epsilon$. 
\end{thm}

Let us make some comments on the different parameter requirements in the theorem. Firstly, we expect the analogous theorem to hold for any $d\ge 3$, but the condition $d\ge 7$ is a barrier to our proof method, and getting down to $d\ge 3$ would require some new ideas to more carefully handle short cycles in $G$ that are all-minus by chance. We leave this to future investigation. 

In the other direction, we can say the following about the large-$d$ asymptotics of  $\varepsilon_0,\beta_0$ and $C$. The requirement $\beta_0 \gtrsim_d \frac{\log d}{d}$ is the threshold below which under $\pi^+$, the set of minus spins, has a giant component, preventing use of the locally treelike geometry to confine minus regions of the stationary distribution. 
The rate of decay of the minimal bias, $\varepsilon_0 \asymp_d \frac{1}{\log d}$, can be improved to $\epsilon_0 \asymp_d \frac{1}{\sqrt{d}}$ at the expense of $\beta_0 \asymp_d \frac{1}{\sqrt{d}}$, but we note that for any fixed $d$, for sufficiently small $\varepsilon$, there may exist configurations with magnetizations near $\varepsilon$ that in the $\beta\to\infty$ limit trap the dynamics; in the combinatorics literature, these are known as (near-)friendly bisections (see e.g.~\cite{Ferber-friendly}). Therefore, the treatment at very small $\epsilon$ would need to subtly address the ability of positive temperature dynamics to escape even though zero-temperature dynamics may not. The constant $C$ in the $C\log n$ time for quasi-equilibration can be taken to be $1+o_d(1)$ as $d$ gets large.  

Finally, the exponential upper bound on $t$ is necessary because after exponential time, the Glauber dynamics fully mixes and therefore gives probability $1/2$ to $\Omega^-$ while $\pi^+_G(\Omega^-) =0$. 

\subsubsection{The Potts Glauber dynamics case}
Let us now describe our results for the Potts model. 
The $q$-state Potts model on a graph $G = (V,E)$ on $|V| = n$ vertices, at inverse temperature $\betap>0$ is the following probability distribution over assignments $\Omega = \{1,...,q\}^n$: 
\begin{align}\label{eq:Potts-measure}
	\pi_{G,\betap,q}(\sigma) \propto \exp\Big(  \betap \sum_{v\sim w} \mathbf 1\{\sigma_v = \sigma_w\}\Big)\,.
\end{align}
The $q=2$ case is naturally associated to the Ising model up to the transformation $\betap = 2\beta$, and the mapping of state-$1$ to $+1$ and state-$2$ to $-1$. 

The Potts Glauber dynamics $(Y_t)_{t\ge 0}$ are defined exactly as in the Ising dynamics, except in the resampling step~\eqref{eq:Ising-transition-rate}, it is now with the Potts distribution $\pi_{G,\betap,q}(\sigma_v \in \cdot \mid (\sigma_w)_{w\ne v} = (Y_{t^-}(w))_{w\ne v})$. Like the Ising dynamics, the Potts Glauber dynamics undergoes a slowdown at a tree uniqueness point $\beta_u(q,d)\asymp_d\frac{1}{d}$ due to bottlenecks between $q$ distinct phases corresponding to configurations where one of the $q$ states dominates. In fact, when $q\ge 3$, there is an intermediate regime of temperatures $\beta_u(d)<\betap < \beta_s(d)$ with $\beta_s(d)\asymp_d \frac{1}{d}$ where an additional disordered phase is metastable along with the $q$ ordered ones (see e.g.~\cite{CGGRSV-Metastability-Potts-ferromagnet}). 

The ordered phase where state-$1$ dominates can be described as follows using a generalized notion of magnetization (which matches~\eqref{eq:Potts-measure} when $q=2$): 
\begin{align}\label{eq:Potts-phase-magnetization}
    \Omega^1 = \{\sigma: m(\sigma)\ge 0\}\,, \qquad \text{where} \qquad m(\sigma) = \frac{1}{n} \Big( \sum_{v\in V} \mathbf{1} \{\sigma_v = 1\} - \max_{i\in \{2,...,q\}} \sum_{v\in V} \mathbf 1\{\sigma_v = i\}\Big)\,,
\end{align}
i.e., the plurality of spins are in state-$1$. We then define the conditional distribution $\pi^1 = \pi(\cdot \mid \Omega^1)$ where state-$1$ dominates. 
Theorem~\ref{thm:main} is the $q=2$ case of the following more general theorem on phase ordering of low-temperature $q$-state Potts Glauber dynamics. 

\begin{thm}\label{thm:main-Potts}
    For all $q\ge 2, d\ge 7$, there exists $\beta_0 \asymp_d \frac{\log (qd)}{d}$ and $\epsilon_0(\betap, d)  \asymp_d \max\{ \frac{1}{\sqrt{d}}, \frac{1}{\betap d}\}<~1$ such that the following holds for all $\betap>\beta_0$ and $\epsilon >\epsilon_0$. There exists $C(\epsilon,\betap,d,q)$ such that if $G\sim \mathcal G_d(n)$, with probability $1-o(1)$, for every initialization $Y_0$ with $m(Y_0) \ge \epsilon$, 
    \begin{align*}
        \|\mathbb P_{Y_0} (Y_t \in \cdot ) - \pi_G(\, \cdot \mid \Omega^1)\|_{\textsc{tv}} \le n^{-10}\,, \qquad \text{for all} \qquad \text{$C \log n \le t \le e^{ n/C}$\,.}
    \end{align*}
\end{thm}

By symmetry, the same holds for quasi-equilibration to $\pi^i$ for $i\in \{2,...,q\}$ if the Potts dynamics is initialized with $\epsilon n$ more vertices in state-$i$ than in any other state.  

\begin{figure}
\begin{tikzpicture}
    \node at (-4.25,0) {\includegraphics[width = 0.47\textwidth]{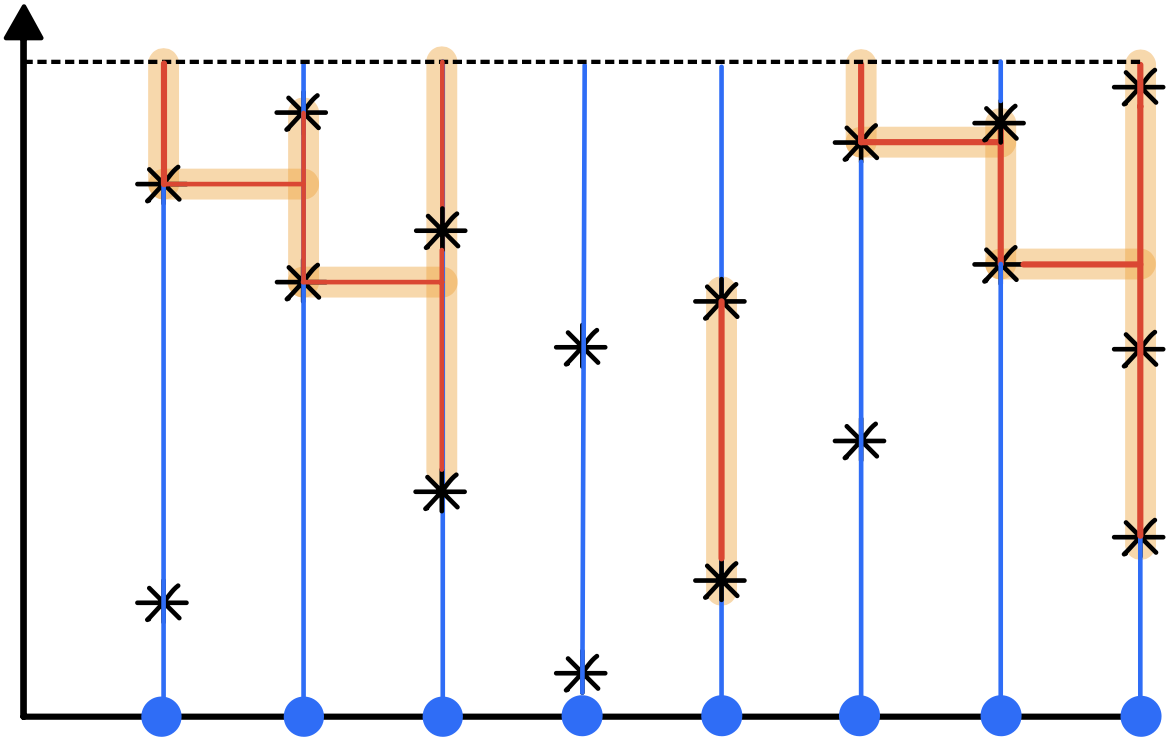}};
    \node at (4.25,0) {\includegraphics[width = 0.47\textwidth]{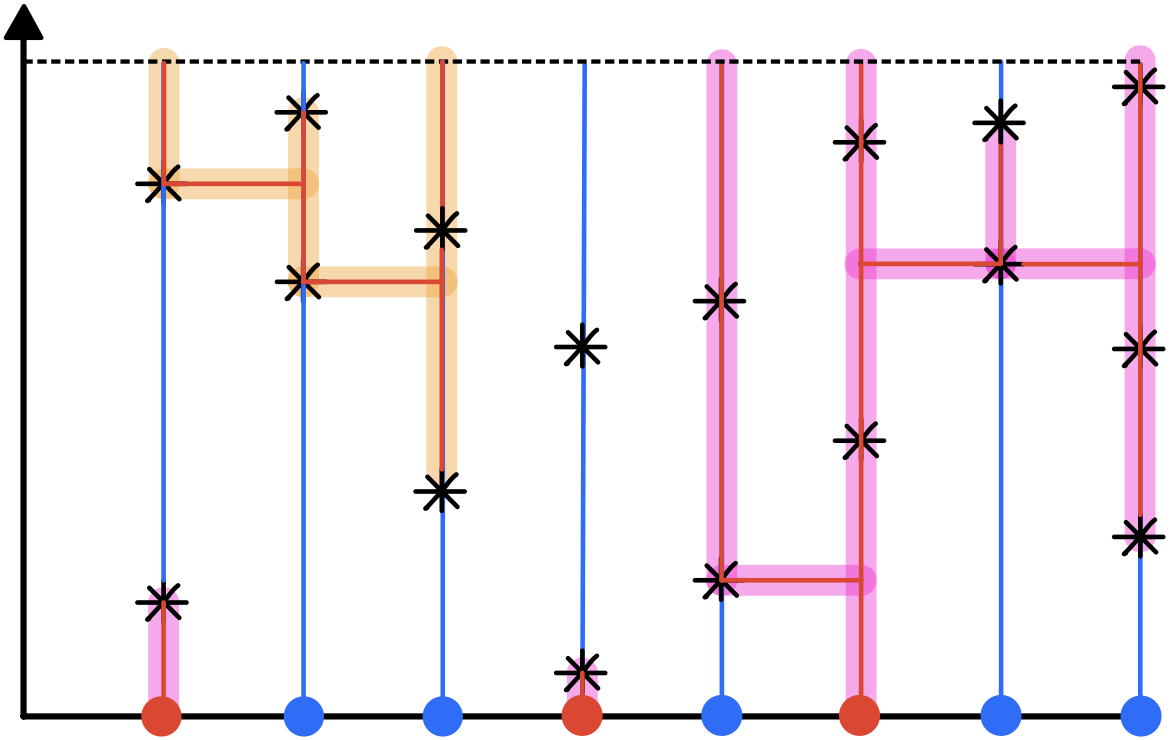}};
    \node at (-4.25, -2.7) {$G$}; 
    \node at (-7.7, 2.35) {$t$};
    \node at (-4.25, 2.5) {$(X_t^+)_{t\ge 0}$}; 
    \node at (4.25, -2.7) {$G$}; 
    \node at (8.5-7.7, 2.35) {$t$};
    \node at (4.25, 2.5) {$(X_t^{x_0})_{t\ge 0}$}; 
\end{tikzpicture}
\caption{Spacetime diagrams for coupled Ising dynamics initialized from all-plus (left) and biased initialization $x_0$ (right). On the left, the minus spacetime clusters of $(X_t^+)_{t\ge 0}$ are highlighted orange; on the right the legacy spacetime cluster (union of minus clusters intersecting initial minus sites) is highlighted in magenta.}\label{fig:different-initialization-minus-clusters}
\end{figure}
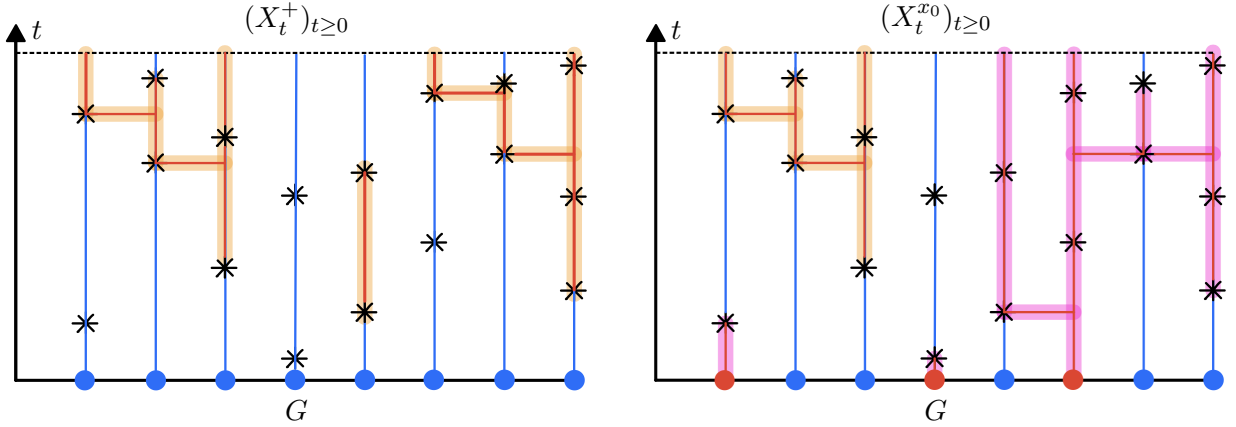

\subsection{Proof ideas}
At a high level, the main challenge of establishing fast mixing of low-temperature dynamics when restricted to a metastable phase, is that non-dominant states retain a strong local influence. This causes approaches based on ``high-temperature" tools that entail worst-case analysis (e.g., spatial-to-temporal mixing implications, path coupling, information percolation and spectral independence) to break down even in the  restricted state space of $\{m(\sigma)\ge 1-\gamma\}$. We present our proof sketch in the case of the Ising dynamics, as most of the mathematical novelty of the paper is already needed in for proving that special case; we discuss the proof of the Potts case at the end.  

\subsubsection{Evolution of the legacy region in spacetime}
The starting point of our approach, is to directly bound the coupling time of the chain $X_{t}^{x_0}$ initialized from $x_0 \in \{\sigma: m(\sigma) \ge 1-\gamma\}$ for $\gamma>0$ sufficiently small (but independent of $\beta$), to the chain $X_t^+$ started from all-plus. This implies coupling to a quasi-stationary chain started from $\pi^+$ by a triangle inequality. 
We do so by studying their evolutions via spacetime diagrams under the standard grand coupling of Glauber dynamics chains, wherein $X_t^+(v) \ge X_t^{x_0}(v)$ for all $v\in V$ and all $t\ge 0$. In the spacetime diagram, an element of $\{-1,+1\}^{G\times [0,t]}$,  has \emph{minus spacetime clusters} $\cC_{u, \le t}$, which are minus connected components in $G\times [0,t]$ containing the spacetime point $(u,t)$. Since the full spacetime realization of $X_t^+$ is more plus than that of $X_t^{x_0}$, and these are Markov processes, it can be shown that as soon as all minus clusters containing the initial minus set $\{v: x_0(v) = -1\}$ are extinct, from that time forth, $X_t^{x_0}$ and $X_t^+$ are coupled perfectly. In this sense, minus spacetime clusters confine the negative information interior to them; in particular, the following spacetime subset confines all negative information contained in the initialization $x_0$.  

\begin{definition}
     The \emph{legacy cluster} $\cL_{\le t}$ in $(X_t^{x_0})_{t\ge 0}$ is the union of all minus spacetime clusters incident to $\mathcal S_0 =\{v: x_0(v) = -1\}$, and its intersection with time-$t$ is denoted $\cL_t$.     
\end{definition}

\noindent Figure~\ref{fig:different-initialization-minus-clusters} depicts spacetime realizations, minus spacetime clusters, and the legacy cluster.

With this setup, the aim of the paper is to understand the evolution of minus spacetime clusters and establish that $|\cL_t|$ has negative drift so that with high probability $\cL_{C \log n} = \emptyset$. Heuristically, it is natural to view this as a low-temperature analogue of the information percolation approach to coupling high-temperature Ising chains pioneered in~\cite{LS-information-percolation}, with plus sites in the spacetime diagram playing the role of oblivious updates that shield information. 

Towards bounding the extinction time of the legacy region $\cL_t$, consider how it evolves in time. Roughly, if the vertex $v$ is being updated at time $t$ in configuration $X_{t}^{x_0}$ with legacy region $\cL_t$, 
\begin{enumerate}
    \item If $v\in \cL_{t}$, it can flip from $-1$ to $+1$ and if it does so, $\cL_t = \cL_{t^-} \setminus \{v\}$. Note that $X_t^+(v)=+1$ by monotonicity, so the disagreement between $X_t^{x_0}$ and $X_t^+$ is removed at $v$. 
    \item If $v$ is at distance at least two from $\cL_{t^-}$, then $\cL_t$ does not change from a flip at $v$.
    \item If $v$ is adjacent to $\cL_{t^-}$, then if $v$ flips, $\cL_t$ grows by absorbing (the time-$t$ slice of) the minus spacetime clusters containing $(w_i,t)$ for $w_i$ adjacent to $v$. 
\end{enumerate}
See Figure~\ref{fig:evolution-of-the-legacy} for a depiction of the evolution of the legacy spacetime cluster over time through a flip of the form of item (3) above. 
Notice that when $\cL_t$ grows by absorbing the minus spacetime clusters adjacent to $v$, it is adding more than just connected sets of minuses in the time-$t$ configuration: it also may add other minus sites that are connected through their past and therefore share negative information with $(w_i,t)$. 

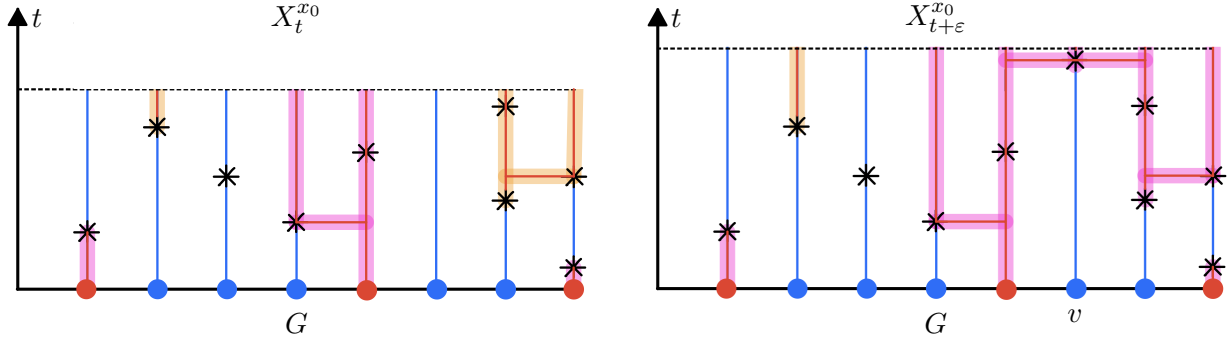
\begin{figure}
\begin{tikzpicture}
    \node at (-4.25,0) {\includegraphics[width = 0.47\textwidth]{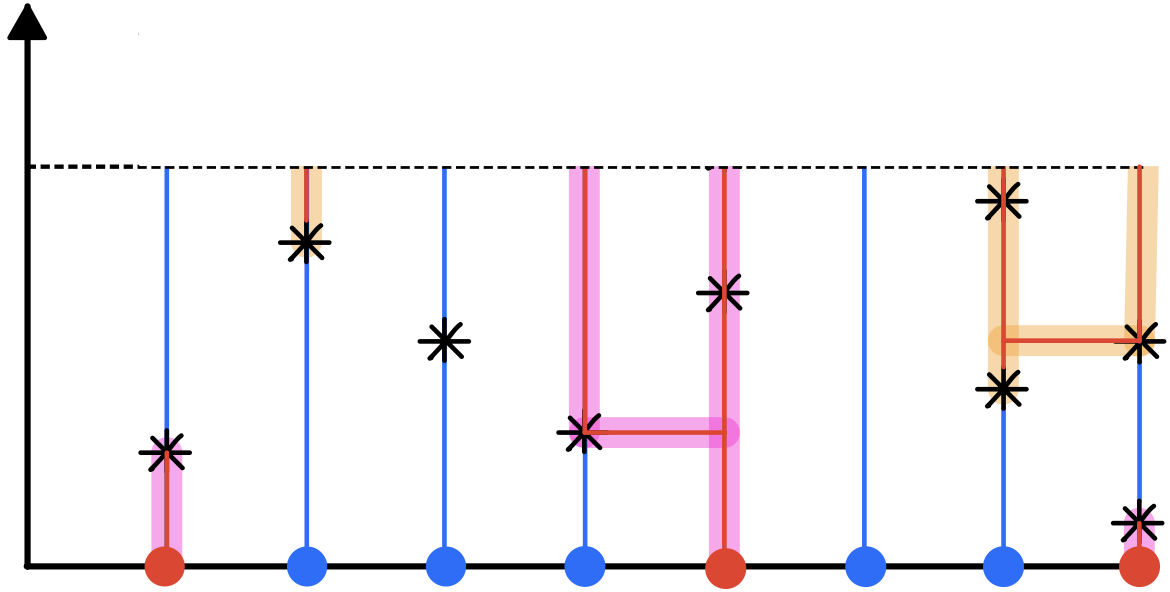}};
    \node at (4.25,0) {\includegraphics[width = 0.47\textwidth]{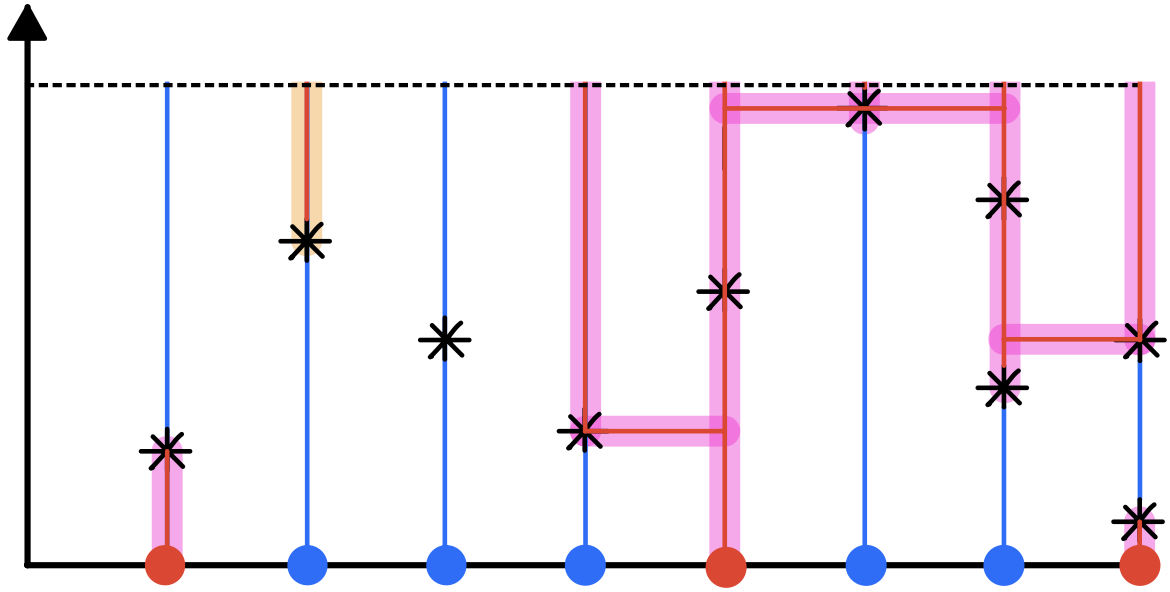}};
        \node at (-4.25, -2.25) {$G$}; 
    \node at (-7.7, 1.8) {$t$};
    \node at (-4.25, 1.8) {$X_t^{x_0}$}; 
    \node at (4.25, -2.25) {$G$}; 
    \node at (8.5-7.7, 1.8) {$t$};
    \node at (4.25, 1.8) {$X_{t+\epsilon}^{x_0}$}; 
    \node at (6.1, -2.15) {$v$}; 
\end{tikzpicture}
\caption{Snapshots of an Ising dynamics spacetime diagram initialized from biased initialization $x_0$ at times $t$ (left) and $t+\epsilon$ (right) demonstrating the evolution of the legacy cluster $\cL_{\le t}$ (highlighted purple). When the vertex $v$ adjacent to $\cL_{t}$ flips from $+1$ to $-1$ at time $s\in [t,t+\epsilon)$, the legacy cluster absorbs that vertex along with all minus spacetime clusters that are adjacent to $v$ at time $s$.}\label{fig:evolution-of-the-legacy}
\end{figure}

In order to establish negative drift for $|\cL_t|$, we then need to show that the expected growth due to flips of type (3) loses to the negative drift from type (1) flips. At this stage, the fact that minus regions confine the negative information interior to them is essential to bound the size distributions of minus regions distinct from $\cL_{\le t}$, conditional on $\cL_{\le t}$. Using this, and the fact that all minus vertices of the initialization belong to $\cL_{\le t}$, we show that conditional on $\cL_{\le t}$, the law of a minus spacetime cluster $\cC_{w_i,\le t}$ distinct from $\cL_{\le t}$ is stochastically dominated by the unconditional law of a minus spacetime region in the Glauber dynamics $X_t^+$ with all-plus initialization.

\subsubsection{Exponential tails on minus spacetime regions in the all-plus chain}

Given the above context, the bulk of the work of the paper is to establish that (at least for polynomial timescales) minus spacetime regions have exponential tails on their sizes (say, the number of vertices ever to be in a cluster, or in its intersection with time-$t$) under $(X_t^+)_{t\ge 0}$. We emphasize that unlike minus connected components at a fixed time, which can be shown to have exponential tails by expansion properties of the graph, the strong temporal correlations pose a significant challenge that cannot be dealt with in that way. 

A first attempt to establish exponential tails on minus spacetime clusters $\cC_{u,\le t}$ would go by showing that they have negative drift using a similar drift analysis to what was sketched for the legacy cluster's evolution, together with the fact that an inductively assumed exponential tail bound can be used on the disjoint minus spacetime clusters through which $\cC_{u,\le t}$ can grow. However, this approach hits a barrier due to the possibility that a vertex $u$ is $+1$ and therefore has $\cC_{u,\le t} = \emptyset$, but is very likely to flip to minus because a majority of its neighbors are from the same minus spacetime cluster. Indeed the influx of mass to non-empty minus clusters at $u$ from scenarios like this would exceed the negative drift of a non-empty minus spacetime cluster. 

The solution to this problem of ``holes" in minus clusters that are likely to flip back to minus is the novel introduction of a \emph{rigid dynamics} which is a coupled non-Markovian variant of the standard Glauber dynamics that does not allow flips at vertex $v$ from $-1$ to $+1$ if the vertex $v$ is what we call a \emph{trifurcation point}, which means it has three or more subtrees in its local neighborhood that intersect its minus spacetime cluster.

The benefit of the rigid dynamics is that it is more minus than the standard Glauber dynamics, and one can show that while minus spacetime clusters are $O(\log n)$ sized, no vertex in the rigid dynamics has more than three minus neighbors from the same minus spacetime cluster. Therefore if $d\ge 7$, it would require two \emph{distinct} minus spacetime clusters to have formed independently for it to have a majority of $-1$ neighbors and be likely to flip to minus at $\beta$ large. 

The core of the proof then becomes utilizing these properties and obtaining a bounding system of differential equations for the time evolutions of the probability mass function $\rho_t(k)  = \mathbb P(|\cC_{u,\le t}|=k)$ of the size of a minus spacetime cluster under the rigid dynamics. Roughly speaking, the differential equations will give negative drift $\frac{d}{dt} \rho_t(k) <0$ for  $k\le \frac{1}{4} \log_d n$ if $\rho_t(k)$ gets close to an exponentially decaying threshold $\psi_k = d^{-Ck}$. However, in implementing this there are several more delicate steps, in particular the fact that this probability mass function bound has to be attained conditional on realizations of other disjoint minus spacetime regions. We will comment more on these details and why they are necessary within the relevant proof segments (see Remark~\ref{rmk:proof-conditional}).

\subsubsection{Extension to the Potts model}
In order to apply the above strategy to the Potts dynamics, we consider a two-spin  dynamics $(X_t)_{t\ge 0}$ which is $-1$ whenever a coupled Potts Glauber dynamics $(Y_t)_{t\ge 0}$ is not in its dominant state-$1$. 
We establish a coupling of Potts Glauber dynamics chains $(Y_t^{y_0})_{t\ge 0}, (Y_t^1)_{t\ge 0}$ with the two-spin dynamics $(X_t^{x_0})_{t \ge 0}$ initialized from $x_0= x_0(y_0)$ which is $+1$ if and only if $y_0$ is $1$, such that when the legacy cluster of $(X_t^{x_0})_{t \ge 0}$ goes extinct, the Potts Glauber chains have successfully~coupled. 

The two-spin process can be defined in such a way that its update probabilities are monotone, and if a vertex $v$ being updated has a strict majority $+1$ neighbors, then it has a sufficiently high probability of taking state $+1$. Notably, these two properties were all we needed about the Ising dynamics to establish $O(\log n)$ extinction time of the legacy region in the argument sketched above. 

\begin{remark}
    The minimal assumptions on the family of low-temperature two-spin dynamics for which we establish $O(\log n)$-time extinction of the legacy minus cluster from initializations sufficiently biased towards plus makes our framework promising for application to other ferromagnetic models. For example, it is directly applicable to the \emph{noisy majority model} for small enough noise $\eps>0$ (see Example~\ref{ex:noisy-majority}). The dynamics of the noisy majority model are known to be challenging to study due to its non-reversibility and lack of explicit stationary distribution.  
\end{remark}

\subsubsection{Outline}
In Section~\ref{sec:minus-regions}, we formally define the minus spacetime clusters, the notion of trifurcation points and rigid dynamics, and outline properties of these combinatorial constructions. In Section~\ref{sec:exp-tails-minus-spacetime-regions} we perform the temporal recursion on the probability mass function for sizes of minus spacetime clusters under the rigid dynamics: this is the technical core of the paper. In Section~\ref{sec:legacy-analysis}, we use the exponential tails we proved on minus spacetime regions, to show negative drift on the legacy region, as long as it has size less than $\frac{n}{1000}$. This establishes the claims of Theorems~\ref{thm:main}--\ref{thm:main-Potts} so long as the initial bias is greater than $1-\gamma_0$ for a universal constant $\gamma_0>0$. Finally, in Section~\ref{sec:magnetization-analysis}, we use a significantly simpler analysis of the drift of the magnetization process at large $d$ to improve to biases that are~$o_d(1)$.

 \subsection*{Acknowledgements}
R.G.\ thanks Shirshendu Ganguly and Eyal Lubetzky for illuminating conversations related to this problem. The research of R.G.\ is supported in part by NSF grant DMS-2246780. A.S. was partially supported by a Simons Investigator grant and conducted this research while a visitor at MIT and Stanford.

\section{Confining negative information spread with plus initialization}\label{sec:minus-regions}
In this section, we construct minus spacetime clusters and their time-$t$ slices, called minus regions, via a graphical spacetime representation of two-spin dynamics. We then establish that these minus spacetime clusters are internally measurable, so that conditioning on them in future parts of the argument is tractable. From there, we introduce the dominating process known as the rigid dynamics and describe important combinatorial properties this modified dynamics retains.

\subsection{A general family of low-temperature two-spin dynamics}
We consider a family of Markov processes on $\{\pm 1\}^{n}$ that include the Ising Glauber dynamics, but more generally can be used to bound the excitations in a Potts model from its dominant color. In fact, we will present minimal assumptions on the $\{\pm 1\}^n$-valued continuous-time chain for which our analysis applies, which can generically be used to control growth of thermal deviations (encoded in minuses) away from a certain dominant state (encoded by the all-plus one). 

\begin{definition}\label{def:general-ising-process}
    In what follows, suppose that $(X_t^{x_0})_{t\ge 0}$ is a continuous-time Markov process on $\{\pm 1\}^n$ initialized from $x_0$, which makes updates as follows: 
\begin{itemize}
    \item Assign each vertex an independent rate-$1$ Poisson clock; 
    \item When the clock at vertex $v$ rings, keep $X_t(w) = X_{t^-}(w)$ for $w\ne v$, and resample 
    \begin{align}\label{eq:update-rule}
        X_t(v) = \begin{cases}
            +1  & \quad \text{with probability $p_+((X_{t^-}(w))_{w\sim v}$}\\ -1 & \quad \text{else}
        \end{cases}\,,
    \end{align}
    where we assume that $p_+: \{\pm 1\}^d \to [0,1]$ satisfies the following. 
    \begin{enumerate}[(A)]
    \item $p_{+}$ is a non-decreasing function on $\{\pm 1\}^d$. 
    \item If $|\{1\le i\le d: \eta_i = +1\}|\ge 4d/7$ and $d\ge 7$, then $p_+(\eta) \ge \frac{1}{1 +  e^{ - 2\beta d/7} }$. 
    \end{enumerate}
\end{itemize}
\end{definition}

We will refer to a Markov process $(X_t)_{t\ge 0}$ in the form of Definition~\ref{def:general-ising-process} as the \emph{two-spin dynamics}. 

\begin{example}\label{ex:Ising}
The continuous-time Ising Glauber dynamics has 
\begin{align*}
    p_+(\eta)  = \frac{1}{1+e^{ - 2\beta \sum_{i=1}^d \eta_i}}\,,
\end{align*}
and therefore satisfies properties (A), (B). For concreteness, the reader can keep this in mind as the process $X_t^{x_0}$ throughout, noting that (A), (B) are the only properties of that Markov processes jump rates that are actually used in our proofs. 
\end{example}

\begin{example}\label{ex:Potts}
For the application to the $q$-state Potts dynamics,  
if $\betap$ denotes the inverse temperature of the Potts Glauber dynamics, then we utilize a dominating two-spin dynamics with $p_{+}$ chosen as follows. Let $\betap=2\beta+\frac{C\log (q-1)}{d}$, and for $\eta=(\eta_i)_{i\in [d]}\in \{\pm 1\}^{d}$, let
\begin{align}\label{eq:Potts-two-spin-chain}
p_{+}(\eta)=\begin{cases}
            \frac{1}{1+e^{-2\beta d/7}}  & \quad \text{if}~~~~~\sum_{i=1}^{d}\one \{\eta_i=+1\}\geq \frac{4d}{7}
            \\ 0 & \quad \text{else}
        \end{cases}\,.
\end{align}
Items (A), (B) are immediate for these transition probabilities. The following lemma will ensure that the set of vertices not in state-$1$ in a Potts Glauber dynamics $(Y_t)_{t\ge 0}$ are a subset of the set of minus vertices of $(X_t)_{t\ge 0}$ with initialization such that $+1$ if $Y_0(v) =1$ and $-1$ if $Y_0(v)\in \{2,...,q\}$. 
\begin{lemma}\label{lem:potts-domination}
    If $\betap=2\beta+\frac{C\log (q-1)}{d}$, then for all $\zeta \in [q]^d$ such that $\sum_{i=1}^{d}\one \{\zeta_i=1\}\geq \frac{4d}{7}$,
    \[
  \frac{e^{ \betap \sum_{i=1}^d \mathbf 1\{\zeta_i = 1\}}}{\sum_{r\in [q]} e^{ \betap \sum_{i=1}^d \mathbf 1\{\zeta_i = r\}}}\geq \frac{1}{1+e^{-2\beta d/7}}.
    \]
\end{lemma}
\begin{proof}
    Let $k_r=\sum_{i=1}^{d}\one\{\zeta_i=r\}$ for $r\in [q]$, and without loss of generality, assume $k_1\geq \ldots \ge k_q$.   Since $x\mapsto e^{\beta x}$ is convex, Karamata's inequality shows that fixing $k_1$, the denominator is maximized when $k_2=d-k_1$ and $k_3=\ldots k_q=0$. Thus, if $k_1\geq \frac{4d}{7}$,
    \[
    \frac{e^{\betap k_1}}{\sum_{r\in [q]}e^{\betap k_r}}\geq \frac{e^{\betap k_1}}{e^{\betap k_1}+e^{\betap(d-k_1)}+q-2}\geq \frac{1}{1+e^{-\betap d/7}+(q-2)e^{-4\betap d /7}}.
    \]
    Finally, note that $e^{-\betap d/7}+(q-2)e^{-4\betap d /7}\leq e^{-2\beta d/7}$ holds whenever $\betap-2\beta \geq \frac{7\log (q-1)}{d}$.
\end{proof}
\end{example}

\begin{example}\label{ex:noisy-majority}
    The noisy majority dynamics with odd-degree $d$ at noise parameter $p$ is the continuous-time Markov chain that has $p_+(\eta) = 1-p$ if $\sum_{i=1}^d \eta_i >0$, and $p_+ = p$ if $\sum_{i=1}^d \eta_i <0$. If $d$ is even and $\sum_{i=1}^d \eta_i = 0$, let $p_+(\eta) = \frac{1}{2}$. It is easy to see that this satisfies conditions (A),~(B) so long as $p<e^{ - 2\beta d/7}$. Thus, all our theorems on fast coupling of biased initializations to the plus initialization chain, etc., could be applied to this Markov process. We mention this example because it is a non-reversible Markov chain whose non-reversibility constrains the applicability of spectral and functional analytic tools. 
\end{example}

\subsection{Spacetime constructions and minus regions}

The starting point is the graphical spacetime representation of the two-spin dynamics, which is also used in the grand coupling of multiple two-spin chains which we will return to in Section~\ref{sec:legacy-analysis}.

\begin{definition}\label{def:grand-coupling}
    For all initializations $x_0 \in \{\pm 1\}^n$, the family of continuous-time two-spin chains $(X_t^{x_0})_{x_0\in \Omega, t\ge 0}$ can be generated as follows: initialize $X_0^{x_0} = x_0$, and 
    \begin{enumerate}
        \item Assign every vertex $v\in V$ a rate-$1$ Poisson clock whose ring times can be enumerated $0<t_{v,1} <t_{v,2}<...$, and a sequence of i.i.d.\ $\text{Unif}[0,1]$ random variables $U_{v,1},U_{v,2},...$. 
        \item At time $t  = t_{v,i}$, from the configuration $X_{t^-}^{x_0}$, generate $X_{t}^{x_0}$ by setting 
        \begin{align}\label{eq:update-in-terms-of-uniform}
            X_{t}^{x_0}(v) = \begin{cases}
                -1 & 0\le U_{v,i} \le  1-p_+((X_{t^-}^{x_0}(w))_{w\sim v}) \\ 
                +1 & \text{else}
            \end{cases}\,.
        \end{align}
        and $X_{t}^{x_0}(w) = X_{t^-}^{x_0}(w)$ for all $w\ne v$. 
    \end{enumerate}
    Let $\cF_{t}$ denote the filtration generated by all clock ring times at or before time $t$, along with all uniform random variables associated to those clock rings. We use the standard notation of $\cF_{t^-} = \lim_{s\uparrow t} \cF_s$, and more generally for cadlag processes, $Z_{t^-}= \lim_{s\uparrow t} Z_s$. 
\end{definition}
It is clear that  this is a valid coupling of the two-spin dynamics chains $(X_t^{x_0})_{x_0 \in \Omega,t\ge 0}$, and furthermore, the monotonicity of the two-spin dynamics ensures that if $x_0 \le x_0'$ in the partial order on $\Omega$, then $X_{t}^{x_0} \le X_{t}^{x_0'}$ for all $t\ge 0$. These are properties we will return to, but for this section we are purely interested in understanding properties of the chain $X_t^+$ that is initialized from the all-plus initialization. 
For that all-plus initialization chain, we define the key notion of \emph{minus spacetime clusters} at time $t$, which will be $\cF_t$-measurable. 

\begin{definition}\label{def:minus-spacetime-clusters}
    Construct the minus subset $\mathcal S_{\le t} = \mathcal S_{\le t}((X_s^+)_{s\in [0,t]})$ of $V\times [0,t]$ by setting 
    \begin{align*}
        \mathcal S_{\le t} = \bigcup  \{(v,s) \in V\times [0,t]:  X_t^+(v) =-1\}\,.
    \end{align*}
    A \emph{minus spacetime cluster} $\mathcal C_{\le t}$ is a maximal connected component of $\mathcal S_{\le t} \subset G \times [0,t]$, where connectedness is through $G$ spatially, and through $\mathbb R_+$ temporally. 
\end{definition}

\begin{definition}\label{def:minus-region}
    A minus spacetime cluster $\mathcal C_{\le t}$ induces its present slice, which we call a \emph{minus region} $\mathcal R_t$ which will be $\mathcal C_{\le t} \cap (V\times \{t\})$ viewed as a vertex subset $\mathcal R_t \subset V$.  
\end{definition}

The spacetime clusters intersecting $V\times \{t\}$ can be indexed by the vertex $v$ that they contain, meaning for a vertex $v\in V$, we write $\mathcal{C}_{v,\le t}$ to mean the (possibly empty) spacetime cluster containing $(v,t)\in V\times [0,t]$. The vertex projection of $\cC_{v,\le t}$ denoted $\cC_{t}(v) \subset V$ is the set of all vertices to have ever been a member of the spacetime cluster $\cC_{v,\le t}$, 
\begin{align}\label{eq:C-t-v}
    \cC_{t}(v) = \{w\in V: (w,s)\in \cC_{v,\le t} \text{ for some $s\le t$}\}\,.
\end{align}
Notice that the vertex projection $\cC_{t}(v)$ of a spacetime cluster $\cC_{v,\le t}$ is $G$-connected. 

In this manner, minus regions can similarly be indexed by constituent vertices, writing $\mathcal R_{t}(v)$ to mean the minus region containing vertex $v$ at time $t$, 
\begin{align}\label{eq:R-t-v}
    \cR_{t}(v)= \cC_{v,\le t}\cap (V\times \{t\})\,.
\end{align}
Notice that $\cR_t(v) \subset \cC_t(v)$ but $\cR_t(v)$ need not be $G$-connected; this is why we refer to it as a minus region as opposed to cluster.

 \begin{definition}\label{def:boundary}
     For a minus spacetime cluster $\mathcal C_{\le t}$ define $\partial_o \mathcal C_{\le t}$ as its outer spacetime boundary, consisting of 
 \begin{itemize}
     \item its outer spatial boundary:  all spacetime points $(v,s)\in V\times [0,t]$ such that $(v,s) \notin \mathcal C_{\le t}$ but there exists a vertex $w$ with $\{v,w\}\in E$ such that  $(w,s) \in \mathcal C_{\le t}$;
      \item its lower outer temporal boundary: if $(w,t_0)$ is the minimum in a connected component of the set $\{s\in [0,\infty):(w,s)\in \cC_{\le t} \}$
      and $t_0>0$, then include $(w,t_0^-): = \lim_{s\uparrow t_0}(w,s)$;
     \item its upper outer temporal boundary: if $(w,t_0)$ is the supremum in a connected component of the set $\{s\in [0,\infty):(w,s)\in \cC_{\le t} \}$ and is not in $\cC_{\le t}$, then include $(w,t_0)$. 
\end{itemize}
 \end{definition}
 
By construction (as a maximal connected minus component of $G\times [0,t]$), at all spacetime points in $(w,s)\in \partial_o \mathcal C_{\le t}$, the two-spin dynamics has $X_s^+(w) =+1$.  
More generally, we have the following important observations.

\begin{observation}\label{obs:properties-of-regions}
    Firstly, for two distinct minus spacetime clusters $\mathcal C_{\le t} \ne \mathcal C_{\le t}'$, the corresponding minus regions $\mathcal R_t, \mathcal R_t'$ are disjoint and in fact have distance at least $2$.  
    
    In $X_t^+$, any connected set of minuses containing a vertex $v$ must belong to the same region $\cR_t(v)$, as the outer boundary of $\mathcal R_{t}(v)$ consists entirely of $+1$ sites. Furthermore observe that  
\begin{align}\label{eq:minus-regions-equal-minuses}
    \bigsqcup \mathcal R_{t} = \{v: X_t(v) = -1\}\,,
\end{align} 
i.e., the set of minus sites at time $t$ are given by the disjoint union of all minus regions at time $t$. 
\end{observation} 

The spacetime construction of Definition~\ref{def:minus-spacetime-clusters} and the definition of regions from that, naturally induces a Markov dynamics on the set of minus regions $(\mathcal R_t(v))_{v\in V}$. The following is immediate from the construction of Definition~\ref{def:minus-spacetime-clusters} and can help with intuition when computing drifts of the size of a minus region in the subsequent. 

\begin{observation}\label{obs:minus-region-evolution}
       Given minus regions $(\mathcal R_{t^-}(v))_{v\in V}$ from a process $(X_t)_{t\ge 0}$ with an update at $w$ at time $t$, the set of minus regions are updated as follows: 
       \begin{enumerate}
       \item If $X_{t^-}(w) = -1$ and $X_t(w) = +1$: then $\mathcal R_{t}(w) = \emptyset$; for every $v\in \mathcal R_{t^-}(w)$, $w$ is removed from its region so that $\mathcal R_{t}(v) = \mathcal R_{t^-}(v)\setminus \{w\}$; and all regions not containing $w$ are unchanged. 
       \item If $X_{t^-}(w) = +1$ and $X_{t}(w) =-1$:  then set $$ \mathcal R_t(w)  = \{w\} \cup \bigcup_{z\sim w} \mathcal R_{t^-}(z)\,;$$
        for every $z\in \mathcal R_t(w)$, let $\mathcal R_t(z) = \mathcal R_t(w)$: i.e., all these minus regions are merged; and all other minus regions are unchanged. 
       \end{enumerate}
       Notice that case (1) happens if $\mathcal R_{t^-}(w)\ne \emptyset$ with probability $p_+((X_{t^-}(z))_{z\sim w})$, which is measurable with respect to $(\mathcal R_{t^-}(v))_{v\in V}$ by~\eqref{eq:minus-regions-equal-minuses}, and similarly for case (2). Therefore, this dynamics is indeed Markovian. 
\end{observation}

\subsection{Internal measurability of minus spacetime clusters}
The importance of working with minus regions as opposed to minus connected components is that they group together minus connected components that came from the same source, and therefore share some negative information, even if they split at some earlier time. 

At various points in the paper, we will condition on the realizations of certain spacetime regions of vertices at different times. Since all specific realizations have probability $0$ due to the time-continuity, let us introduce some notation to describe the family of conditionings that are valid. 

\begin{definition}\label{def:admissible-clock-process}
    A set of clock rings $\cT = (t_{v,i})_{v\in V,i\ge 1}$ (which is a point process in $V\times [0,\infty)$) is \emph{admissible} if $\cT$ consists of distinct points, and is locally finite. 
\end{definition}

Note that the clock ring sequence $\mathcal T$ given by rate-1 Poisson clocks at each vertex per Definition~\ref{def:grand-coupling} is almost surely admissible. 

\begin{definition}\label{def:admissible-cluster}
    A spacetime cluster $C\subset V\times [0,t]$ is $(u,t)$-admissible if there exists an admissible $\cT$ such that $\mathbb P( \cC_{u,\le t} =C \mid \cT)>0$.  
\end{definition}

We make the following observation that the information learned from conditioning on the event $\{\cC_{u,\le t} = C\}$ is only internal, or on the spacetime boundary, of $C$. 
\begin{observation}\label{obs:spacetime-cluster-internally-measurable}
    The event $\{\cC_{u,\le t} = C\}$ is measurable with respect to the following:
        \begin{itemize}
        \item The set $\mathcal T(C)$ of all clock ring times $t_{w,j}$ where $(w,t_{w,j}^-) \in C$ while $(w,t_{w,j}) \notin C$ or vice versa. 
        \item The set of all uniform random variables associated to those times, i.e., $U_{w,j}$ for $w,j$ such that $t_{w,j}\in \mathcal T(C)$; 
        \item The indicator random variables $(I_{(w,r)})_{(w,r)\in \partial_o C}$ that the outer boundary spacetime points $(w,r)$ in $C \cap (V \times [0,t])$ all had $X_r(w) = +1$. 
    \end{itemize}
\end{observation}

Informally speaking, since the spacetime boundary of the spacetime cluster is forced to be all-plus, and the dynamics' update rules are local (an update at distance two from a spacetime cluster $\cC_{v,\le t}$ is not influenced by the interior of the spacetime cluster) conditioning on the above information only gives increasing information on its exterior. In this sense, a minus spacetime region confines the negative information of its interior. This notion is formalized in a few different ways in our proof: see the upcoming Lemmas~\ref{lem:conditional-flip-rate} and~\ref{lem:spacetime-cluster-increasing-on-exterior}.  

Let us end the subsection with a comment on analogous constructions for two-spin dynamics from other initializations, and the forthcoming rigid dynamics. 

\begin{remark}\label{rem:spacetime-cluster-construction-general-init}
The above construction of Definition~\ref{def:minus-spacetime-clusters} (and subsequent definitions of minus regions) can be done with any process $(\bar X_s)_{s\in [0,t]}= (\bar X_s(\mathcal T,\mathcal U))_{s\in [0,T]}$ that takes values in $\{+1,-1\}$ at every spacetime point and only flips states of vertices at clock ring times with $\cF_{t^-}$-measurable probability. In particular, the construction is analogous for two-spin dynamics from any other initialization $(X_t^{x_0})_{t\ge 0}$ and for the subsequently-defined rigid dynamics mentioned in the proof sketch because the property of being a trifurcation point is $\cF_{t^-}$-measurable. Moreover, Observations~\ref{obs:properties-of-regions},~\ref{obs:minus-region-evolution}, and~\ref{obs:spacetime-cluster-internally-measurable} applies mutatis mutandis in all these cases; indeed for the rigid dynamics that will be the case because the status of a vertex $v$ as a trifurcation point will be $\mathcal C_{v,\le t}$-measurable. 
\end{remark}

\subsection{Trifurcation points and the rigid dynamics}
Our aim is to control the growth of minus spacetime clusters in time (and therefore provide exponential tails on their slices which form the minus regions). However, these regions will not necessarily have negative drift due to possible ``holes" in their slices which are incident to many vertices of the same minus region and therefore flip back to minus quickly. 

Our solution to address this difficulty is to consider a variant of the two-spin dynamics, which we call the \emph{rigid dynamics} $\widetilde X_t$ which will not allow such holes in minus regions to develop in the first place. This rigid dynamics will on the one hand be stochastically below $X_t^+$, and on the other hand, exhibit the desired negative drift on the size of its minus regions (at least while they are $O(\log n)$ sized). 
Towards that definition, in what follows, let \begin{align}\label{eq:R}R = \frac{1}{4}\log_{d} n\,,\end{align} so that with probability $1-o(1)$, $G$ is such that for every $v$, the ball $B_R(v)$ has at most $1$ cycle. 

\begin{definition}\label{def:1-treelike}
    $G$ is $1$-locally-treelike if for every $v$, $B_R(v)$ has at most one cycle. 
\end{definition}

The following fact is classical and easy to check by the configuration model contiguity to the random $d$-regular graph: see e.g.,~\cite[Lemma 2.1]{LuSly-Cutoff-RRG}. 

\begin{fact}\label{fact:G_d(n)-1-treelike}
    If $d\ge 3$ and $\alpha<1/4$ then with probability $1-o(1)$, for every $v$, $B_{\alpha \log_{d-1} n}(v)$ has at most one cycle. In particular, with probability $1-o(1)$, $G\sim \mathcal G_d(n)$ is $1$-locally-treelike.  
\end{fact}

Recall the minus spacetime clusters and minus region defined per Definitions~\ref{def:minus-spacetime-clusters}--\ref{def:minus-region} with respect to the two-spin dynamics $(X_t^+)_{t\ge 0}$ initialized from all-plus. 

\begin{definition}[Trifurcation point] For a set $A$, a vertex $v\in A$ is called a \emph{trifurcation point} of $A$ if at least three of the connected components of $B_{R}(v)\setminus \{v\}$ have non-empty intersection with $A$.  We denote this by $v\in \mathsf{Tri}(A)$.  
\end{definition}

We now define the variant of the chain $(X_s^+)_{s\ge 0}$ which ignores updates that would flip a trifurcation point to a plus. This variant will be denoted $(\widetilde X_s)_{s\ge 0}$ and it is naturally incorporated into the coupling of Definition~\ref{def:grand-coupling} using the same $(\mathcal T,\mathcal U)$ sequence, so by Remark~\ref{rem:spacetime-cluster-construction-general-init} it has its own minus spacetime clusters and minus regions, which we will denote $\widetilde \cC_{v,\le t}$ and $\widetilde R_{t}(v)$ respectively.  

\begin{definition}\label{def:rigid-dynamics}
    The \emph{rigid dynamics} $(\widetilde X_t)_{t\ge 0}$ is the variant of $(X_t^+)_{t\ge 0}$ initialized from $\widetilde X_0 \equiv +1$, which, if it goes to update a vertex $v$ at time $t$, rejects the update if $v\in \mathsf{Tri}(\widetilde \cR_{t^-}(v))$. 
\end{definition}

\begin{remark}
    Notice that $(\widetilde X_t)_{t\ge 0}$ is still a jump stochastic process with a.s.\ distinct jump-times and with transition rates at time $t$ that are measurable with respect to $\cF_{t^-}$. However, note that it is \emph{not} Markovian, as whether a minus vertex $v$ is a trifurcation point or not depends on the entire history of its minus spacetime cluster. 
\end{remark}

Observe that in order for $v\in \mathsf{Tri}(\wR_{t^-}(v))$, it must be that $\widetilde \cR_{t^-}(v) \ne \emptyset$ so $\widetilde X_t(v) = -1$. Thus the rigid dynamics is only rejecting updates that would flip minus sites to plus sites, and by monotonicity of the two-spin dynamics' update rules and the trifurcation property, one has \begin{align}\label{eq:rigid-below-normal-dynamics}\widetilde X_t \le X_t^+\,, \qquad \text{for all $t\ge 0$}\,.\end{align}
Notice also, that the status of $v$ as a trifurcation point is $\widetilde{\mathcal C}_{v,\le t}$-measurable.

The following shows that in any set smaller than $R$, the number of trifurcation points is at most half the vertices in the set, so that in any minus region of size smaller than $R$, the rigid dynamics can always flip at least half of the vertices in $\widetilde\cR_{t}(v)$ to plus. This property of trifurcation points is what prevents us from only blocking flips at ``bifurcation points". 

\begin{lemma}\label{lem:size:trifurcation}
    Fix any $d\ge 3$. If $G$ is $1$-locally-treelike, then the number of trifurcation points in any set $A$ of diameter less than $R$ is at most $|A|/2$.
\end{lemma}

\begin{proof}
    Fix an arbitrary vertex $\rho \in A$ to be a `root'. Note that since $\diam(A)\le R$, we have $A\subseteq B_R(\rho)$, which contains at most $1$ cycle. Moreover, removing an edge from $B_R(\rho)$ only increases the number of connected components of $B_{R}(v)\setminus \{v\}$ that have a non-empty intersection with $A$. Thus, the number of trifurcation points of $A$ in $G$ is at most that number in $G'$, the graph where one edge is removed from $B_R(\rho)$ to make it a tree.
    
    Therefore, it suffices to show that if $A$ is a subset of $\mathbb T_d$, the $d$-regular tree for $d\ge 3$, the number of trifurcation points in $A$ is at most $|A|/2$. Associate to $A$ a tree $T_A$ rooted at $\rho$, where we define a vertex $u\in A$ to be a `child' of $v\in A$ if $v$ is the first vertex in $A$ encountered along the path in $\mathbb T_d$ from $u$ to $\rho$. Notice that every trifurcation point in $A$ has at least two distinct children in $A$. Moreover, a vertex $u\in A$ cannot be a child of two distinct vertices in $A$ by our definition. Thus, there are two injective maps $\Phi_1$ and $\Phi_2$ that take as input the trifurcation points $u\in \mathsf{Tri}(A)$ and output two distinct children of $u$. Their images are disjoint, and the union of their images is a subset of $A$, so $|A| \ge 2|\mathsf{Tri}(A)|$. 
\end{proof}

We also make a few structural observations about the $1$-neighborhoods of certain vertices. These observations will play a crucial role in our drift analysis of minus spacetime clusters in the subsequent sections.

\begin{observation}\label{obs:minus-four-neighbors-tripoint}
    Suppose $G$ is $1$-locally-treelike. If $\widetilde X_t(v)=-1$ and  $v\notin \mathsf{Tri}(\widetilde \cR_{t}(v))$, then $v$ has at most three neighbors which are minus in $\widetilde X_t$. 
\end{observation}

\begin{proof}
    Suppose by contradiction that $v$ has four (or more) neighbors which are minus in $\widetilde X_t^+$. Because $B_R(v)$ has at most $1$ cycle, the graph $B_R(v)\setminus \{v\}$ has at least $d-1$ connected components, and the four neighbors of $v$ which are minus in $\widetilde X_t^+$ are in at least $3$ distinct connected components $C_1, C_2,C_3$ of the set $\{w\in B_R(v)\setminus \{v\}: \widetilde X_t^+(w) = -1\}$. Recalling from Observation~\ref{obs:properties-of-regions} that connected sets of minuses are subsets of a single minus region, $C_1\cup C_3\cup C_3 \cup \{v\}$ is a connected set of minuses, and therefore by Observation~\ref{obs:properties-of-regions}, a subset of $\widetilde{\cR}_{t}(v)$ implying that $v\in \mathsf{Tri}(\widetilde\cR_{t}(v))$. 
\end{proof}

In addition to the minus spacetime clusters of the rigid dynamics, $\wC_{v,\le t}$ and their induced minus regions $\widetilde \cR_{t}(v)$, recall the projection of the spacetime cluster, i.e., all vertices that at some time belong to $\wC_{v,\le t}$, which we denote $\wC_t(v) \subset V$.
\begin{observation}\label{obs:minus-four-neighbors-outside-part}    Suppose $|\wC_{t}(u)|\le R$ and $G$ is $1$-locally-treelike. If $\widetilde X_t(v)=+1$, but $v$ is adjacent to a minus region $\widetilde {\cR}_{t}(u)$, then at most three neighbors of $v$ are in $\widetilde \cR_{t}(u)$. 
\end{observation}

\begin{proof}
    Suppose by contradiction that $v$ is incident to four (or more) vertices in a region $\widetilde{\cR}_{t}(u)$. Then, we first claim that $v\in \wC_t(u)$: suppose that $v\notin \wC_t(u)$. Since $\wC_t(u)$ must be connected and $|\wC_t(u)|\leq R$, and because $v$ has $4$ distinct neighbors in $\wC_t(u)$, it follows that $v$ is contained in at least $4$ cycles of length at most $R+1$. This contradicts the assumption that $G$ is $1$-locally-treelike per Definition~\ref{def:1-treelike}. Thus, $v\in \wC_t(u)$ holds.
    
    Let $t_1>0$ be the most recent time $v$ flipped from minus to plus under the rigid dynamics. This is well-defined since $v\in \wC_t(u)$. At time $t_1^-$, $v$ must not have been a trifurcation point of $\widetilde \cR_{t_1^-}(v)$, as the rigid dynamics does not permit flipping trifurcation points. Thus, at most $2$ of the connected components of $B_R(v) \setminus \{v\}$ intersected $\widetilde {\cR}_{t_1^-}(v)$, and at most $3$ of the neighbors of $v$ were minus at time $t_1$ by Observation~\ref{obs:minus-four-neighbors-tripoint}. 

    In order for a fourth of $v$'s neighbors, call it $w$, to be in $\wR_{t}(u)$, it must therefore be in a third connected component of $B_R(v)\setminus \{v\}$ distinct from the $2$ that intersected  $\widetilde {\cR}_{t_1^-}(v)$. This requires a spacetime connected subset $U$ of $\widetilde \cC_{w,\le t} \cap (G \times [t_1,t])$ connecting $(z,t_1)$ and $(w,t)$ for some $z$ in a distinct connected component of $B_R(v)\setminus \{v\}$ from $w$. Moreover, since $v$ is plus at all times in $[t_1,t]$, $U$ may not intersect $\{v\}\times [t_1,t]$. 
 This implies that $\widetilde C_{t}(u)$ must contain a path connecting $z$ to $w$ in $G\setminus \{v\}$ which by the $1$-locally-treelike property, must have length at least $R$. 
\end{proof}

\section{Exponential tails on minus spacetime clusters and regions}\label{sec:exp-tails-minus-spacetime-regions}
In this section, we establish the core estimate that controls minus spacetime regions along the evolution of the rigid dynamics chain initialized from all-plus. 

Let $(\psi_k)_{k\geq 0}$ be the probability mass function defined as 
\begin{align}\label{eq:psi_k}
\psi_k:= k^{-2}d^{-1000-100 k }\,,\qquad k\geq 1\,,\qquad \psi_0:=1-\sum_{k\geq 1} \psi_k\,.
\end{align}
Let $\Psi$ be the law of such a random variable. The precise value of the constants $1000, 100$ will not be important, and they may be replaced by sufficiently large universal constants $\Gamma_0,\Gamma_1$ such that $100\le \Gamma_1\le \Gamma_0/10$, at the cost of enlarging the constant $C_0$ in the bound $\beta > \frac{C_0\log d}{d}$.

The goal of this section is to establish the following exponential tail bound on the probabilities of minus regions under the dynamics initialized from all-plus, so long as they are of size at most $R$ which we recall from~\eqref{eq:R} to be $\frac{1}{4} \log_{d} n$. 

\begin{proposition}\label{prop:exp-tails-minus-regions}
    Let $d\geq 7, \beta>\frac{C_0\log d}{d}$, and suppose $G$ is $1$-locally-treelike. Consider $(\widetilde X_t)_{t\ge 0}$, the rigid dynamics initialized from all-plus. For every $u\in V$, and integers $1\le k\le \ell \le R$, we have
    \begin{align*}
        \mathbb P\Big( |\wR_t(u)| \ge k,\,|\wC_t(u)|=\ell \Big) \le  \psi_k\cdot \psi_{\ell} \,.
    \end{align*}
\end{proposition}

Notice that since $\wR_t(u) \subset\wC_{t}(u)$, this only counts minus regions of sizes between $k$ and $R$. 

Before proceeding to the proof of Proposition~\ref{prop:exp-tails-minus-regions}, we state the following lemma which summarizes the crucial properties of the tail probabilities, and convolution properties for $d$ independent copies from the random variable $Z\sim \Psi^{\otimes 2}$. Its proof consists of explicit calculations involving exponential tilting of $\Psi^{\otimes 2}$. Thus, we defer the proof to Section~\ref{sec:deferred-proof-of-psi-bounds}. We say $Z = (Z^{(1)}, Z^{(2)}) \ge (k,\ell)$ for a vector $Z\in \mathbb N^2$ if $Z^{(1)}\geq k$ and $Z^{(2)}\geq \ell$.

    \begin{lemma}\label{lem:psi-tail-bound}
       Fix $d\geq 7$ and let $(Z_i)_{i\geq 1}\stackrel{i.i.d.}{\sim}\Psi^{\otimes 2}$. The following holds for all $k,\ell \ge 0$ and $q\leq d$.
       \begin{enumerate}
           \item [(a)]$\P(\sum_{i=1}^{q}Z_i \ge (k,\ell)) \leq 2 q^2 \psi_k \psi_\ell$.
           \item[(b)] $\P(J\ge 2, \sum_{i=1}^q Z_i \ge (k,\ell))\leq d^{-20}\psi_{k+1} \psi_{\ell+1}$, where $J=\sum_{i=1}^{q}\one\{Z_i\geq (1,1)\}$.
       \end{enumerate}
    \end{lemma}

Further setting up towards proving Proposition~\ref{prop:exp-tails-minus-regions}, let us generalize the notion of $(u,t)$-admissible spacetime cluster.

\begin{definition}
    Given $m$ vertices $\vec{u}=(u_1,\ldots, u_m)\in V^m$ and $m$ times $\vec{t}=(t_1,\ldots, t_m)\in [0,\infty)^m$, we say $\vec{C}=(C_1,\ldots, C_m)$ with $C_i \subset V\times [0,t_i]$ is $(\vec{u},\vec{t})$-admissible if there exists an admissible clock ring sequence $\cT$ (cf. Definition~\ref{def:admissible-clock-process}) such that $\P(\wE_{\vec{u},\vec{t}, \vec{C}}\mid \cT)>0$, where
    \[
   \wE_{\vec{u}, \vec{t},\vec{C}}:=\Big\{\wC_{u_i,\leq t_i}=C_i~~\forall i \in [m]\Big\}.
    \]  
    Note that the event $\wE_{\vec{u}, \vec{t},\vec{C}}$ is measurable w.r.t. the sigma algebra $\widetilde{\cF}_{\vec{u},\vec{t}}$, where
\[
  \wFF=\sigma(\wC_{u_1,\leq t_1}, \ldots, \wC_{u_m,\leq t_m}).
    \]
 
\end{definition}

Proposition~\ref{prop:exp-tails-minus-regions} is a consequence of the following more general conditional version of it, by taking $m=0$. 

\begin{proposition}\label{prop:exp-tail-minus-regions-conditional}
    Let $d\ge 7$, $\beta > \frac{C_0 \log d}{d}$ and suppose $G$ is $1$-locally-treelike. For every $(u,t)\in V\times [0,\infty)$, and integers $1\le k\le \ell \le R$, the following holds. For any integer $m\geq 0$, $\vec{u}=(u_1,\ldots,u_m)\in V^m$, $\vec{t}=(t_1,\ldots, t_m)\in [0,\infty)^m$ such that $t_1\geq t_2\geq \ldots \geq t_m >t$, and $(\vec{u},\vec{t})$-admissible $\vec{C}=(C_1,\ldots, C_m)$, we have 
    \begin{align*}
        \mathbb P\Big( |\wR_t(u)| \geq k, |\wC_t(u)| = \ell \mid \wEE\Big)\cdot \one\Big\{(u,t)\notin \bigcup_{i=1}^{m}C_i\Big\}\le \psi_k \cdot \psi_\ell\,.
    \end{align*}
\end{proposition}
Note that the event $\wEE$ is the same as $\wX_t(v)=-1$ for all $(v,t)\in \bigcup_{i=1}^{m}C_i$ and $\wX_t(v)=+1$ for all $(v,t)\in \bigcup_{i=1}^{m} \partial_o C_i$. In particular, conditioning on $\wEE$ forces specific update times in the temporal boundary of $\bigcup_{i=1}^{m} C_i$, thus $\P(\wEE)=0$. However, since we assumed $\vec{C}$ is $(\vec{u},\vec{t})$-admissible and conditioning on some clock rings gives another Poisson process, the conditional distribution given $\wEE$ is well-defined.

\begin{remark}\label{rmk:proof-conditional}
    Before proceeding to the proof, let us comment on why the stronger conditional formulation in Proposition~\ref{prop:exp-tail-minus-regions-conditional}, as compared to Proposition~\ref{prop:exp-tails-minus-regions}, is necessary. To establish the negative drift for the probability, the main difficulty lies in controlling the positive contribution to the time-derivative that arises from merging of minus regions. One natural approach would be to exploit the monotonicity in the conditioning on $\wC_t(u)$ (which is a positive information to its complement), in order to inductively apply the tail estimate to the merging clusters. However, such a monotonicity-based bound would lead to terms of the form $\mathbb P(|\wR_{t}(w)| \ge k, |\wC_t(w)|\ge \ell)$ rather than equality on $|\wC_t(w)|$. This includes the event of clusters larger than size $R$ for which we lose the locally tree-like geometry.
    Therefore, we strengthen the inductive hypothesis into the conditional form, performing the entire inductive procedure conditional on disjoint spacetime clusters $(\wC_{u_i,\leq t_i})_{1\leq i\leq m}$. Technically speaking, this conditional form allows us to argue Lemma~\ref{lem:general-tail-bound-domination-by-iid-Z}.
\end{remark}

\subsection{Domination of conditional flip rates}
A key step in proving Proposition~\ref{prop:exp-tail-minus-regions-conditional} is to show that conditioning on spacetime clusters $\wEE$ only makes the rigid dynamics flip faster to $+1$ in the complement of $\bigcup_{i=1}^{m}C_i$. To this end, define the conditional flip rate 
\begin{align}\label{eq:rigid-conditional-flip-rate}
    \widetilde{c}_{v,t}= \widetilde c_{v,t}((\widetilde X_s)_{s<t}) := \lim_{\eps\downarrow 0}\frac{1}{\epsilon } \mathbb P( \wX_{t+\epsilon}(v) \ne \widetilde X_t(v) \mid \wF_{t^-} , \wEE\big)\,,
\end{align}
where $\wF_{t^-}$ is the sigma algebra generated by the past trajectory $(\wX_s)_{s<t}$.

\begin{lemma}\label{lem:conditional-flip-rate}
    For $(\vec{u},\vec{t})$-admissible $\vec{C}$, and $(v,t)$ such that $t\notin \overline{\{s:(v,s)\in \bigcup_{i=1}^{m}C_i\}}$, we have the following. For every realization of $(\wX_{s})_{s<t}$ such that it is compatible with the event $\wEE$ and $\wX_{t^-}(v)=+1$, we have 
    \begin{align*}
        \widetilde{c}_{v,t}  \le 1- p_+((\wX_{t^-}(w))_{w\sim v})\,.
    \end{align*}
    If on the other hand, $\wX_{t^-}(v) = -1$ and $v\notin \mathsf{Tri}(\wR_{t^-}(v))$, then 
    \begin{align*}
       \widetilde{c}_{v,t} \ge p_+((\wX_{t^-}(w))_{w\sim v}) \ge \frac{1}{1+e^{ -2\beta d/7}}\,.
    \end{align*}
\end{lemma}
\begin{proof}
 We start with the proof when $\widetilde{X}_{t^-}(v) = +1$.
    Let us additionally condition on all clock ring times $\mathcal T_{\vec{C}}$ which are the clock rings $t_{w,j}$ such that $(w,t_{w,j})\in \bigcup_{i=1}^{m}C_i$ or its temporal closure. Note that since we have $(v,t)\notin \bigcup_{i=1}^{m}C_i$, none of these are at times  in $[t,t+\epsilon]$ for $\epsilon$ sufficiently small. 
    Then, using Bayes' formula,~\eqref{eq:rigid-conditional-flip-rate} is 
    \begin{align*}
                \lim_{\epsilon \downarrow 0} &  \frac{\mathbb P(\wEE \mid \wF_{t-}, \cT_{\vec{C}}, \wX_{t+\epsilon}(v) \ne \wX_{t^-}(v))}{\mathbb P(\wEE \mid \wF_{t-}, \cT_{\vec{C}})}\cdot \frac{\mathbb P(\wX_{t+\epsilon}(v) \ne \wX_{t^-}(v) \mid \wF_{t^-}, \cT_{\vec{C}}) }{\epsilon} \\ 
        & =  \frac{\mathbb P(\wEE \mid \wF_{t^-}, \cT_{\vec{C}}, \wX_{t}(v) \ne \wX_{t^-}(v))}{\mathbb P(\wEE \mid \wF_{t^-}, \cT_{\vec{C}})} \cdot  \widehat{c}_{v,t}((\wX_s)_{s<t})\,,
    \end{align*}
    where $\widehat{c}_{v,t}$ here denotes the rigid dynamics' flip rate given $(\wX_{s})_{s<t}$ but not conditioning on $\wE_{\vec{u},\vec{t},\vec{C}}$. This latter rate is only given $(\wX_s)_{s<t}$ because the flip rate is independent of $\cT_{\vec{C}}$, and it is exactly $1-p_+((\wX_{t^-}(w))_{w\sim v})$. 
    Note that the conditioning on $\cT_{\vec{C}}$ allowed this use of Bayes' formula since without conditioning on $\cT_{\vec{C}}$, the probabilities on the right would be zero. 
    
    Conditioning further on any admissible $\cT$ compatible with $\cT_{\vec{C}}$ and with a clock ring at $(v,t)$, it suffices to prove that for $(v,t)$ such that $t\notin \overline{\{s:(v,s)\in \bigcup_{i=1}^{m}C_i\}}$,
    \begin{align}\label{eq:flip-rate-after-Bayes}
 \mathbb P(\wEE \mid \wF_{t^-}, \cT, \wX_{t}(v) =-1)\leq \mathbb P(\wEE \mid \wF_{t^-}, \cT, \wX_t(v)=+1) \,.
 \end{align}
    To prove \eqref{eq:flip-rate-after-Bayes}, we consider coupling two copies of rigid dynamics with the same history before time $t$, one of which flipped $v$ at time $t$ from $+1$ to $-1$ and the other one which did not. These are coupled via the grand coupling of Definition~\ref{def:grand-coupling}, and denoted by $\widetilde{Z}_t^v$ and $\widetilde{Z}_t$ respectively.  

    \medskip
    \noindent \emph{Base case}: At time $t$, both $\widetilde{Z}_t^v$ and $\widetilde Z_t$ are compatible with $\wEE$ because their histories up to time $t^-$ were compatible, and $(v,t)$ is such that $t\notin \overline{\{s:(v,s)\in \bigcup_{i=1}^{m}C_i\}}$. 

    \medskip
    Enumerate the clock rings of $\cT$ (possible using admissibility of $\cT$) that follow time $t$ as $T_0 = t < T_1 < T_2 <...$. We show that for every such $T_i$, $\widetilde{Z}_{T_i}^v\in \wEE$ implies the same for $\widetilde Z_{T_i}$.   

    \medskip
    \noindent \emph{Inductive step}: 
    Suppose that for $t< T_k$ we have $$\widetilde{Z}(V\times [0,t] \setminus  \vec{C})\ge \widetilde{Z}^v(V\times [0,t]\setminus \vec{C}) \qquad\text{and} \qquad \widetilde{Z}^v(V\times [0,t]), \widetilde Z(V \times [0,t]) \in \wEE$$
    where the latter $\in$ means it is compatible with it (there is a realization of the future of the processes that would be in $\wEE$). We claim that we retain the domination on the complement of $\vec{C}$, and equality inside $\vec{C}$ if $\widetilde{Z}^v$ remains consistent with $\wEE$. This would imply that the inductive assumption holds for times $[T_k,T_{k+1})$. 
    
    We divide the discussion according to which vertex $w$ has its clock ring at time $T_k$: 
    \begin{itemize}
        \item $(w,T_k)$ at graph distance at least two from $\vec{C} \cap \{V\times \{T_k\}\}$: its neighbors are more positive under $\widetilde Z$ because they are exterior to $\vec{C}$. Moreover, the event of $w$ being a trifurcation point is measurable and monotone with respect to its past $(\widetilde{Z}^v_s)_{s<T_k}, (\widetilde{Z}_s)_{s<T_k}$ on the complement of $\vec{C}$, on which we have assumed the ordering between $\widetilde{Z}$ and $\widetilde{Z}^v$ holds. 
        Therefore, the rigid dynamics updates satisfy $\widetilde{Z}_{T_k}^v(w) = +1 \implies \widetilde{Z}_{T_k}(w) = +1$ retaining the domination (in this case consistency with $\wEE$ is evidently unaffected). 
        \item $(w,T_k)\in  \vec{C}$: In both chains, this vertex's neighbors at time $T_k$ are either in $\vec{C}$ in which case they are minus, or they are on its boundary in which case they are $+1$ in both copies by the inductive assumption that $(\widetilde{Z}^v_{s})_{s<T_k} \in \wEE$ and also $(\widetilde{Z}_s)_{s<T_k} \in \wEE$. Moreover, the trifurcation point status of the vertex $w$ is measurable with respect to the past of its spacetime cluster, on which we have assumed that the two chains agree since both are compatible with $\wEE$. 
        Thus, the update is the same in the two chains. 
        \item $(w,T_k) \in \partial_o  \vec{C}$: The neighbors of $w$ at time $T_k$ are either in $\vec{C}$ in which case they are minus in both copies, or they are exterior to $\vec{C}$ in which case they are more plus under $\widetilde{Z}$ than under $\widetilde{Z}^v$. Thus the update is more plus under $\widetilde{Z}$ than under $\widetilde{Z}^v$ and the domination on the complement is retained. 
        If the update maintained $+1$ under the lower $\widetilde{Z}^v$ chain, then it also maintained $+1$ under the $\widetilde Z$ chain, and both remain compatible with $\wEE$. If the update under $\widetilde{Z}^v$ was to $-1$, then it was no longer compatible with $\wEE$ because $(w,T_k) \in \partial_o  \vec{C}$. 
    \end{itemize}
    Together this implies we preserve the inductive hypothesis, and ultimately conclude that $\widetilde{Z}^v \in \wEE$ implies the same of $\widetilde Z$, so that~\eqref{eq:flip-rate-after-Bayes} indeed holds. 

    The proof in the case of a $-1$ to $+1$ flip goes by the same reasoning to compare $\widetilde{c}_{v,t}$ to its rate only conditioning on $\wF_{t^-}$, and that latter rate of flipping to $+1$ is at least $p_+((\wX_{t^-}(w))_{w\sim v})$, which by assumption (B) of Definition~\ref{def:general-ising-process} is at least $1/(1+e^{ - 2\beta d/7})$ because $v$ has at most $3$ minus neighbors by Lemma~\ref{obs:minus-four-neighbors-tripoint} and $d\ge 7$ (and it is not a trifurcation point by assumption). 
\end{proof}
\subsection{Proof of Proposition~\ref{prop:exp-tail-minus-regions-conditional}}
Let us set up some notation towards proving Proposition~\ref{prop:exp-tail-minus-regions-conditional}. Throughout, we let $\vec{u}=(u_1,\ldots, u_m)\in V^m$ and $\vec{t}=(t_1,\ldots, t_m)$ such that $t_1\geq \ldots \geq t_m$. Define the following quantity: 
\[
\begin{split}
 &\rho_t(u;k,\ell):= \max_{m}\sup_{(u_i,t_i)_{i=1}^m: t_m >t}~~\sup_{\vec{C}} ~\rho_t(u;k,\ell \mid \wEE),~~\textnormal{where}\\ &\rho_t(u;k,\ell \mid \wEE):=\mathbb P\big( |\wR_t(u)| \ge k, |\wC_t(u)| = \ell \mid \wEE\big)\cdot \one\Big\{(u,t)\notin \bigcup_{i=1}^{m}C_i\Big\}.
\end{split}
\]
where the supremum over $\vec{C}$ is over $(\vec{u},\vec{t})$-admissible $\vec{C}=(C_1,\ldots C_m)$. Here, for $m=0$, we define $t_m\equiv \infty$. We define the time  $T\in [0,\infty) \cup \{\infty\}$ up to which we have control on minus regions as 
\begin{align}\label{eq:def-T}
    T := \inf\big\{t\ge 0: \max_u \rho_t(u;k,\ell) \ge \psi_k \psi_\ell \text{ for some }1\le k\le \ell\le R\big\}\,.
\end{align}

Note that at $t=0$, $\wX_0(u)=+1$ deterministically, and thus $\rho_0(u;k,\ell)=0$. Our aim is to show $T=\infty$. We will prove in Lemma~\ref{lem:NTS-rho-negative-drift} that the time derivative of $\rho_t(u;k,\ell\mid \wEE)$ has a uniform upper bound, which implies that $T>0$. Moreover, we establish negative drift of $\rho_t(u;k,\ell\mid \wEE)$ assuming $t<T$ if $\max_u\rho_t(u;k,\ell)$ gets close to the threshold $\psi_k\psi_\ell$. In what follows, we denote the event
\begin{equation}\label{eq:def-NI}
\mathsf{NI}_{u,\vec{u},\vec{t}}(t):=\Big\{(u,t)\notin \bigcup_{i=1}^{m}\wC_{u_i,\leq t_i}\Big\}.
\end{equation}
The following lemma will be useful for proving the negative drift as it establishes that while $t< T$, there is an (approximate) domination of the joint sizes of minus spacetime clusters disjoint to $\vec{C}$ by independent draws from $\Psi$. 

\begin{lemma}
    \label{lem:general-tail-bound-domination-by-iid-Z}
 Let $t<T$, $\vec{u},\vec{t}$ such that $t_m\geq t$, and $1\le k\le \ell \le R$. For any $q$ vertices  $w_1,...,w_q\in V$ where $q\leq d$, we have the following. Denote the number of distinct non-empty minus regions containing $(w_i)_{i=1}^q$ by 
 \begin{align}\label{eq:def-distinct-regions}
     \wJ=\sum_{i=1}^{q}\one\Big\{\Big|\wR_{t^-}(w_i) \setminus \bigcup_{a=1}^{i-1} \wR_{t^-}(w_a)\Big|\geq 1\Big\}
 \end{align}
 and for $(Z_i)_{1\leq i \leq q}$ i.i.d.\ draws  from $\Psi^{\otimes 2}$, let $J=\sum_{i=1}^{q}\one\{Z_i\geq (1,1)\}$. Then, on the event $\bigcap_{i=1}^{q} \mathsf{NI}_{w_i,\vec{u},\vec{t}}(t^-)$, we have 
        \begin{align*}
       (1) & \quad  \mathbb P\Big( |\bigcup_{i=1}^q \wR_{t^-}(w_i)|\ge k, |\bigcup_{i=1}^q \wC_{t^-}(w_i)| \in [\ell, R] \, \Big| \,\wFF \Big) \le (1+d^{-10})\mathbb P\Big(\sum_{i=1}^q Z_i \ge 
        (k,\ell)\Big)\,. \\ 
        (2) &  \quad \mathbb P\Big( |\bigcup_{i=1}^q \wR_{t^-}(w_i)|\ge k, |\bigcup_{i=1}^q \wC_{t^-}(w_i)| \in [\ell, R], \wJ \ge 2 \, \Big| \, \wFF \Big) \le (1+d^{-10})\mathbb P\Big(J\ge 2, \sum_{i=1}^{q} Z_i \ge 
        (k,\ell)\Big)\,. \\ 
        (3) &  \quad \mathbb P\Big( |\bigcup_{i=1}^q \wR_{t^-}(w_i)|\ge k, |\bigcup_{i=1}^q \wC_{t^-}(w_i)| \ge \ell, \max_i |\wC_{t^-}(w_i)| \le R-1 \, \Big| \, \wFF \Big) \\
        & \qquad \qquad \qquad \qquad \qquad \qquad \qquad \qquad \qquad \qquad \qquad \qquad \le (1+d^{-10})\mathbb P\Big( \sum_{i=1}^q  Z_i \ge 
        (k,\ell)\Big)\,.
    \end{align*}
\end{lemma}

\begin{proof}
    Let us begin with (1). To deal with $t^-$ properly and apply the bound implied by $t<T$, we associate to each $w_i$ a time $s_i$, with $s_1>s_2>...>s_q$ and we will take $\lim_{s_q\uparrow t} \cdots\lim_{s_1\uparrow t}$ in that order. We are taking that limit of the following expectation of product of indicators 
    \begin{align}\label{eq:sum-product-indicators}
       \mathbb E \bigg[ \mathbf 1\Big\{|\bigcup_{i=1}^{q} \wR_{s_i}(w_i)|\ge k\Big\}\mathbf 1\Big\{|\bigcup_{i=1}^q \wC_{s_i}(w_i)|\in [\ell,R]\Big\} \mathbf 1\{\mathsf{NI}_{w_1,\vec{u},\vec{t}}(s_1)\}\cdots \mathbf 1\{\mathsf{NI}_{w_q,\vec{u},\vec{t}}(s_q)\}\, \Big| \, \widetilde{\cF}_{\vec{u},\vec{t}}\bigg]\,,
    \end{align}
    noting that $\mathsf{NI}_{w_i,\vec{u},\vec{t}}(s_i)$ are measurable with respect to $\widetilde{\cF}_{\vec{u},\vec{t}}$ and guaranteed by $\bigcap_{i=1}^{q} \mathsf{NI}_{w_i,\vec{u},\vec{t}}(t^-)$ for $s_1,...,s_q$ sufficiently close to $t$. 
    Then~\eqref{eq:sum-product-indicators} can be rewritten as 
    \begin{align*}
       \mathbb E \bigg[\sum_{k\le j\le R}\,\sum_{\ell \le j'\le R} \,\sum_{j_1,...,j_q |j} \,\sum_{j'_1,...,j'_q |j'} \,\prod_{i=1}^q \mathfrak W_{i}(j_i,j'_i) \, \Big|\, \widetilde{\cF}_{\vec{u},\vec{t}}\bigg]\,,
    \end{align*}
    where $j_1,...,j_q|j$ indicates a sum over integer partitions of $j$, and where
    \begin{align*}
        \mathfrak W_{i}(j_i,j'_i) = \mathbf 1\Big\{\big|\wR_{s_i}(w_i)\setminus \bigcup_{a=1}^{i-1} \wR_{s_a}(w_a)\big|= j_i\Big\} \mathbf 1\Big\{\big|\wC_{s_i}(w_i)\setminus \bigcup_{a=1}^{i-1} \wC_{s_a}(w_a)\big|= j_i'\Big\} \mathbf 1\{\mathsf{NI}_{w_i, \vec{u},\vec{t}}(s_i)\}\,.
    \end{align*}

    We now take expectation, and iteratively reveal $\wC_{w_1,\leq s_1},\ldots \wC_{w_q,\leq s_q}$ to bound 
    \begin{align*}
        & \mathbb E\Big[\prod_{i=1}^q \mathfrak W_{i}(j_i,j_i')\mid \wF_{\vec{u},\vec{t}}\Big] \\
        & = \mathbb E\Big[\mathfrak W_{1}(j_1,j_1') \mathbb E[ \mathfrak W_{2} (j_2,j_2') \mid \wF_{(w_{1},\vec{u}),(s_1,\vec{t})}] \cdots \mathbb E[ \mathfrak W_{q}(j_q,j_q') \mid \wF_{(w_1,...,w_{q-1}, \vec{u}),(s_1,...,s_{q-1},\vec{t})}] \, \Big|\,  \wF_{\vec{u},\vec{t}}\Big]\,,
    \end{align*}
    where we are using that $\mathfrak W_{i}(j_i,j_i')$ is measurable with respect to $\wF_{(w_1,...,w_i,\vec{u}),(s_1,...,s_i,\vec{t})}$. Observe that if $(w_i,s_i)\in \bigcup_{a=1}^{i-1}\wC_{w_a,\leq s_a}$, then $\wR_{w_i}(s_i)\subseteq \bigcup_{a=1}^{i-1} \wR_{w_a}(s_a)$, and likewise for $\wC_{s_i}(w_i)$. Thus, for $j_i,j_i'\geq 1$, $\mathfrak W_i(j_i,j_i')=0$ unless the event $\mathsf{NI}_{w_i,(w_1,\ldots,w_{i-1},\vec{u}),(s_1,\ldots, s_{i-1}, \vec{t})}(s_i)$ holds. Moreover, on the event $\mathsf{NI}_{w_i,(w_1,\ldots,w_{i-1},\vec{u}),(s_1,\ldots, s_{i-1}, \vec{t})}(s_i)$, we have $\wR_{s_i}(w_i)$ doesn't intersect $\bigcup_{a=1}^{i-1}\wR_{s_a}(w_a)$ and likewise for $\wC_{s_i}(w_i)$. Combining with our assumption that $s_i<s_{i-1}<t$, we have for $j_i,j_i'\geq 1$,
    \begin{align*}
   &\E\Big[\mathfrak W_i(j_i,j_i')\Big|\wF_{(w_1,\ldots, w_i,\vec{u}),(s_1,\ldots, s_i, \vec{t})}\Big]\\
   &=\P\big(\wR_{s_i}(w_i)=j_i, \wC_{s_i}(w_i)=j_i'\big| \wF_{(w_1,\ldots, w_i,\vec{u}),(s_1,\ldots, s_i, \vec{t})}\big)\one\{\mathsf{NI}_{w_i,(w_1,\ldots,w_{i-1},\vec{u}),(s_1,\ldots, s_{i-1}, \vec{t})}(s_i)\}\leq \psi_{j_i}\psi_{j_i'},
    \end{align*}
    since we assumed $t<T$ so that $\P\big(\wR_{s_i}(w_i)\geq j_i, \wC_{s_i}(w_i)=j_i'\mid \wF_{(w_1,\ldots, w_i,\vec{u}),(s_1,\ldots, s_i, \vec{t})}\big)\leq \psi_{j_i}\psi_{j_i'}$ holds on the event $\mathsf{NI}_{w_i,(w_1,\ldots,w_{i-1},\vec{u}),(s_1,\ldots, s_{i-1}, \vec{t})}(s_i)$ for all realizations of $\wC_{u_i,\leq t_i}$'s and $\wC_{w_i,\leq s_i}$'s.
   If $j_i=0$ or $j_i'=0$, the LHS is trivially bounded by $1\le (1+d^{-100})\psi_0$. Therefore, \eqref{eq:sum-product-indicators} is at most
    \begin{align*}
        (1+d^{-100})^{2q} \sum_{k\le j \le R }\sum_{\ell \le j'\le R} \sum_{j_1,...,j_q |j} \sum_{j_1',...,j_q' | j'} \psi_{j_1} \cdots \psi_{j_q} \psi_{j_1'}\cdots \psi_{j_q'} \le (1+d^{-10}) \mathbb P\Big( \sum_{i=1}^q Z_i \ge (k,\ell)\Big)\,,
    \end{align*}
    where we used $q\leq d$. Finally, taking the limit $\lim_{s_q \uparrow t}\cdots \lim_{s_1\uparrow t}$ gives the claimed bound.

    For item (2), note that $\wJ=\sum_{i=1}^{q}\one\big\{(|\wR_{t^-}(w_i)|,|\wC_{t^-}(w_i)|)\geq (1,1)\}$ holds since $\wR_{t^-}(w_i)=\emptyset$ if and only if $\wC_{t^-}(w_i)=\emptyset$. With this observation, the proof of item (2) follows identically except that the sum only runs over integer partitions $j_1,...,j_q$ of $j$ and $j'_1,...,j'_q$ of $j'$ where the number of $i$ such that $(j_i,j'_i)\geq (1,1)$ is at least $2$. This gives the same bound but with an indicator that at least two of the $Z_i$'s are greater or equal than $(1,1)$. 

    Similarly, the proof of item (3) follows by the same reasoning, noticing that even though in this case $j,j'\ge R$ is allowed, the summands $j_1,...,j_q$ and $j_1',...,j_q'$ are all constrained to be less than $R-1$ and therefore, the bound by $\rho_{s_i}$ applies while $t< T$ as before. 
\end{proof}
\begin{remark}\label{rmk:k-zero}
    Although we stated Lemma~\ref{lem:general-tail-bound-domination-by-iid-Z} for $1\leq k \leq \ell \leq R$, it remains true for any $0\leq k \leq \ell \leq R$ under the same assumptions. This is because for $k=\ell=0$, the conclusion in (2) follows from the case $k=\ell=2$ because of the indicator $\wJ\geq 2$ while (1) and (3) become vacuous. For $k=0, \ell\geq 1$, note that $|\bigcup_{i=1}^{q}\wC_{t^-}(w_i)|\geq 1$ implies $|\bigcup_{i=1}^{q}\wR_{t^-}(w_i)|\geq 1$. 
\end{remark}
With Lemma~\ref{lem:general-tail-bound-domination-by-iid-Z} in hand, we next show that $t\mapsto \rho_t(u;k,\ell\mid \wEE)$ has negative drift. First, note that because of the indicator on $(u,t)\notin \bigcup_{i=1}^{m}C_i$, we have $\rho_t(u;k,\ell\mid \wEE)=0$ for all $t\in \{s:  (u,s)\in\bigcup_{i=1}^{m}C_i\}$. Furthermore, again thanks to the indicator, the map $t\mapsto \rho_{t}(u;k,\ell\mid \wEE)$ is continuous for all times $t \in [0,t_m]$ and differentiable in the interior of $ [0,t_m] \setminus \{s:  (u,s)\in\bigcup_{i=1}^{m}C_i\}$. Indeed, at the initial times (besides $0$) in every segment of $[0,t_m] \setminus \{s:  (u,s)\in\bigcup_{i=1}^{m}C_i\}$, we have $\rho_t(u;k,\ell\mid \wEE) =0$ because at those times $u$ was just removed from $C_i$ for some $i$, which means it was flipped from $-1$ to $+1$ so that $\wR_t(u)=\emptyset$. Similarly, $\wR_t(u)=\emptyset$ holds at the left limit points of the final times (besides $t_m$) in these segments.

Proposition~\ref{prop:exp-tail-minus-regions-conditional} will be a straightforward consequence of the following lemma. 

\begin{lemma}\label{lem:NTS-rho-negative-drift}
  Let $d\ge 7$, $\beta > \frac{C_0 \log d}{d}$ and suppose $G$ is $1$-locally-treelike.
 For $u\in V$, $1\leq k\leq \ell \leq R$, $\vec{u},\vec{t}$ such that $t_m>t$, and $(\vec{u},\vec{t})$-admissible $\vec{C}$, the following holds. Suppose that $t$ lies in the interior of $ [0,t_m] \setminus \{s:  (u,s)\in\bigcup_{i=1}^{m}C_i\}$. If $t< T$ and $\rho_t(u;k,\ell \mid \wEE)\ge \frac{1}{2}\psi_k \psi_{\ell}$, then $$\frac{d}{d t} \rho_{t}(u;k,\ell \mid \wEE) <0\,.$$ Furthermore, $\frac{d}{d t} \rho_{t}(u;k,\ell \mid \wEE) \leq dk^2$ holds for all $t$ in the interior of $ [0,t_m] \setminus \{s:  (u,s)\in\bigcup_{i=1}^{m}C_i\}$ (regardless of $t<T$).
\end{lemma}

\begin{proof}
Throughout, we assume that $t\notin \overline{\{s:(u,s)\in \bigcup_{i=1}^{m}C_i\}}$ as otherwise for $k,\ell \ge 1$ one has $\rho_t(u;k,\ell \mid \wEE)=0$. This guarantees that $\wX_{t^-}=\wX_{t}$, almost surely, even after conditioning on the event $\wEE$. Define the function of the past trajectory, whose expectation gives $\rho_t(u; k,\ell)$ by 
$$W_t(k,\ell)= [W_t(k,\ell)] ((\wX_s)_{s\leq t}) := \one\Big\{|\wR_{t}(u)|\geq k, |\wC_{t}(u)|=\ell\Big\}.$$

\noindent Recalling the definition of the flip rate $\widetilde{c}_{v,t}$ in \eqref{eq:rigid-conditional-flip-rate}, since this process's jumps are governed by a Poisson process, the time derivative is given by 
\begin{align}\label{eq:rho-time-derivative}
    \frac{d}{d t} \rho_t(u; k,\ell\mid \wEE) = \mathbb E \Big[ \sum_v \widetilde{c}_{v,t} \big(W_t^{v}(k,\ell) - W_{t^-}(k,\ell)\big) \mid\wEE\Big],
\end{align}
where $W_t^v$ to denote $W_t(k,\ell)$ evaluated on $(\wX_s^v)_{s\le t}$ instead of $(\wX_s)_{s\leq t}$, where $\wX_s^v(w)=\wX_s(w)$ for all $(w,s)\in V\times [0,t]$ except $\wX_t^v(v)\neq \wX_t(v)$, i.e. with a spin flip at $v$ at time $t$. Similarly, we use the notations $\wR_t^{v}(u), \wC_t^{v}(u), \wC_{u,\leq t}^v$ to respectively denote $\wR_{t}(u), \wC_t(u), \wC_{u,\leq t}^v$ evaluated on $(\wX_s^v)_{s\le t}$.

 We now split the sum on the right-hand side into sums over the following types of $v$ on the event $\wR_{t^-}(u) \ne \emptyset$: 
 \begin{enumerate}[label=(\Alph*)]
     \item $v\in \mathsf{Tri} :=  \mathsf{Tri}(\wR_{t^-}(v))$, in which case $v$ is flipping from $-1$ to $+1$;
     \item $v\in \mathsf{NTri}:= \wR_{t^-}(u)\setminus\mathsf{Tri}$, in which case $v$ is flipping from $-1$ to $+1$;
     \item $v\in \mathsf{Nbd}:=\{v\notin \wR_{t^-}(u): d(v,\wR_{t^-}(u)) = 1\}$, in which case $v$ is flipping from $+1$ to $-1$.
 \end{enumerate}
 Here, the vertices $v$ at distance at least two from $\wR_{t^-}(u)$ can be neglected since for those vertices, $\wR_t^{v}(u)=\wR_t(u)$ and $\wC_t^v(u)=\wC_t(u)$ (see Definition~\ref{def:minus-region}) so their contribution to the sum in~\eqref{eq:rho-time-derivative} is zero. On the event $\wR_{t^-}(u)=\emptyset$, the only case we must consider is
 \begin{enumerate}[label=(D)]
    \item $v=u$, in which case $v$ is flipping from $+1$ to $-1$.
 \end{enumerate}
 We consider these four different contributions case by case, then sum them all to bound~\eqref{eq:rho-time-derivative}:
 \begin{align}\label{eq:decomposition-of-rho-drift}
     \widetilde{\mathbb E}\Bigg[ \Big(\sum_{v\in \mathsf{Tri}}[\cdots]_v + \sum_{v\in \mathsf{NTri}}[\cdots]_v+ \sum_{v\in \mathsf{Nbd}} [\cdots]_v \Big)\mathbf{1}\{\wR_{t^-}(u)\ne \emptyset\} +  [\cdots]_u \mathbf 1\{\wR_{t^-}(u) = \emptyset\} \Bigg],
 \end{align}
 where we used the shorthand $[\cdots ]_v =\widetilde{c}_{v,t}\big(W_{t}^v(k,\ell)-W_{t^-}(k,\ell)\big)$, and $\widetilde{\mathbb E} = \mathbb E [\cdot \mid \wEE]$ for ease of notation. The bulk of the proof is dedicated to establishing $\frac{d}{d t} \rho_{t}(u;k,\ell \mid \wEE) <0$ under the assumptions $t < T$ and $\rho_t(u;k,\ell \mid \wEE)\ge \frac{1}{2}\psi_k \psi_{\ell}$. The simpler, crude bound $\frac{d}{d t} \rho_{t}(u;k,\ell \mid \wEE) \leq dk^2$ is provided at the conclusion of the proof.

 \medskip
 \noindent (A) \emph{$v\in \mathsf{Tri}$: Trifurcation points in the minus region of $u$.} Such a vertex does not flip under the rigid dynamics $(\wX_t)$ so the flip rate $\widetilde{c}_{v,t}$ is zero.

 \medskip
 \noindent (B) \emph{$v\in  \mathsf{NTri}$: Non-trifurcation points in the minus region of $u$.} These vertices contribute to the rapid shrinking of minus regions, resulting in a negative drift of $\rho_t(\cdot)$. Specifically, such a vertex flips from $-1$ to $+1$, decreasing $|\wR_{t^-}(u)|$ by exactly one while leaving $|\wC_{t^-}(u)|$ unchanged unless $|\wR_{t^-}(u)| = |\wC_{t^-}(u)| = 1$, in which case both quantities decrease to zero. That is, $|\wR_t^v(u)|=|\wR_{t^-}(u)|-1$ and $|\wC_t^v(u)|=|\wC_{t^-}(u)|$ holds  if $|\wC_{t^-}(u)|\geq 2$, and $|\wR_t^v(u)|=|\wC_t^v(u)|=0$ if $|\wC_{t^-}(u)|=1$. Thus, for any $\ell\geq k\geq 1$,
    \begin{align}\label{eq:contribution-ntri}
    \widetilde{\mathbb E} \Big[  \sum_{v\in \mathsf{NTri}} \widetilde{c}_{v,t} \big(W_t^v(k,\ell)-W_{t^-}(k,\ell)\big)\mathbf{1}\{\wR_{t^-}(u)\ne \emptyset\}  \Big]  
    = - \widetilde{\mathbb E} \Big [ \sum_{v\in \mathsf{NTri}} \widetilde{c}_{v,t} \one\Big\{|\wR_{t^-}(u)|=k,|\wC_{t^-}(u)|=\ell\Big\}\Big].
    \end{align}
    Since $v\in \mathsf{NTri}$, Lemma~\ref{lem:conditional-flip-rate} shows that $\widetilde{c}_{v,t}\geq(1+e^{-2\beta d/7})^{-1}$, $\widetilde{\P}\equiv \P(\cdot\mid \wEE)$-almost surely. Moreover, by Lemma~\ref{lem:size:trifurcation}, the number of vertices in $\mathsf{NTri}$ is at least $|\wR_{t^-}(u)|/2$. Finally, using our assumption that $t<T$ and $\rho_t(u;k,\ell \mid \wEE)\ge \frac{1}{2} \psi_k \psi_\ell$, we have for $(u,t)\notin \bigcup_{i=1}^m C_i$ that
    \[
    \widetilde{\P}\Big(|\wR_{t^-}(u)|=k,|\wC_{t^-}(u)|=\ell\Big)=\rho_t(u;k+1,\ell\mid\wEE)-\rho_t(u;k,\ell\mid\wEE)\leq \psi_{k+1}\psi_{\ell}-\frac{1}{2}\psi_k\psi_\ell,
    \]
    which is at most $-\psi_k\psi_{\ell}/3$. Putting these together, we have 
    \begin{align}\label{eq:Ntri-contribution}
        \widetilde{\mathbb E}\Big[ \sum_{v\in \mathsf{NTri}} [\cdots ]_v \mathbf{1}\{\wR_{t^-} \ne \emptyset\}\Big]  \le - \frac{k}{2} (1+e^{ - 2\beta d/7})^{-1}\cdot \frac{\psi_k \psi_{\ell}}{3}\le  - \frac{k}{7}\psi_k\psi_{\ell}\,,
    \end{align}
    so long as $d \ge 7$ and $\beta \ge 100/d$.

    \medskip
    \noindent (C) \emph{$v\in \mathsf{Nbd}$: Vertices at distance exactly one from $\wR_{t^-}(u)$.} These vertices cause the growth of a minus region and require the full strength of the conditional recursive assumption. In particular, we crucially utilize Lemma~\ref{lem:general-tail-bound-domination-by-iid-Z}. A neighbor $v$ of $\wR_{t^-}(u)$ must satisfy $\wX_{t^-}(v) = +1$ by Observation~\ref{obs:properties-of-regions}. Thus, its flipping will merge the neighboring minus spacetime clusters. To this end, let us enumerate the neighbors of $v$ that are not in $\wR_{t-}(u)$ by $w_1,\ldots, w_{d-1}$ (if there are fewer, then stop earlier). Then, flipping $v$ from $+1$ to $-1$ increases $|\wR_{t^-}(u)|\to |\wR_{t^-}(u)|+ j$, where $j=1+|\bigcup_{i=1}^{d-1}\wR_{t^-}(w_i)|$, and increases $|\wC_{t^-}(u)|\to |\wC_{t^-}(u)\cup \{v\}\cup \bigcup_{i=1}^{d-1}\wC_{t^-}(w_i)|$. Thus, for $v\in \textsf{Nbd}$, $\big(W_t^{v}(k,\ell)-W_{t^-}(k,\ell)\big)\one\big\{|\wR_{t^-}(u)|\neq \emptyset\big\}$ is bounded above by
    \begin{align*} 
    \sum_{\substack{1\le j\le k-1 \\ 0\le j'\le \ell-1}}\mathbf 1\bigg\{|\wR_{t^-}(u)|\in[k-j,k-1]\,,\,|\wC_{t^-}(u)|=\ell-j'\,,\,\Big|\bigcup_{i=1}^{d-1} \wR_{t^-}(w_i)\Big|= j-1\,,\,\Big|\bigcup_{i=1}^{d-1} \wC_{t^-}(w_i)\Big| \in [j'-1,\ell]\bigg\}.
\end{align*}
Bounding the size of $\textsf{Nbd}$ on the indicator $\one\{|\wR_{t^-}(u)|\leq k-1\}$ by $dk$, $\widetilde{\mathbb E}\big[\sum_{v\in \textsf{Nbd}}  [\cdots ]_v \mathbf{1}\{\wR_{t^-}(u) \ne \emptyset\}\big]$ is bounded above by 
\begin{align}\label{eq:Nbd-contribution}
   dk\sum_{\substack{1\le j\le k-1 \\ 0\le j'\le \ell-1}} \widetilde{\mathbb E}\bigg[W_{t^-}(k-j,\ell-j')\sup_{v\in \textsf{Nbd}}\widetilde{\E}\Big[ \widetilde{c}_{v,t}  \mathbf 1\Big\{|\bigcup_i\wR_{t^-}(w_i)|=j-1\,,\,| \bigcup_i \wC_{t^-}(w_i)| \in[j'-1,\ell]\Big\}\Big|  \widetilde{\cF}_{u,t^{-}}\Big] \bigg],
\end{align}
 where $\wF_{u,t^-}=\sigma(\wC_{u,<t})$. Note that if we denote the number of distinct non-empty $\wR_{t^-}(w_i)$'s by $\wJ$ as in \eqref{eq:def-distinct-regions}, then the number of minus neighbors of $v$ is at most $3+3\wJ$ by Observation~\ref{obs:minus-four-neighbors-outside-part} on the event that $|\wC_{t-}(u)|\leq \ell \le R$ and $|\wC_{t^-}(w_i)|\leq \ell \le R$ for any $i\leq d-1$. Combining this with Lemma~\ref{lem:conditional-flip-rate} and using $d\geq 7$, we have on the event $|\wC_{t^-}(u)|\leq \ell$ and $|\bigcup_{i=1}^{d-1}\wC_{t^-}(w_i)|\leq \ell$,
\[
\widetilde{c}_{v,t}\leq e^{ - 2\beta d/7} \mathbf 1\{\wJ= 0\} + \mathbf 1\{\wJ\ge 1\}\,.
\]
(Here we used the fact that on $\widetilde{\mathcal J} = 0$, then $v$ has at most $3 \le 3d/7$ minus neighbors, and assumption (B) of Definition~\ref{def:general-ising-process} implies $1-p_+ \le e^{ - 2\beta d/7}$.)

Plugging into the previous display, we observe that on $\one\{\wJ=0\}$, we only get a contribution from the $j=1$ and $j'\in \{0,1\}$ summands. In particular, reducing the sum to those two terms, and bounding all indicators by $1$ otherwise, the total contribution to \eqref{eq:Nbd-contribution} from $\one\{\wJ=0\}$ is at most
\begin{align}\label{eq:dist1-J0-contribution}
    dk e^{ - 2\beta d/7} \widetilde{\mathbb E}\Big[  W_{t^-}(k-1,\ell) +W_{t^-}(k-1,\ell-1)\Big]\le dk e^{ - 2\beta d/7} \psi_{k-1}(\psi_\ell + \psi_{\ell-1})\,,
\end{align}
since we assumed $t<T$. For the contribution from the $\mathbf 1\{\wJ \ge 1\}$, we must have $\big|\bigcup_{j=1}^{d-1}\wR_{t^-}(w_i)\big|\geq 1$ and $\big|\bigcup_{j=1}^{d-1}\wC_{t^-}(w_i)\big|\geq 1$ on the event $\{\wJ\geq 1\}$. Moreover, observe that if $(w_a,t^-)\in \bigcup_{i=1}^{m} C_i$ for some $a\leq d-1$, which is $\wF_{t^-}$ measurable, then $\widetilde{c}_{v,t}=0$. Indeed, conditioning on $\wEE$ freezes spacetime boundary $\bigcup_{i=1}^{m} \partial_o C_i$ to be $+1$, thus $(v,t)$ lying in the outer spatial boundary (cf. Definition~\ref{def:boundary}) means that a.s., it cannot flip conditional on $\wEE$. Thus, $\widetilde{c}_{v,t}=0$ unless the event $\bigcap_{i=1}^{d-1}\mathsf{NI}_{w_i, \vec{u},\vec{t}}(t^-)$ holds. Further, since $(w_i,t^-)\notin \wC_{u,<t}$ holds, we may restrict to the event $\bigcap_{a=1}^{d-1}\mathsf{NI}_{w_i,(\vec{u},u),(\vec{t},t^-)}(t^-)$, and appeal to item (1) of Lemma~\ref{lem:general-tail-bound-domination-by-iid-Z} with $((\vec{u},u),(\vec{t},t^-))$ replacing $(\vec{u},\vec{t})$, to bound the contribution to \eqref{eq:Nbd-contribution} from $\one\{\wJ\geq 1\}$ by
\begin{align}\label{eq:dist1-J1-contribution}
\begin{split}
   &(1+d^{-10})dk \sum_{\substack{2\le j\le k-1  \\ 0\le j'\le \ell-1}} \mathbb P\bigg( \sum_{i=1}^{d-1} Z_i \ge \big(j-1,(j'-1)\vee 1\big)\bigg) W_t (k-j,\ell-j')\\
   & \qquad \qquad \qquad \qquad \qquad \leq 2(1+d^{-10})d^3 k \sum_{\substack{2\le j\le k-1 \\ 0\le j'\le \ell-1}}  \psi_{j-1} \psi_{(j'-1)\vee 1} \psi_{k-j}\psi_{\ell-j'}\le d^{ - 10} k \psi_k \psi_\ell\,,
\end{split}
\end{align}
where $(Z_i)_i$ are i.i.d.\ from $\Psi^{\otimes 2}$, the first bound is due to item (a) of Lemma~\ref{lem:psi-tail-bound} and $t<T$ and the second bound is due to $\sum_{j'=0}^{2}\psi_{\ell-j'}\leq 3\psi_{\ell-2}$ and using item (b) of Lemma~\ref{lem:psi-tail-bound} with $q=2$.

\medskip
\noindent(D) \emph{The vertex $u$ when $\wR_{t^-}(u) = \emptyset$.} 
The final contribution to consider comes from the case when $\wR_{t^-}(u) = \emptyset$, and $v=u$ is flipping from $+1$ to $-1$. Letting $w_1,\ldots, w_d$ be the neighbors of $u$, its flipping increases $0\mapsto j$ where $j = 1+|\bigcup_{i=1}^{d} \wR_{t^-}(w_i)|$ and $0\mapsto \big|\bigcup_{i=1}^{d}\wC_{t^-}(w_i)\bigcup\{u\}\big|$. Thus, $\widetilde{\E}\big[[\cdots]_u\one\{\wR_{t^-}(u)=\emptyset\}\big]$ is at most
\begin{align}\label{eq:contribution-from-empty}
\widetilde{\E}\bigg[\widetilde{c}_{u,t}\one\Big\{\wR_{t^-}(u)=\emptyset\,,\,\Big|\bigcup_{i=1}^{d}\wR_{t^-}(w_i)\Big|\geq k-1\,,\, \Big|\bigcup_{i=1}^{d}\wC_{t^-}(w_i)\Big|\in\{\ell-1,\ell\}\Big\}\bigg]\,.
\end{align}
Meanwhile, letting $\wJ$ as in \eqref{eq:def-distinct-regions} with $q=d$, the number of minus neighbors of $u$ is at most $3\wJ$ by Observation~\ref{obs:minus-four-neighbors-outside-part} on the event $|\bigcup_{i=1}^{d}\wC_{t^-}(w_i)|\leq \ell$. Thus, by Lemma~\ref{lem:conditional-flip-rate}, we have on the event  $|\bigcup_{i=1}^{d}\wC_{t^-}(w_i)|\leq \ell$ that
\begin{align*}
   \widetilde{c}_{u,t} \le e^{ - 2\beta d/7} \mathbf 1\{\wJ\le 1\} + \mathbf 1\{\wJ \ge 2\}\,,
\end{align*}
By bounding $\one\{\wR_{t^-}(u)=\emptyset\}\leq 1$ and using item (1) of Lemma~\ref{lem:general-tail-bound-domination-by-iid-Z} (see also Remark~\ref{rmk:k-zero}), the contribution to \eqref{eq:contribution-from-empty} from $\one\{\wJ\leq 1\}$ is at most
\begin{align}\label{eq:contribution-from-u-le1}
    e^{ - 2\beta d/7} \mathbb P\Big(\sum_{i=1}^d Z_i \ge (k-1,\ell-1)\Big) \le 2d e^{ - 2\beta d/7} \psi_{k-1}\psi_{\ell-1}\,,
\end{align}
where the final inequality holds by Lemma~\ref{lem:psi-tail-bound}. 
For the contribution from the $\mathbf 1\{\wJ\ge 2\}$, by item (2) of Lemma~\ref{lem:general-tail-bound-domination-by-iid-Z} and item (b) of Lemma~\ref{lem:psi-tail-bound}, if $J=\sum_{i=1}^{q}\one\{Z_i\geq (1,1)\}$, we get at most 
\begin{align}\label{eq:contribution-from-u-ge2}
    \mathbb P\Big( J\ge 2, \sum_{i=1}^d Z_i \ge (k-1,\ell-1) \Big) \le d^{-10} \psi_{k}\psi_{\ell}\,.
 \end{align}

We now combine the contributions from~\eqref{eq:Ntri-contribution},\eqref{eq:dist1-J0-contribution},\eqref{eq:dist1-J1-contribution},\eqref{eq:contribution-from-u-le1}, and~\eqref{eq:contribution-from-u-ge2} to get: 
\begin{align*}
    \frac{d}{dt} \rho_t(u;k,\ell\mid \wEE) & \le - \frac{k}{7} \psi_k \psi_\ell + dk e^{ - 2\beta d/7} \psi_{k-1}(\psi_{\ell} + \psi_{\ell -1}) + d^{-10} k \psi_k \psi_\ell \\ 
    & \qquad + 2d e^{ - 2\beta d/7}\psi_{k-1} \psi_{\ell-1} + d^{-10} \psi_{k}\psi_{\ell}\,.
\end{align*}
Using the fact that $\psi_{k-1} \le 4d^{100} \psi_{k}$, as long as $\beta >\frac{C_0\log d}{d}$ for a sufficiently large ($d$-independent) constant $C_0$, we have that the right-hand side is at most $- \frac{k}{10} \psi_k \psi_\ell <0$ as claimed.

It remains to prove $\frac{d}{dt} \rho_t(u;k,\ell \mid \wE_{\vec{u},\vec{t},\vec{C}}) \le dk^2$ for all $t$ in the interior of $ [0,t_m] \setminus \{s:  (u,s)\in\bigcup_{i=1}^{m}C_i\}$. The expression in~\eqref{eq:rho-time-derivative} and the decomposition of~\eqref{eq:decomposition-of-rho-drift} did not use $t<T$. As before, the contribution from $v\in \textsf{Tri}$ is $0$ while the contribution from $v\in \textsf{NTri}$ is non-positive as seen in \eqref{eq:contribution-ntri}. The contribution from $v\in \textsf{Nbd}$ follows by starting with~\eqref{eq:Nbd-contribution} (again, no use of $t<T$). Note that the events in $W_{t^-}(k,\ell)$ and $W_{t^-}(k,\ell')$ are disjoint for $\ell\ne \ell'$, so we can bound $\sum_{j'\le \ell-1} W_{t^-}(k-j,\ell-j')$ by $1$. Also, since $v\in \textsf{Nbd}$ is flipping from $+1$ to $-1$, we have $\widetilde{c}_{v,t}\leq 1$ by Lemma~\ref{lem:conditional-flip-rate}. Consequently, the contribution from $v\in \textsf{Nbd}$ is at most $dk^2\widetilde{\E}[\one\{\wR_{t^-}(u)\neq\emptyset]$. By the same argument, the fourth sum in \eqref{eq:decomposition-of-rho-drift} is upper bounded by \eqref{eq:contribution-from-empty}, which is at most $\widetilde{\E}[\one\{\wR_{t^-}(u)=\emptyset\}]$. Combining all four contributions yields the claimed bound $\frac{d}{dt} \rho_t(u;k,\ell \mid \wE_{\vec{u},\vec{t},\vec{C}})\leq dk^2$.
\end{proof}

\begin{proof}[\textbf{\emph{Proof of Proposition~\ref{prop:exp-tail-minus-regions-conditional}}}]
Recall the definition of $T$ in \eqref{eq:def-T} and suppose by contradiction that $T<\infty$. Consider any $t\leq T$, $1\leq k\leq \ell\leq R$, $\vec{u},\vec{t}$ such that $t_m>t$, and $(\vec{u},\vec{t})$-admissible $\vec{C}$. Since $\rho_t(u;k,\ell\mid \wEE)=0$ for all $t\in \{s:(u,s)\in \bigcup_{i=1}^{m}C_i\}$ and $t\mapsto \rho_t(u;k,\ell\mid \wEE)$ is continuous, Lemma~\ref{lem:NTS-rho-negative-drift} implies that $\rho_t(u;k,\ell\mid \wEE)\leq \psi_k \psi_{\ell}/2$ for $t\leq T$. In particular, whenever $T<t_m$, $\rho_T(u;k,\ell\mid \wEE)\leq \psi_k\psi_{\ell}/2$ (this includes $T=0$ since $\rho_0(u;k,\ell)=0$). Again by Lemma~\ref{lem:NTS-rho-negative-drift}, for $\eps>0$ small enough such that $T+\eps<t_m$, we have $\rho_{T+\eps}(u;k,\ell\mid \wEE)\leq \psi_k \psi_{\ell}/2+dk^2\eps$. Thus, $\max_{u\in V} \rho_{T+\eps}(u;k,\ell)\leq \psi_{k}\psi_{\ell}$ for any $1\leq k \leq \ell \leq R$ holds for $\eps>0$ such that $dR^2\eps\leq \psi_R^2/2$. This contradicts the definition of $T$. Therefore $T=\infty$, which concludes the proof.
\end{proof}
\subsection{Properties of convolutions of $\Psi$}\label{sec:deferred-proof-of-psi-bounds}

In this section, we prove Lemma~\ref{lem:psi-tail-bound}. To prove item (b), we consider an exponential tilt of $\Psi$, where the following result will be useful.
\begin{lemma}\label{lem:simple}
    Fix constants $\Gamma_0 \in [0,1]$ and $\Gamma_1 \geq0$. Consider i.i.d. random variables $(Y_i)_{i\geq 1}\in \Z_{\geq 0}$ distributed as $\P(Y_i=k)=\phi_k$, where $\phi_k\equiv \Gamma_0 k^{-2}d^{-1000-\Gamma_1 k}$ for $k\geq 1$, and $\phi_0\equiv 1-\sum_{k\geq 1} \phi_k$. Then, for any integers $q\geq 1$ and $k\geq 1$,
    \begin{equation}
    \label{eq:estimate:convolution}
    \P\Big(\sum_{i=1}^{q}Y_i=k\Big)\leq f_{q}\cdot \phi_k\,,~~\textnormal{where}~~f_{q}=\sum_{i=0}^{q-1} (1+16d^{-1000})^i.
    \end{equation}
\end{lemma}
\begin{proof}
    We proceed by induction on $q\geq 1$. The base case is trivial. Assume $\P(\sum_{i=1}^{q}Y_i=k)\leq f_{q}\cdot \phi_k$ for any $k\geq 1$. Then,
    \[\P\Big(\sum_{i=1}^{q+1}Y_i=k\Big)=\sum_{j=0}^{k}\P\Big(\sum_{i=1}^{q}Y_i=k-j\Big)\cdot\P(Y_{q+1}=j)\leq (f_{q}+1)\phi_k+f_{q}\sum_{j=1}^{k-1}\phi_{j}\phi_{k-j}\,.
    \]
    Using $\Gamma_0 \leq 1$, the sum in the right-hand side is at most
\begin{equation}\label{eq:technical}
2d^{-1000}f_{q}\phi_k \sum_{j=1}^{\lceil \frac{k-1}{2}\rceil}j^{-2}(1-j/k)^{-2}\leq 16d^{-1000} f_q\phi_k,
     \end{equation}
    where we used that the above sum is at most $4\sum_{m\geq 1}m^{-2}<8$. This concludes the induction.
    \end{proof}
 
    \begin{proof}[\textbf{\emph{Proof of Lemma~\ref{lem:psi-tail-bound}}}] Throughout, we let $(Y_i)_{i\geq 1}\stackrel{i.i.d.}{\sim}\Psi$. To prove item (a), it suffices to prove that $\P(\sum_{i=1}^{q}Y_i\geq k)\le \sqrt{2} q \psi_k$ for all $k\geq 0$ and $q\leq d$. Since $1\leq (1+d^{-10})\psi_0$, this is trivial for $k=0$. For $k\geq 1$, Lemma~\ref{lem:simple} applied with $\Gamma_0=1$ and $\Gamma_1=100$ gives the desired conclusion by noting that $f_q\leq \sqrt{2}q$ for $q\leq d$. 

    For item (b), denote $Z_i=(Y_i,Y_i')$ and let $I_q=\sum_{i=1}^{q}\one\{Y_i\geq 1\}$ and $I'_q=\sum_{i=1}^{d}\one\{Y_i'\geq 1\}$. Since the event $\{J\geq 2\}$ is contained in $\{I_q\geq 2, I'_q\geq 2\}$, it suffices to prove for $k\geq 0$ that
     \[
     \P\Big(I_q\geq 2\,,\, \sum_{i=1}^{q}Y_i\geq k\Big)\leq d^{-10}\psi_{k+1}.
     \]
     Since $I_q\geq 2$ implies $\sum_{i=1}^{q}Y_i\geq 2$ and $\psi_k$'s are decreasing in $k$, it suffices to consider $k\geq 2$.
 Consider the tilted random variable $(\widehat Y_i)_{i \geq 1}$ which are independent and distributed as 
    \[
    \P\big(\widehat Y_i=j\big)= \frac{\psi_{j} \cdot d^{100j}}{\E[ d^{100 Y_i}]},\quad j\geq 0\,. 
    \]
 Let $\widehat I_q=\sum_{i=1}^{q}\one\{\widehat Y_i \geq 1\}$ for $q\geq 1$. Then,
    \begin{align}\label{eq:exponential-tilt}
    \P \Big(I_q \geq 2\,,\,\sum_{i=1}^{q} Y_i \geq k \Big)
    =\left(\E[d^{100 Y_1}]\right)^q \sum_{j=k}^{\infty}  d^{-100j}\P\Big(\widehat I_q\geq 2\,,\, \sum_{i=1}^{q}\widehat Y_i=j \Big)\,.
    \end{align}
    Note that $\E[d^{100 Y_i}]=\psi_0+\sum_{j \geq 1} j^{-2}d^{-1000}\leq 1+ 2d^{-1000}$, so $\left(\E[d^{100 Y_1}]\right)^q\leq 2$ for all $q\leq d$. Thus, the proof is complete if we show that for all $j\geq 2$ and all $q$ such that $1\leq q\leq d$
    \begin{equation}\label{eq:final:goal}
    \P\Big(\widehat I_q\geq 2\,,\, \sum_{i=1}^{q}\widehat Y_i= j\Big)\leq 8q^2 j^{-2}d^{-2000}\,.
    \end{equation}
   We prove this by induction on $q\geq 1$. For $q=1$, the left-hand side is $0$. Assuming~\eqref{eq:final:goal} for some $q$ such that $q\leq d$, note that
   \[
   \begin{split}
   \P\Big(\widehat I_{q+1} \geq 2\,,\, \sum_{i=1}^{q+1}\widehat Y_i= j\Big)
\leq 8q^2 j^{-2}d^{-2000}+\sum_{j'=1}^{j-1}\P\Big(\sum_{i=1}^{q} \widehat Y_i=j-j'\Big)\cdot\P\Big( \widehat  Y_{q+1}=j'\Big),
   \end{split}
   \]
   where we used the inductive hypothesis to bound the contribution from the event $\{\widehat Y_{q+1}=0\}$. Observe that the distribution of $\widehat Y_i$ is of the form given in Lemma~\ref{lem:simple} where $\Gamma_0=\big(\E[d^{100 Y_i}]\big)^{-1}\leq 1$ and $\Gamma_1=0$, thus
   \[
  \P\Big(\sum_{i=1}^{q} \widehat Y_i=j-j'\Big)\leq f_{q}\cdot (j-j')^{-2}d^{-1000}\leq 2q \cdot (j-j')^{-2}d^{-1000},
   \]
   where we used $q \leq d$ in the last inequality. Plugging this into the previous display,
   \[
    \P\Big(\widehat I_{q+1} \geq 2\,,\, \sum_{i=1}^{q+1}\widehat Y_i= j\Big)
\leq 8q^{2} j^{-2}d^{-2000}+2q d^{-2000}\sum_{j'=1}^{j-1} (j-j')^{-2}j'^{-2}\leq 8(q+1)^2 j^{-2}d^{-2000}\,,
   \]
  where the last inequality holds since $\sum_{j'=1}^{j-1} (j-j')^{-2}j'^{-2}\leq 4j^{-2}$ holds by the same argument as the one used in \eqref{eq:technical}. Therefore, \eqref{eq:final:goal} holds, which concludes the proof.
\end{proof}

\subsection{Lower bound on time to develop logarithmically sized minus regions}
We can use Proposition~\ref{prop:exp-tails-minus-regions} to lower bound the time it takes for the rigid dynamics to develop a minus spacetime cluster of size $R = \frac{1}{4} \log_d n$, and guarantee in that way that for polynomially long times, all rigid dynamics' clusters are small, and obey tail bounds of the form of $\Psi$ from~\eqref{eq:psi_k}. 

For that purpose, define the hitting time 
\begin{align}\label{eq:tau-treelike}
    \tau^R = \inf\big\{t\ge 0: \max_v |\wC_t(v)| \ge R\big\}\,.
\end{align}
Since $\widetilde{R}_t(u)\subset \widetilde {\cC}_t(u)$ and the latter is a connected subset of $G$, this implies that $\diam(\widetilde{\cR}_t(u))< R$ while $t<\tau^R$.

\begin{cor}\label{cor:tau-is-polynomially-large}
        Fix $d\ge 7$, $\beta \ge \frac{C_0 \log d}{d}$ for $C_0$ a large ($d$-independent) constant, and suppose $G$ is $1$-locally-treelike. Then 
        \begin{align*}
            \mathbb P(\tau^R \le n^2) = O (n^{-10})\,.
        \end{align*}
\end{cor}
   \begin{proof}
        Let $t>0$, $\Delta = n^{-12}$ and begin by constructing a time mesh, $$ M = \{\Delta, 2\Delta,...,tn^{12}\Delta\} \subset  [0,t]\,.$$
For each $i$, the probability that there are two clock rings in $V\times [i\Delta, (i+1)\Delta]$ is at most 
\begin{align*}
    \mathbb P(|\mathcal T\cap [i\Delta, (i+1)\Delta]|\ge 2) \le \mathbb P(\text{Pois}(n\Delta) \ge 2)\le \exp( - 2 \log n^{-12} +2) = O(n^{-24})\,.
\end{align*}
By a union bound with $tn^{12}$ terms, the probability that there exists $i$ such that there are two clock rings in $V\times [i\Delta, (i+1)\Delta]$ is then at most $O(t n^{-12})$. 

Now notice that $\tau^R$ must happen at a clock-ring time, and on the event of no two clock rings in any segment, at the mesh time (call it $s$) immediately prior to that clock-ring time, there must be some vertex $v$ whose $d$ adjacent vertices $w_1,...,w_d$ are such that $\max_{i} |\wC_{s}(w_i)| \le R-1$ but $|\bigcup_{i=1}^d \wC_{s}(w_i)|\ge R-1$.  
We can therefore bound  
\begin{align*}
    \mathbb P(\tau^R \le t) &\le \bigcup_{v} \bigcup_{s\in M} \mathbb P\big( \max_{w_i \sim v} |\wC_s(w_i)| \le R-1, |\bigcup_{w_i \sim v} \wC_s(w_i)| \ge R-1\big) \ + O(t n^{-12}) \,.
\end{align*}
Applying item (3) of Lemma~\ref{lem:general-tail-bound-domination-by-iid-Z} with $\vec{u} = \emptyset$, and the tail bounds on $\Psi$ from Lemma~\ref{lem:psi-tail-bound}, we get 
\begin{align*}
    \mathbb P(\tau^R\le t) \le (1+d^{-7}) n |M| \mathbb P\Big(\sum_{i=1}^{d} Z_i \ge (1,R-1)\Big) + O(tn^{-12}) \le 2d^2t n^{13} \psi_{R-1} + O(tn^{-12})\,, 
\end{align*}
where $Z_i \sim \Psi^{\otimes 2}$ are i.i.d. 
Recalling $\psi_k$ and plugging in $R = \frac{1}{4} \log_{d} n$, we find that for $n$ sufficiently large, this is at most $2 d t n^{13} d^{-1000- \frac{100}{4}\log_d n} + O(tn^{-12})$. This is $O(tn^{-12})$, and we get the desired by setting $t= n^2$.  
    \end{proof}

    \begin{cor}\label{cor:expectation-bound}
        We have for all $t\ge 0$ and $r\ge 1$, that $\mathbb P(|\wR_{t}(w)|\ge r)\le \sum_{k\ge r}\psi_{k} + O(tn^{-12})$. In turn, for every $w\in V$ and every $t\le n^2$, we have $
            \mathbb E[ |\wR_t(w)|] \le d^{ - 100}$.
    \end{cor}
    \begin{proof}
        For the first bound, split 
        \begin{align*}
            \mathbb P(|\wR_t(w)|\in [r,R], |\wC_t(w)| \in [r,R]) + \mathbb P(|\wC_t(w)|\ge R) \le \sum_{k,\ell \ge r} \psi_{k}\psi_\ell + O(tn^{-12})\,,
        \end{align*} 
        where the first term followed by Proposition~\ref{prop:exp-tails-minus-regions} and the second from $\{|\wC_t(w)|\ge R\} \implies \{\tau^R \le t\}$ and Corollary~\ref{cor:tau-is-polynomially-large}. Since $(\psi_r)_{r\ge 1}$ is a probability mass function, this gives the claimed bound. 

        For the expected value, sum the tail bound provided over $r\ge 1$ to get a bound of $(1+d^{-10}) \psi_1 + O(n^{-10})$ which in turn is at most $d^{-100}$.  
    \end{proof}

\section{Fast quasi-equilibration from sufficiently biased initializations}\label{sec:legacy-analysis}
In this section, we use the bounds from the previous section on the sizes of minus information regions, to show that if we start with a sufficiently biased (but $\beta, d$ independent) initialization in the two-spin dynamics, the minus spacetime clusters of the initial minuses (called the legacy region) die out in $O(\log n)$ time. We then couple a Potts Glauber dynamics to the two-spin dynamics per Example~\ref{ex:Potts} to translate that to an $O(\log n)$ mixing time bound in the phase. 

In what follows, we will use $\gamma_0$ to denote a sufficiently small constant independent of $d$ ($\gamma_0 = 10^{-5}$ would suffice for instance, and already if we restricted to $d\ge 20$, then $\gamma_0 = \frac{1}{20}$ would suffice). Recall that $(Y_t)_{t\ge 0}$ denotes the continuous-time Potts Glauber dynamics and when we use $m$ for the magnetization of the Potts model it is defined in~\eqref{eq:Potts-phase-magnetization} (maximized by the all-$1$ state). 

\begin{theorem}\label{thm:from-d-independent-epsilon}
 If $G\sim \mathcal G_d(n)$ for $d\ge 7$, the following holds with probability $1-o(1)$.  There exists a universal constant $C_0$ such that for every $\beta> \beta_0 = \frac{C_0 \log (qd)}{d}$, for every $Y_0$ having $m(Y_0) \ge 1-\gamma_0$,  
 \begin{align*}
     \|\mathbb P_{Y_0}(Y_t \in \cdot ) - \pi_G(\cdot \mid \Omega^1)\|_\tv \le n^{ - 10}\,, \qquad \text{for every} \qquad C \log n \le t \le e^{ n/C}\,,
 \end{align*}
 for some $C(\beta,d,q)$. As $d\to\infty$, we could take $C = 1+o_d(1)$ in the lower bound on $t$.
\end{theorem}

Our proof will go by coupling the Potts dynamics to the two-spin dynamics of Example~\ref{ex:Potts} with initialization that is sufficiently biased towards plus. Extinction of the latter's minus spacetime clusters touching the initialization will imply mixing to the plurality-$1$ phase for the Potts Glauber dynamics. Therefore, throughout this section we use $X_t$ for the two-spin dynamics of Definition~\ref{def:general-ising-process} and $Y_t$ for Potts Glauber dynamics.

\subsection{The legacy region and disagreement sets}
Recall the grand monotone coupling of continuous-time two-spin dynamics chains from Definition~\ref{def:grand-coupling} using clock ring times $\mathcal T = (t_{v,j})_{v,j}$ and uniform random variables at each spacetime clock ring point $\mathcal U = (U_{v,j})_{v,j} = (U_{t_{v,j}})_{v,j}$. Notationally, the chain $X_t$ in what follows will be initialized from an $x_0$ with $m(x_0)\ge 1-\gamma_0$.  We define a union of minus spacetime regions of $X_t$ that contain all the negative information of the initialization. 

It will be helpful to be able to refer to general spacetime clusters that have died out by time $t$ in the set of spacetime clusters of Definition~\ref{def:minus-spacetime-clusters}. For that purpose, for a spacetime point $\mathbf{a} =(u,s)$ for $s\le t$, let $\cC_{\mathbf{a},\le t}$ be the (possibly empty) minus spacetime cluster containing $(u,s)$ in the set of minus spacetime clusters generated by the history of $(X_s)_s$ up to time $t$. 

\begin{definition}\label{def:region-evolution}
    For an initialization $x_0$, construct the minus spacetime clusters of $(X_{s})_{0\le s \le t}$ as in Definition~\ref{def:minus-spacetime-clusters} (per Remark~\ref{rem:spacetime-cluster-construction-general-init}). Define the \emph{legacy spacetime cluster} up to time $t$ as 
    \begin{align*}
        \mathcal L_{\le t} = \mathcal L_{\le t}((X_s)_{s\in [0,t]}) = \bigcup_{\mathbf{a} = (v,0)\,:\ x_0(v) = -1} \mathcal C_{\mathbf{a},\le t}((X_s)_{s\in [0,t]})\,.
    \end{align*}
    The \emph{legacy region} at time $t$ is the time-$t$ slice $\mathcal L_t = \mathcal L_{\le t} \cap (V\times \{t\})$. This is also the union of all minus regions of spacetime clusters touching $\{v\in V: x_0(v) = -1\}$. 
\end{definition}

The role played by the legacy region is that it can be used to confine the vertex set in a Potts Glauber dynamics chain that are not in the dominant state. That observation is the content of the following lemma. 

For any Potts initialization $y_0\in [q]^n$, let $x_0\equiv x_0(y_0)$ be the $\{\pm 1\}^n$-valued initialization which sets $x_0(v) =+1$ if $y_0(v) =1$ and $x_0(v) = -1$ if $y_0(v)\in \{2,\ldots,q\}$. Consider the two-spin Markov process $(X_t^{x_0})_{t\geq 0}$ in Definition~\ref{def:general-ising-process} initialized at $x_0$ and following the update rule \eqref{eq:update-rule} with $p_{+}:\{\pm 1\}^{d}\to [0,1]$ defined as in~\eqref{eq:Potts-two-spin-chain} with $\beta>0$ taken so that $\betap=2 \beta+\frac{C\log (q-1)}{d}$ as set in Lemma~\ref{lem:potts-domination}.

\begin{lemma}\label{lem:disagreement-subset-of-legacy}
    Fix any $q\ge 2$. For any $y_0\in [q]^n$ and $x_0 = x_0(y_0)$, there exists a coupling of the processes $(Y_t^{y_0})_{t\ge 0}, (Y_t^1)_{t\ge 0}$ and $ (X_t^{x_0})_{t\ge 0}$ such that for all $t\geq 0$,
    \begin{align*}
       \cD_{t}:=\big\{v\in V: Y_t^{y_0}(v) \ne Y_t^1(v)\big\} \subseteq \cL_{t}\,,
    \end{align*}
    i.e., the disagreement set in the Potts chains is a subset of the two-spin dynamics' legacy region. 
\end{lemma}

\begin{proof}
    We use the same Poisson clock ring sequence $\cT=\bigcup_{v\in V} \bigcup_{j\ge 1} \{t_{v,j}\}$ for all three chains. When a clock rings at $t=t_{v,j}$, we use the same uniform random variable $U_{v,j}\sim \textnormal{Unif}([0,1])$ to update the vertex $v$ at time $t$ for the three chains as follows. Given configurations $(Y_{t^-}^{y_0}, Y_{t^-}^{1}, X_{t^-}^{x_0})$, let the probability vectors $(p_k^{y_0})_{k\in [q]}$ and $(p_k^{1})_{k\in[q]}$ denote 
    \begin{align*}
        p_k^{y_0} &= \pi(\sigma(v) = k \mid (\sigma(w))_{w\sim v} = (Y_{t ^-}^{y_0}(w))_{w\sim v})\,, \\   p_k^{1} & = \pi(\sigma(v) = k \mid (\sigma(w))_{w\sim v} = (Y_{t ^-}^{1}(w))_{w\sim v})\,.
    \end{align*}
    and let $p_{+}\equiv p_{+}\big((X_{t^-}^{x_0}(w))_{w\sim v}\big)$. Partition the interval $[0,1]$ as $[0,1]=\sqcup_{k\in [q]}I_k$ where $I_1=[0,p_1^{y_0})$ and $I_k$ has length $p_k^{y_0}$ for $k\in [q]$. We then set $Y_{t}^{y_0}(v)=k$ if $U_{v,j}\in I_{k}$. Similarly, let $[0,1]=\sqcup_{k\in [q]}J_k$ where $J_1=[0,p_1^{1})$ and $J_k$ has length $p_k^{1}$. We set $Y_t^{1}(v)=k$ if $U_{v,j}\in J_k$. Finally, $X_{t}^{x_0}(v)=+1$ if $U_{v,j}\in [0,p_{+})$, and otherwise, set $X_{t}^{x_0}(v)=-1$.

    An important property of this coupling is that if $p_{+}\leq p_{1}^{y_0}\wedge p_1^{1}$, then $X_{t}^{x_0}(v)=+1$ implies $Y_{t}^{y_0}(v)=Y_{t}^{1}(v)=1$. With this property, we argue that at any time $t\geq 0$, we have the inclusions
    \begin{align}\label{eq:inclusion}
    \cD_t\subseteq \cL_t\quad\textnormal{and}\quad \cY_t \subseteq \cS_t:=\{v\in V:X_{t}(v)=-1\}\,,
    \end{align}
    where $\cY_{t}$ is the set of $v\in V$ such that either $Y_{t}^{y_0}(v)\neq 1$ or $Y_t^{1}(v)\neq 1$. Note that if these inclusions hold for $\cT=(t_i)_{i\geq 1}$, where $t_0:=0\leq t_1<t_2<\ldots$, then they hold for all $s\in [t_i, t_{i+1}^-)$ so we proceed by induction on $i\geq 0$. Evidently, these inclusions hold for $t_0=0$ by our definition of $x_0\equiv x_0(y_0)$.
    
    Suppose \eqref{eq:inclusion} holds for all $(t_{i})_{i\le k}$, and the next clock ring time $t= t_{k+1}$ is on a vertex $v$. First, observe that if $\sum_{w\sim v}\one \{X_{t^-}(w)=+1\}<\frac{4d}{7}$, then $X_t(v)=-1$ holds by our choice of $p_{+}(\cdot)$. Thus, $\cL_{t}\supseteq \cL_{t^-}\cup \{v\}$ and $\cS_t=\cS_{t^-}\cup \{v\}$. Thus, no matter if $v$ is added to the sets $\cD_{t^-}$ or $\cY_{t^-}$ or not, the inclusions \eqref{eq:inclusion} are retained. Therefore, it suffices to consider the case $\sum_{w\sim v}\one\{X_{t^-}(w)=+1\}\geq \frac{4d}{7}$. In this case, we have by the second inclusion in our inductive hypothesis that $\sum_{w\sim v}\one\{Y_{t^-}^{y_0}(w)=1\}\geq \frac{4d}{7}$ and likewise for $Y_{t^-}^1$. By Lemma~\ref{lem:potts-domination}, this implies that $p_{+}\leq p_1^{y_0}\wedge p_1^1$, so by our construction of the coupling, $X_{t}^{x_0}(v)=+1$ implies $Y_t^{y_0}(v)=Y_t^1(v)=1$. In particular, $\cY_t\subseteq \cS_t$ is retained. 
    
    To prove that $\cD_t\subseteq \cL_t$ is retained, observe that the only way this is violated at time $t$ assuming $\cD_{t^-}\subseteq \cL_{t^-}$ is when $v\in \cD_{t}\setminus \cD_{t^-}$ whereas $v$ is removed from $\cL_{t^-}$ at time $t$. 
    This cannot happen because if $v$ was removed from $\cL_{t^-}$, then necessarily $X_{t}^{x_0}(v)=+1$, in which case $Y_t^{y_0}(v)=Y_t^1(v)=1$, so there $Y_t^{y_0}$ and $Y_t^1$ must agree on $v$ at time $t$. This concludes the proof of the inclusion~\eqref{eq:inclusion}.
\end{proof}

\subsection{Properties of the legacy region}
Lemma~\ref{lem:disagreement-subset-of-legacy} reduces our main task to establishing that the legacy region of the two-spin dynamics initialized with $x_0$ having at least $(1-\gamma_0)n$ many pluses becomes empty after time $O(\log n)$, as that implies the same of the Potts Glauber disagreement set $\cD_t$.  

\begin{definition}\label{def:admissible-legacy}We say a spacetime subset $L \subset V\times [0,t]$ is a \emph{$t$-admissible legacy cluster} if it is made of the union of $(u_i,t)$-admissible minus spacetime clusters with each constituent cluster intersecting the initial minus set $\cD_0^{x_0} = \{v: x_0(v) = -1\}\times \{0\}$, and every point in $\cD_{0}^{x_0}$ belonging to some cluster in $L$. 
\end{definition}

By Observation~\ref{obs:spacetime-cluster-internally-measurable} and the fact that this is a union of spacetime clusters, the legacy region $\cL_{\le t}$, and in particular events of the form $\cE_{t,L}$ where we are using 
$$\cE_{t,L} = \{\cL_{\le t} = L\}\,,$$ for $t$-admissible legacy cluster $L \subset V\times [0,t]$, are measurable with respect to the filtration $\cF_{\cL,t}$ which is the sigma-algebra generated by $(\cF_{\mathbf{a},\le t})_{\mathbf{a}\in \cD_0^{x_0}}$.

The important property we rely on is that minus regions of $X_t$ in the complement of the legacy region (through which the legacy region grows) are stochastically dominated by those of the rigid dynamics initialized from all-$+1$ which we controlled in Section~\ref{sec:exp-tails-minus-spacetime-regions}. 
We begin with a variant of the flip rate monotonicity of Lemma~\ref{lem:conditional-flip-rate} that compares the standard two-spin dynamics' flip rates on the complement of the legacy cluster, conditional on the legacy equaling some $L$, with an unconditional realization of the two-spin dynamics initialized from all-$+1$. 

Given an admissible legacy spacetime cluster $L\subset V\times [0,t]$, define the conditional flip rate  
\begin{align}\label{eq:legacy-conditioned-flip-rate}
    c_{v,s} := c_{v,s}((X_r)_{r<s},L) = \lim_{\epsilon \downarrow 0} \frac{1}{\epsilon} \mathbb P( X_{s+\epsilon}(v) \ne X_{s^-}(v) \mid (X_r)_{r<s}, \cE_{t,L})\,.
\end{align}

\begin{lemma}\label{lem:conditional-flip-rates-legacy}
    Fix an initialization $x_0$ for $(X_t)_{t\ge 0}$. Let $L \subset V\times [0,t]$ be a $t$-admissible legacy spacetime cluster.  For any $s<t$, and any $v$ such that $(v,s)\notin L$, if $X_{s^-}(v) = +1$, 
    \begin{align*}
        c_{v,s} \le 1-p_+((X_{s^-}(w))_{w\sim v})\,,
    \end{align*}
    and if $X_{s^-}(v) = -1$, then 
        \begin{align*}
        c_{v,s} \ge p_+((X_{s^-}(w))_{w\sim v})\,.
    \end{align*}
\end{lemma}
\begin{proof}
    The proof is the same as the proof of Lemma~\ref{lem:conditional-flip-rate}. The only distinction is that it is not the rigid dynamics we are considering but the standard two-spin dynamics. But this still has local monotone update rules, and therefore the inductive proof that $\cE_{t,L}$ only becomes more probable given the flip at $(v,s)$ to $+1$ happens (or the flip to $-1$ doesn't happen) goes through identically. 
\end{proof}
 
\begin{lemma}\label{lem:spacetime-cluster-increasing-on-exterior}
    Consider any initialization $x_0$ for $(X_t)_{t\ge 0}$ fix any $t$-admissible  $L$. The conditional law of $X(V\times [0,t] \setminus L)$ given the event $\cE_{t,L}$ stochastically dominates the law of $X^+(V\times [0,t] \setminus L)$. 
\end{lemma}

\begin{proof}
 We construct a coupling of $(X_s)_{s\le t}$ (conditional on $\cE_{t,L}$) with the unconditional $(X_s^+)_{s\le t}$ such that for all $s\le t$ and $w: (w,s)\notin L$, we have $X_s(w)\ge X_s^+(w)$. 

 To show the domination, note first that it holds at time $t=0$, since all vertices that are minus at time $0$ in $x_0$ are necessarily in $L$. 
 Now using the above notation, let $c_{v,s}(X_{s^-}, L)$ be the flip rate for $(X_s)_{s\le t}$ conditioned on $\mathcal E_{t,L}$, and let $c^+_{v,s}= c^+_{v,s}(X^+_{s^-})$ be the flip rate for the unconditional two-spin dynamics initialized from all-plus. 
 
 The coupling of the two processes goes as follows. At time $s$, to each vertex $v$ with $(v,s) \notin L$, associate an independent Poisson process of intensity 
 \begin{align*}
     \lambda_v(r) = \max\{c_{v,s+r}, c^+_{v,s+r}\} \qquad r\ge 0\,.
 \end{align*}
 Notice that by Lemma~\ref{lem:conditional-flip-rates-legacy}, if $X_{s^-}(w)\ge X_{s^-}^+(w)$ for $w\sim v$ and $X_{s^-}(v) = -1$ and neither process has had any flips in $[s,s+r)$, then the maximum in $\lambda_v(r)$ is attained by $c_{v,s+r}$. On the other hand,  if $X_{s^-}(v) = +1$, then the maximum is attained by $c^+_{v,s+r}$. 
 
 At the first time after $s$ that some vertex $v$'s Poisson process has a clock ring, say at spacetime point $(v,s+r)$, 
 \begin{itemize}
     \item If $\lambda_v(r) = c_{v,s+r}$ then flip $X_{s+r^-}(v)$ with probability $1$ and flip $X_{s+r^-}^+(v)$ with probability $1-\frac{1}{c_{v,s+r}} (e^{c_{v,s+r} - c^+_{v,s+r}}-1)$. 
     \item If $\lambda_v(r) = c^+_{v,s+r}$ then flip $X^+_{s+r^-}(v)$ with probability $1$ and flip $X_{s+r^-}(v)$ with probability $1-\frac{1}{c^+_{v,s+r}} (e^{c^+_{v,s+r} - c_{v,s+r}}-1)$.
 \end{itemize}
 Then at time $s+r$, restart Poisson processes at every vertex with new intensities $\lambda_v$ given by the new conditional flip rates. It is straightforward to check that this is a valid coupling of the two processes.  
    To argue that this coupling maintains monotonicity in the spacetime complement of $L$, suppose that there is a flip at time $s$ at vertex $v$ with $(v,s)\notin L$, and suppose inductively that we have $X_{s^-}(w)\ge X^+_{s^-}(w)$ for all $(w,s^-)\notin L$. If $(v,s) \in \partial_o L$, it must have $X_{s}(v) = +1$ and therefore also $X_s(v)\ge X_s^+(w)$ and no other vertex changes until the next flip time. If $(v,s)\notin \partial_o L$ then all of $v$'s neighbors in $X_{s^-}$ are more plus than they are in $X_{s^-}^+$ and by construction of the coupling the domination will be retained. 
 \end{proof}

Combining Lemma~\ref{lem:spacetime-cluster-increasing-on-exterior} with the domination of the rigid dynamics $\wX_t$ by $X_t^+$ everywhere per~\eqref{eq:rigid-below-normal-dynamics}, we get the following. 

\begin{corollary}\label{cor:domination-by-i.i.d.-Psi}
             Consider any initialization $x_0$ for the two-spin dynamics $(X_t)_{t\ge 0}$. Let $\wX_t$ be the rigid dynamics initialized from all-$+1$. For any $t$ and any $t$-admissible spacetime legacy cluster $L$, we have that $$\text{Law}(X(V\times [0,t]\setminus L) \mid \cE_{t,L}) \succeq \text{Law}(\wX(V\times [0,t]\setminus L))\,.$$ In particular, for any $w$ with $(w,t)\notin L$, we have 
             $$\text{Law}(\cR_t(w) \mid \cE_{t,L}) \preceq \text{Law}(\wR_{t}(w))\,,$$
             where the partial order on the minus regions is the natural one given by inclusion. 
\end{corollary}

Corollary~\ref{cor:domination-by-i.i.d.-Psi} allows us to control the growth of the legacy region $\cL_t$ (via merging with other spacetime regions) by bounding the neighboring spacetime regions using Corollary~\ref{cor:expectation-bound}. 

\subsection{Spacetime analysis of the legacy region's evolution}

In this subsection, we establish the following, which bounds the extinction time of the legacy.

\begin{proposition}\label{prop:extinction-time-of-legacy}
    Fix any initialization $x_0$ with at least $(1-\gamma_0) n$ many $+1$ spins. There exists a constant $C$ such that for any $\beta\ge \beta_0 = \frac{C_0 \log d}{d}$, with probability $1-o(1)$, $G\sim\mathcal G_{d}(n)$ is such that  
    \begin{align*}
        \mathbb P(\cL_{C\log n} \ne \emptyset ) \le  n^{-10}\,.
    \end{align*}
    Moreover, as $d\to\infty$, we could take $C = 1+o_d(1)$ and still get $o(1)$-probability on the right. 
\end{proposition}

Unlike the analysis in Section~\ref{sec:minus-regions} for the drift of $\wC_t(u)$, since $\mathcal L_t$ is macroscopic, we cannot rely on the local treelike geometry of $G$ in the drift analysis, and need to instead rely on more global expansion properties of $G$. Specifically we use the following two lemmas. (If we were only interested in $d\ge 20$ the first would be sufficient, while if we were only interested in sufficiently large biases, regardless of the $d\to\infty$ dependence of the bias threshold $\epsilon_0$, the second would be sufficient.) 

    \begin{lemma}\label{lem:expansion-edgenumber-facts:alternate}
    Fix any $d\ge 3$. For any constants $\delta,\eta$ such that $0<\delta<\eta<1/2$, the following holds for $G\sim \mathcal G_d(n)$ with probability $1-o(1)$: for all sets $|S|$ with $|S|\leq \delta n$, 
    \[
    \left|\big\{u\in V:\deg_S(u)\geq \eta d\big\}\right|\leq \frac{4}{d(\eta-\delta)^2}|S|,
    \]
    where $\deg_S(u)$ denotes the number of vertices in $S$ that are adjacent to $u$.
    \end{lemma}
    \begin{proof}
        Let $A\in \R^{n\times n}$ be the adjacency matrix of $G$. The maximum eigenvalue of $A$ is $d$ where the corresponding eigenvector is $\bone\in \R^n$, the all-ones vector. Moreover, by Friedman's second eigenvalue theorem~\cite{Friedman-second-eigenvalue}, for any fixed $\epsilon>0$, we have with probability $1-o(1)$, if $G\sim \mathcal G_d(n)$, 
        \[
        \Big\|\frac{1}{d}A-\frac{1}{n}\bone\bone^{\sT}\Big\|_{\op}\leq \frac{2\sqrt{d-1}+\eps}{d}\leq \frac{2}{\sqrt{d}},
        \]
        where the last inequality holds for small enough $\eps>0$. Consequently, on the high-probability event for $G$ where the above inequality holds, the following holds. For $S\subset V$, let $\bone_S\in \R^n$ denotes the vector with entries $(\bone_S)_i=\one\{i\in S\}$. Then, for any $S\subset V$, 
        \[
         \Big\|\Big(\frac{1}{d}A-\frac{1}{n}\bone\bone^{\sT}\Big)\bone_{S}\Big\|_2^2\leq \frac{4}{d}|S|.
        \]
        At the same time, 
        for any $S$ such that $|S|\leq \delta n$,
        \[
          \Big\|\Big(\frac{1}{d}A-\frac{1}{n}\bone\bone^{\sT}\Big)\bone_{S}\Big\|_2^2
          =\sum_{u\in V}\Big(\frac{\deg_S(u)}{d}-\frac{|S|}{n}\Big)^2\geq (\eta-\delta)^2|\{u\in V:\deg_S(u)\geq \eta d\}|\,.
        \]
        The desired claim follows by combining these last two inequalities. 
    \end{proof}

    \begin{lemma}\label{lem:expansion-edgenumber-facts-configuration-model}
        Suppose $d\ge 7$. With probability $1-o(1)$, $G\sim \cG_{d}(n)$ is such that for all $S$ with $|S|\le 10\gamma_0 n$, we have 
        \begin{align*}
            |\{v: \deg_S(v) > 3d/7\}|\le \frac{|S|}{3}\,.
        \end{align*}
    \end{lemma}
    \begin{proof}
    It suffices to establish this for the configuration model, due to the contiguity of $\mathcal G_d(n)$ with the corresponding configuration model. 
        Consider any fixed $S$ with $|S|\le \gamma n$, of which there are $\binom{n}{\gamma n} \le e^{  \gamma n \log (1/\gamma)}$ many. In the configuration model, the number of vertices that are incident to at least $rd > 3d/7$ (for $d=7$, we take $r=4/7$ and for $d \ge 8$ simply take $r=3/7$) vertices of $S$ is dominated by $$\text{Binom}\Big(n,\binom{d}{rd}\Big(\frac{|S|}{n-|S|}\Big)^{rd}\Big) \preceq \text{Binom}\Big(n, e^{ rd\log (1/r)} (\tfrac{\gamma}{1-\gamma})^{rd}\Big)\,,$$
        as even conditional on all other matchings on the half-edges of $S$ this bounds the probability of matching  $rd>3d/7$ of the $d$ half-edges of a vertex $v\notin S$ to half-edges of $S$. 
        
        The expected value of the right-hand side is $n e^{ rd ( \log \frac{1}{r} + \log \frac{\gamma}{1-\gamma})}$ which for $\gamma <1/100$ is at most $n e^{ - 12 d/7}\le n e^{ - 12}$. Therefore, if $H(\cdot \,||\, \cdot)$ denotes the relative entropy, its probability of exceeding $\frac{\gamma n}{3}$ is at most 
        $$\exp\Big(- n H(\frac{\gamma}{3} \,||\, e^{ rd (\log \frac{1}{r} + \log \frac{\gamma}{1-\gamma})})\Big) \le \exp\Big(- \frac{\gamma n}{3} \log \Big[3 \gamma^{- rd-1} e^{-rd (\log (1/r) - \log (1-\gamma))}\Big]\Big)\,.$$
        Now observe that for $r\le 4/7$ and $\gamma \le 1/100$, we have $\log (1/r) - \log (1-\gamma)\le 1 \le \log 3$ and so we get that the negative exponent is at least $\frac{\gamma n}{3} rd \log (1/\gamma)$. 
        The number of such $S$ we recall is at most $e^{ \gamma n \log(1/\gamma)}$. Since $\frac{rd}{3} \ge \frac{8}{7}$, this is exponentially small in $\gamma n \log (1/\gamma)$ and can be summed over all $\gamma n \le 10\gamma_0 n \le n/100$ to conclude the proof. 
    \end{proof}

    The last a priori estimate we need before analyzing the drift of the legacy region and concluding the proof of Proposition~\ref{prop:extinction-time-of-legacy} is a lemma that shows that the magnetization remains positive for a long time if it starts positive. Let $\tau^{m^X}_{r}$ be the hitting time of $r$ for the magnetization of the two-spin dynamics $(X_t)_{t\ge 0}$.

    \begin{lemma}\label{lem:retain-large-magnetization}
    Fix $d\ge 7$ and $\beta \ge \beta_0$. Suppose $G$ satisfies the properties of Lemma~\ref{lem:expansion-edgenumber-facts-configuration-model}. Consider any initialization $x_0$ with $m^X(x_0) \ge 1-2\gamma_0$. There exists a constant $C$ such that  
    \begin{align*}
        \mathbb P(\tau^{m^X}_{1-4\gamma_0} \le e^{ n/C}) \le e^{ - n/C}\,.
    \end{align*}
    \end{lemma}

    This estimate is a much simpler drift analysis than those on the legacy region and spacetime minus clusters, because at low-temperatures, the magnetization at time $t$ has a drift away from $0$ uniformly over $(X_s)_{s<t}$. We therefore defer its proof to Section~\ref{sec:magnetization-analysis} to first present the heart of the remaining proof of the negative drift of the legacy region.

\begin{proof}[\textbf{\emph{Proof of Proposition~\ref{prop:extinction-time-of-legacy}}}]
    We work on the high probability event that $G$ is $1$-locally-treelike and satisfies the properties of Lemma~\ref{lem:expansion-edgenumber-facts:alternate} with $\delta = 10\gamma_0$ and $\eta = \frac{3}{7}$, and of Lemma~\ref{lem:expansion-edgenumber-facts-configuration-model}. 
    
    Initially, $|\cL_0|\le \gamma_0 n$. We compute the drift of $|\cL_t|$ while $t\le n^2$ and $t\le \tau^{\cL}$, where 
    \begin{align*}
        \tau^{\mathcal L}: = \inf \big\{t \ge 0: |\cL_t| \ge 2 \gamma_0 n\big\}\,.
    \end{align*}
    By Lemma~\ref{lem:retain-large-magnetization} and the fact that $m^X(X_t) \le 1-\frac{2|\cL_t|}{n}$ (as legacy vertices must be minus), we have  
    \begin{align}\label{eq:legacy-bad-hitting-times-are-large}       
    \mathbb P(\tau^{\cL} \le n^2) = O(n^{-10})\,,
    \end{align}
    and it will ultimately suffice to just show that the drift on $t\le n^2 \vee \tau^{\cL}$ is negative enough to ensure hitting $\{\cL_t = \emptyset\}$ in time $t= O(\log n)$ with high probability.

    We can write the expected drift for $|\cL_t|$ as follows, recalling the definition of the flip rate $c_{v,t}= c_{v,t}(X_{t^-})= \lim_{\epsilon\downarrow 0} \frac{1}{\epsilon}\mathbb P(X_{t+\epsilon}(v) \ne X_{t^-}(v) \mid \cF_{t^-})$: 
    \begin{align}\label{eq:legacy-region-drift}
        \frac{d}{dt} \mathbb E[|\cL_t|] & =  \mathbb E\Big[\sum_{v} c_{v,t} (|\cL^v_t| - |\cL_{t^-}|)\Big] \nonumber \\ 
        & \le \mathbb E\Big[ \mathbb E\Big[\sum_{v} c_{v,t} (|\cL^v_t| - |\cL_{t^-}|) \mid \cF_{\cL,t^-}\Big] \mathbf 1\{t\le \tau^\cL\}\Big] + n^2 \mathbb P(\tau^\cL \le n^2)\,,
    \end{align}
    where we are using $\cL_t^v
    $ to denote the legacy region after flipping $v$ at time $t$.
    Since there is no forward conditioning in this case, $c_{v,t}$ is simply the probability of a spin flip at $v$ if its clock rings at time $t$ given $X_{t^-}$. Now fix any $t^-$-admissible legacy spacetime cluster $L$ compatible with $\mathbf 1\{t\le \tau^\cL\}$ with slice at $t^-$ denoted $L_{t^-}$, and consider the conditional expectation given $\mathcal E_{t^-,L}$: 
    \begin{align*}
        \mathbb E\Big[ \sum_v c_{v,t} (|\cL_t^v| - |L_{t^-}|) \mid \cE_{t^-,L}\Big]\,.
    \end{align*}
    We split the contribution to the sum into two cases: 
    \begin{enumerate}[label=(\Alph*)]
        \item  $v \in L_{t^-}$: such vertices satisfy $X_{t^-}(v) = -1$ and are flipping to $+1$;
        \item  $v\in \mathsf{Nbd}$ are vertices at distance exactly $1$ from $L_{t^-}$ and necessarily have $X_{t^-}(v) = +1$ and are flipping to $+1$. 
    \end{enumerate}
     Here, the vertices $v$ at distance at least two from $L_{t^-}$ can be neglected since for those vertices, one necessarily has $\cL_t^v = L_{t^-}$. 
    Then, we can clearly write  for $t\le n^2$
    \begin{align*}
        \mathbb E\Big[\sum_{v} c_{v,t} (|\cL^v_t| - |L_{t^-}|) \mid \cE_{t^-,L}\Big] &= \mathbb E\Big[\sum_{v\in L_{t^-}} c_{v,t} (|\cL^v_t| - |L_{t^-}|)\mid \cE_{t^-,L} \Big]  + \mathbb E\Big[\sum_{v\in \mathsf{Nbd}} c_{v,t} (|\cL^v_t| - |L_{t^-}|)\mid \cE_{t^-,L} \Big]\,.
    \end{align*}

    \smallskip
    \noindent (A) 
    If a vertex in $L_{t^-}$ flips at time $t$, necessarily $|\cL_t^v| - |L_{t^-}| =-1$. Among such summands, for any vertex $v\in L_{t^-}$ having $\deg_{L_{t^-}}(v)\le \frac{3d}{7}$, since its other neighbors are $+1$, it has flip rate $c_{v,t} \ge 1-e^{ - 2\beta d/7}$. 
    Since $t\le \tau^\cL$, we get a contribution of at most  
    \begin{align*}
        \mathbb E\Big[\sum_{v\in L_{t^-}} c_{v,t} (|\cL^v_t| - |L_{t^-}|)\mid \cE_{t^-,L} \Big] & \le - (1 - e^{ - 2\beta d/7}) \Big|\Big\{v\in L_{t^-}: \deg_{L_{t^-}}(v) \le \frac{3d}{7}\Big\}\Big| \\ 
        & \le -  (1 - e^{ - 2\beta d/7}) \Big(1- \min \Big\{ \frac{4}{d(\frac{3}{7} - 10\gamma_0)^2} ,  \frac{1}{3}\Big\}\Big) |L_{t^-}|\,,
    \end{align*}
    where we used the bound of Lemma~\ref{lem:expansion-edgenumber-facts-configuration-model} for the $\frac{1}{3}$ bound, and Lemma~\ref{lem:expansion-edgenumber-facts:alternate} to get the large-$d$ behavior. Taking expectation over the realizations of $\{\cL_{<t} = L\}$, and changing the indicator $\mathbf 1\{t\le \tau^\cL\}$ to $1$ paying an additional additive error of $n\mathbb P(\tau^\cL \le n^2)$, the contribution to~\eqref{eq:legacy-region-drift} from $v\in L_{t^-}$ is at most 
    \begin{align*}
       - (1-e^{ - 2\beta d/7})  \Big(1- \min \Big\{ \frac{4}{d(\frac{3}{7} - 10\gamma_0)^2} ,  \frac{1}{3}\Big\}\Big) \mathbb E[|\cL_{t^-} |]  + n \mathbb P(\tau^{\cL} \le n^2)\,.
    \end{align*}
    
    \smallskip
    \noindent (B) 
    If $v\in \mathsf{Nbd}$, then if it does flip from $+1$ to $-1$, the change $|\cL_t^v| - |L_{t^-}|$ is given by $1+ |\bigcup_{w\sim v: w\notin L_{t^-}} \cR_{t^-}(w)|$. Thus the total contribution of this term is split as 
    \begin{align*}
        \mathbb E \Bigg[ & \sum_{\substack{v\in \mathsf{Nbd} \\  \deg_{L_{t^-}}(v) > 3d/7}} c_{v,t} \Big(1 + \Big|\bigcup_{w\sim v: w\notin L_{t^-}} \cR_{t^-}(w)\Big|\Big) \mid \cE_{t^-,L}\Bigg] \\
        & \qquad \qquad \qquad \qquad \qquad +  \mathbb E \Bigg[ \sum_{\substack{v\in \mathsf{Nbd} \\ \deg_{L_{t^-}}(v) \le 3d/7}} c_{v,t} \Big(1 + \Big|\bigcup_{w\sim v: w\notin L_{t^-}} \cR_{t^-}(w)\Big|\Big) \mid \cE_{t^-,L}\Bigg]\,.
    \end{align*}
    
   On $t\le \tau^\cL$, by Lemma~\ref{lem:expansion-edgenumber-facts:alternate} and Lemma~\ref{lem:expansion-edgenumber-facts-configuration-model}, for at most $\min \{ \frac{4}{d(\frac{3}{7} - 10\gamma_0)^2}, \frac{1}{3}\}|L_{t^-}|$ many $v$'s does the vertex have more than $3d/7$ neighbors in $L_{t^-}$. For these, we take the trivial bound of $1$ on $c_{v,t}$, so that the first expected value above is at most 
   \begin{align*}
       \min \Big\{ \frac{4}{d(\frac{3}{7} - 10\gamma_0)^2}, \frac{1}{3}\Big\}|L_{t^-}| \cdot \max_{v\in \mathsf{Nbd}: \deg_{L_{t^-}}(v) > 3d/7} \mathbb E \Big[ 1+ \Big|\bigcup_{w\sim v: w\notin L_{t^-}} \cR_{t^-}(w)\Big| \mid \cE_{t^-,L}\Big] \,.
   \end{align*}
   By the domination relation of Corollary~\ref{cor:domination-by-i.i.d.-Psi} and linearity of expectation, this is at most  
   \begin{align}\label{eq:contribution-big-Legacy-degree}
       \min \Big\{ \frac{4}{d(\frac{3}{7} - 10\gamma_0)^2}, \frac{1}{3}\Big\}|L_{t^-}|  \Big(1+ d \max_{w} \mathbb E[ \wR_{t^-}(w)] \Big) \le  \min \Big\{ \frac{4}{d(\frac{3}{7} - 10\gamma_0)^2}, \frac{1}{3}\Big\}|L_{t^-}|  \Big(1+ d^{ - 99}\Big)\,,
   \end{align}
   where the latter bound followed from Corollary~\ref{cor:expectation-bound} on the rigid minus region's expected size.  

   The other contribution is at most 
   \begin{align*}
       \sum_{\substack{v\in \mathsf{Nbd} \\ \deg_{L_{t^-}}(v)\le 3d/7}} \Bigg( \mathbb E\big[c_{v,t} \mid \cE_{t^-,L}\big] + \mathbb E\Big[ \big| \bigcup_{w\sim v: w\notin L_{t^-}}\cR_{t^-}(w)\big| \mid \cE_{t^-,L}\Big]\Bigg)\,. 
   \end{align*}
    The latter expectation is at most $d^{ - 99}$ per Corollary~\ref{cor:domination-by-i.i.d.-Psi} and Corollary~\ref{cor:expectation-bound} as it was in the previous bound. 
   For the bound on the conditional expectation of $c_{v,t}$, we use that if $v$ has $\deg_{L_{t^-}}(v)\le 3d/7$, if all the other neighbors of $v$ are $+1$, then it has rate bounded by $e^{ - 2\beta d/7}$, to get 
    \begin{align*}
        \mathbb E[c_{v,t} \mid \cE_{t^-,L}] &\le \mathbb P\Big(|\bigcup_{w\sim v: w\notin \cL_{t^-}} \mathcal R_t(w)| \ge 1 \mid \cE_{t^-,L}\Big) + e^{ - 2 \beta d/7} \\ 
        & \le d \max_{w\sim v: w\notin L_{t^-}} \mathbb P\big( |\mathcal \wR_t(w)|\ge 1 \big) + e^{ - 2\beta d/7} \\ 
        & \le d^{ - 99} + e^{ - 2\beta d/7}
    \end{align*}
    where we again used the domination of Corollary~\ref{cor:domination-by-i.i.d.-Psi} to drop the conditioning and move to the rigid dynamics, and then the bound of Corollary~\ref{cor:expectation-bound} to conclude. 
    
    Combining the above bounds, and then taking the outer expectation in~\eqref{eq:legacy-region-drift} over $L$, while on the terms from $v\in \mathsf{Nbd}$ dropping the indicator $\mathbf 1\{t\le \tau^\cL\}$, we get that for all $t\le n^2$, 
\begin{align*}
    \frac{d}{dt} \mathbb E[|\cL_{t}|] & \le - (1 - e^{ - 2\beta d/7}) \Big(1- \min \Big\{ \frac{4}{d(\frac{3}{7} - 10\gamma_0)^2} ,  \frac{1}{3}\Big\}\Big) \mathbb E[|\cL_{t^-}|]  \\
    & \qquad + \min \Big\{ \frac{4}{d(\frac{3}{7} - 10\gamma_0)^2}, \frac{1}{3}\Big\}\mathbb E[|\cL_{t^-}|]  \big(1+ d^{ - 99}\big) + d \mathbb E[|\cL_{t^-}|] (d^{-99} + e^{ - 2\beta d/7}) \\
    & \qquad + 2n^2 \mathbb P(\tau^\cL \le n^2)\,.
\end{align*}  
By~\eqref{eq:legacy-bad-hitting-times-are-large} the last term is $O(n^{ - 8})$. To combine the first three terms, we use that for all $d$, $\min \{ \frac{4}{d(\frac{3}{7} - 10\gamma_0)^2}, \frac{1}{3}\} \le 1/3$  and if $\beta = \frac{C_0 \log d}{d}$ for a large universal constant $C_0$, then $e^{ - 2\beta d/7} \le d^{ - 10}$. As a result, for all $d\ge 7$, for all $t\le n^2$, 
\begin{align*}
    \frac{d}{dt} \mathbb E[|\cL_t|] &  \le  \mathbb E[|\cL_{t^-}|]\Big(- \frac{2}{3} +  \frac{1}{3} + 4d^{ - 8}\Big) + O(n^{-8}) \\ 
    & \le - \frac{1}{4} \mathbb E[|\cL_{t^-}|]\,,
\end{align*}
(and at $d$ large, $\frac{1}{4}$ can be replaced by $1-o_d(1)$). 

    By Gronwall's inequality, for $t\le n^2$, 
    \begin{align*}
        \mathbb E[|\cL_t|] \le |\cL_0| e^{ - t/4}\,.
    \end{align*}
    By Markov's inequality and the trivial bound $|\cL_0| \le n$, we therefore get for $t > 40\log n$, 
    \begin{align*}
        \mathbb P(|\cL_t| \ne 0) \le  n^{-10}\,,
    \end{align*}
    as claimed. Notice that for $d$ large, $t= (1+ o_d(1))\log n$ time would suffice to get $o(1)$-probability of the legacy surviving to time $t$. 
\end{proof}

\subsection{Retaining proximity to the metastable measure}

We now conclude the proof of Theorem~\ref{thm:from-d-independent-epsilon} showing that after $O(\log n)$ time, from sufficiently biased initializations, the total-variation distance of the Potts Glauber dynamics to the state-$1$ metastable measure, $\pi^1 = \pi(\cdot \mid \Omega^1)$ from~\eqref{eq:Potts-phase-magnetization} is small for exponential time.

\begin{proof}[\textbf{\emph{Proof of Theorem~\ref{thm:from-d-independent-epsilon}}}]
Consider the coupling from Lemma~\ref{lem:disagreement-subset-of-legacy} of Potts and two-spin dynamics chains $(Y_t^{y_0}, Y_t^1, X_t^{x_0})_{t\ge 0}$ with $\betap=2 \beta+\frac{C\log (q-1)}{d}$ and with  $x_0 (v) = +1$ wherever $y_0(v) = 1$ and $x_0(v) = -1$ where $y_0(v) \ne 1$. Under that coupling, by Lemma~\ref{lem:disagreement-subset-of-legacy}, 
\begin{align*}
    \mathbb P(Y_t^{y_0} \ne Y_t^1) = \mathbb P(\cD_t \ne \emptyset) \le \mathbb P(\cL_t \ne \emptyset)\,.
\end{align*}
Since we assumed $m(y_0) \ge 1-\gamma_0$, we have at least $(1-\gamma_0) n$ many $+1$ spins in $x_0$. Thus, 
 Proposition~\ref{prop:extinction-time-of-legacy}  established that with probability $1-o(1)$, $G$ is such that for any initialization $x_0$ satisfying $|\cL_0| \le \gamma_0 n$, we have for all $\beta \ge \beta_0=  \frac{C_0 \log d}{d}$, that the right-hand side above is $O(n^{-10})$ at time $t\geq C \log n$. 
 
We also have the same for $y_0 \sim \pi^1$ from~\eqref{eq:Potts-phase-magnetization}  because (e.g., as a consequence of~\cite{DeMo10} in the Ising case and~\cite{BaDeSl-Potts-treelike,CGGRSV-Metastability-Potts-ferromagnet} in the Potts cases), with probability $1-o(1)$, $G$ is such that for large $C$, there exists $c(C)>0$ such that for $\betap>\frac{C \log q}{d}$, we have 
\begin{align}\label{eq:stationary-magnetization-bottleneck}
    \pi^1(\sigma: m(\sigma) \ge 1- \gamma_0) \ge  1-e^{-cn}\,.
\end{align}
By a triangle inequality with total-variation distance, this implies that with probability $1-o(1)$, $G$ is such that for all $\beta \ge \beta_0 = \frac{C_0 \log d}{d}$, if $t \ge C\log n$
\begin{align}\label{eq:close-to-pi^+-init}
    \|\mathbb P(Y_{t}^{y_0}\in \cdot ) - \mathbb P( Y_{t}^{\pi^1}\in \cdot)\|_\tv   = O(n^{-10})\,.
\end{align}
Now consider a restricted chain $\widehat Y_t^{\pi^1}$ which is initialized from $\pi^1$  and rejects any updates that take it outside $\Omega^1$. On the one hand, this chain is coupled perfectly to $Y_t^{\pi^1}$ until the hitting time $\widehat \tau = \inf\{t: m(\widehat{Y}_t^{\pi^1}) \le \frac{2}{n}\}$. On the other hand, $\pi^1$ is stationary for $\widehat Y_t^{\pi^1}$. Thus, 
\begin{align*}
    \|\mathbb P(Y_t^{\pi^1} \in \cdot) - \pi^1\|_{\tv} \le \mathbb P(\widehat \tau \ge t)\,.
\end{align*} 
Finally, by a union bound, we can bound $\mathbb P(\widehat \tau \ge  t)$, for any $t \in [C\log n, e^{ n/C}]$ by the probability that there are more than $2nt$ clock rings in times $[0,t]$ is at most $n^{-10}$, and the probability that in any one of those, $\widehat Y_t^{\pi^1}$ hits  $m(Y_t^{\pi^1}) \le \frac{2}{n}$ is at most the probability of that under $\pi^1$ which was bounded by~\eqref{eq:stationary-magnetization-bottleneck}. As such, with probability $1-o(1)$, $G$ is such that for every $t\in [C\log n,e^{n/C}]$ for $C$ large enough, one has 
\begin{align*}
    \mathbb P(\widehat \tau \ge t)\le  2nt e^{ - cn} + O(n^{-10})   = O(n^{-10})\,,
\end{align*}
which combined with the preceding two displays and a triangle inequality concludes the proof. 
\end{proof}

\section{Drift for the magnetization process away from zero}\label{sec:magnetization-analysis}

In this section, we first analyze the drift of the magnetization process for the two-spin dynamics $(X_t)_{t\ge 0}$ and show that if it starts sufficiently close to $1$, then it will stay close to $1$ for exponentially long times. This will prove the deferred Lemma~\ref{lem:retain-large-magnetization} which was the last step for the proof of Theorem~\ref{thm:from-d-independent-epsilon} for fast mixing from sufficiently biased initializations. Then using a very similar drift analysis of the magnetization for the Potts Glauber dynamics $(Y_t)_{t\ge 0}$, with different expansion lemmas, will give us that for $d$ sufficiently large, the requisite initial bias can be taken arbitrarily close to $0$ boosting Theorem~\ref{thm:from-d-independent-epsilon} to deduce Theorem~\ref{thm:main-Potts}. 

Note that the arguments in this section are significantly simpler than those in the previous sections, following from a direct drift analysis, and martingale concentration, applied to a magnetization process. They show that the magnetization remains significantly biased for long times, but are not refined enough to show that it converges to the stationary measure's magnetization, let alone show anything about the (quasi-)mixing of the full process.

\subsection{Deferred proof of retaining magnetization bias}
In this subsection, $(X_t)_{t\ge 0}$ is the two-spin dynamics of Definition~\ref{def:general-ising-process} initialized from $x_0$ having $m^X(x_0)  = \frac{1}{n}\sum_v x_0(v) \ge 1-2\gamma_0$. The quantity $c_{v,t}= c_{v,t}(X_{t^-})$ is denoting the flip rate $\lim_{\epsilon \downarrow 0} \frac{1}{\epsilon} \mathbb P(X_{t+\epsilon}(v) \ne X_t(v) \mid \cF_{t^-})$. 

\begin{proof}[\textbf{\emph{Proof of Lemma~\ref{lem:retain-large-magnetization}}}]
Let $\cS_t = \cS(X_t)$ be the set of minus sites in $X_t$ and observe that $m_t^X = 1- \frac{2 |\cS_t|}n$ is the (normalized) magnetization process. We split the contribution to the drift of $|\cS_t|$ according to different sets of vertices. For any set $S$, let $V^g(S)$ (using $g$ to indicate ``good") be the set of vertices in $V$ having at most $3 d/7$ neighbors in $S$. Then, if we write $V^g_{t^-} = V^g(\cS_{t^-})$, we can split 
\begin{align*}
    \frac{d}{dt} \mathbb E[|\cS_t| \mid \cF_{t^-}] & = - \sum_{v\in \cS_{t^-}\cap V^g_{t^-}} c_{v,t} - \sum_{v\in \cS_{t^-} \setminus V^g_{t^-}} c_{v,t}  +\sum_{v\in \cS_{t^-}^{c}\cap V^g_{t^-}} c_{v,t} + \sum_{v\in \cS_{t^-}^c \setminus V^g_{t^-} } c_{v,t}\,.
\end{align*}
We aim to give an upper bound on this quantity, so we can drop the second sum as all rates are non-negative. To bound the first term, we note that any such $v$ is flipping from $-1$ to $+1$ and has at least $\frac{4d}{7}$ many plus neighbors so for $v\in \cS_{t^-} \cap V^g$, by assumption (B) in Definition~\ref{def:general-ising-process}, $c_{v,t}\ge \frac{1}{1+e^{-2\beta d/7}}$. For the third term, these vertices are flipping from $+1$ to $-1$ while having at most $3d/7$ many $-1$ neighbors, so their rate satisfies the bound $c_{v,t} \le e^{ - 2\beta d/7}$. Finally, for the fourth term, we can trivially bound the flip rates by $1$. Combining, we get 
\begin{align}\label{eq:minus-set-drift}
    \frac{d}{dt}  \mathbb E[|\cS_t| \mid \cF_{t^-}] \le - (1+e^{ - 2\beta d/7})^{-1} |\cS_{t^-} \cap V^g_{t^-}| + n e^{ - 2\beta d/7} +  | V \setminus V^g_{t^-}|\,.
\end{align}
We now work on the probability $1-o(1)$ event that our graph $G\sim \mathcal G_d(n)$ satisfies the following: for $\eta = \frac{3d}{7}$, for every $S$ with $|S|\le 10\gamma_0 n$, the number of vertices in $\{v: \deg_S(v) >3d/7\}$ is at most $|S|/3$. 
This event has probability $1-o(1)$ by Lemma~\ref{lem:expansion-edgenumber-facts-configuration-model} for all $d\ge 7$. 

On this event, while $|\cS_{t^-}|\le 10\gamma_0 n$, one has $|\cS_{t^-} \cap V^g_{t^-}|\ge  \frac{2|\cS_{t^-}|}{3}$. At the same time, we also can bound  $|V\setminus V^g_{t^-}|\le  \frac{|\cS_{t^-}|}3$. Plugging these bounds in to~\eqref{eq:minus-set-drift}, we get 
\begin{align*}
    \frac{d}{dt}  \mathbb E[|\cS_t| \mid \cF_{t^-}] \le - (1-e^{ - 2\beta d/7} - \frac{1}{3}) |\cS_{t^-}| + ne^{ - 2\beta d/7}\,.
    \end{align*}
If $\beta>\frac{C_0 \log d}{d}$ for a large $d$-independent constant $C_0$, then while $|\cS_{t^-}| \in [\frac{\gamma_0}{10}, 10 \gamma_0]$, this satisfies $\frac{d}{dt}  \mathbb E[|\cS_t| \mid \cF_{t^-}] \le - \frac{1}{2} |\cS_{t^-}|$. This is equivalent, using $m_t^X = 1- 2 |\cS_{t}^-|/n$, to
\begin{align*}
    \frac{d}{dt} \mathbb E[m_t^X \mid \cF_{t^-}] \ge 1- m_t^X \ge \frac{\gamma_0}{5}\,, \qquad \text{while $t\le \tau^{m^X}_{1-5\gamma_0} \vee \tau^{m^X}_{1-\gamma_0/5}$}\,.
\end{align*}
Now consider the continuous-time Doob martingale 
    \begin{align*}
        N_t := m_t^X - m_0^X - \int_0^t \frac{d}{ds} \mathbb E[m_s^X \mid \mathcal F_{s^-}]ds\,.
    \end{align*}
     This is a jump martingale with jumps that are almost surely bounded by $2/n$ and square bracket $[N]_t \le t/n$ (recall that the square bracket of $N_t$ is the predictable process for which $\partial_t \mathbb E[N_t^2 - [N]_t\mid \cF_{t^-}]= 0$). Thus, by Doob's maximal inequality and Azuma's inequality, 
     \begin{align*}
         \mathbb P \Big (\sup_{t\le T} |N_t| \ge r\Big) \le 2\exp( - r^2 n/8T)\,.
     \end{align*}
     This implies that except with probability $2e^{ - r^2 n/8T}$, 
     \begin{align*}
         m_t^X \ge m_0^X + \int_0^t \frac{d}{ds} \mathbb E[ m_s^X \mid \cF_{s^-}] ds + N_t \ge m_0^X + \int_0^t F(m_s^X) ds - r\,,
     \end{align*}
     where $F(m_s^X)$ is at least $1-m_s^X$ for $m_s^X \in [1-5\gamma_0, 1-\gamma_0/5]$ and is $0$ on $[1-\gamma_0/5, 1]$. 
     If $\bar m_t$ is the process solving the $1$-dimensional ODE, $\bar m_t = m_0 + \int_0^t F(\bar m_s) ds$, then Gronwall's inequality and the fact that $F$ is $1$-Lipschitz implies that except with probability $2e^{- r^2 n/8T}$, one has 
     $$\sup_{t\le T}|m_t^X - \bar m_t| \le r e^{ 2T}\,.$$
     Taking $r = \gamma_0/1000$, we get that with probability $1-e^{ - cn/T}$ for a sufficiently small constant $c$, for all $t\le \tau_{1-5\gamma_0}^{m^X} \vee T$, if $m_0^X \ge 1- 2\gamma_0$, then 
     \begin{align*}
         m_t^X \ge \min \Big\{ m_0^X + \frac{\gamma_0 t}{5}, 1- \frac{\gamma_0}{5}\Big\} - \frac{\gamma_0}{1000}e^{2T} \,.
     \end{align*}
    Taking $T=1$, one gets $\inf_{t\in [0,1]} m_t^X\ge 1-4\gamma_0$ and $m_1^X \ge 1-2\gamma_0$. Repeating this bound $e^{ cn/2}$ many times using a union bound, we deduce that if $m_0^X \ge 1-2\gamma_0$, then $\mathbb P(\tau^{m^X}_{1-4\gamma_0} \le e^{cn/2}) \le e^{ - cn/2}$ as claimed. 
\end{proof}

\subsection{Allowing for small biases when $d$ is large}
We now use the same drift analysis to establish that for every $\betap>\frac{C_0 \log (qd)}{d}$, the Potts dynamics in $O(1)$ time attains $m(Y_t) \ge (1-\gamma_0) n$ even if it is initialized from a configuration with only magnetization $m(y_0)\ge \epsilon_d$ where $m(Y)$ is defined as in~\eqref{eq:Potts-phase-magnetization}. Combined with Theorem~\ref{thm:from-d-independent-epsilon}, this allows us for $d$ large to take a vanishing bias and still quasi-equilibrate to the plus phase.

\begin{lemma}\label{lem:from-d-dependent-initialization}
    For every $d$ large, there exists $\epsilon_0(\beta, d) = O_d(\max\{d^{ - 1/2}, (\beta d)^{-1}\})$ such that if $G \sim \mathcal G_d(n)$, the following holds with probability $1-o(1)$. For every $\beta\ge \frac{C_0 \log (qd)}{d}$, for every $\epsilon>\epsilon_0$ and every $X_0$ with $m(X_0)\ge \epsilon$, 
    \begin{align*}
        \mathbb P(\tau_{1-\gamma_0}^m> 6/\gamma_0) = \exp( - \Omega(n))\,.
    \end{align*}
\end{lemma}

We will use the following $q$-part generalization of the expansion estimate Lemma~\ref{lem:expansion-edgenumber-facts:alternate}. 

    \begin{lemma}\label{lem:expansion-edgenumber-facts:Potts}
    Fix any $d\geq 3, q\geq 2$, and constants $0<\eta<\delta<1$. Then, the following holds for $G\sim \mathcal G_d(n)$ with probability $1-o(1)$: for all $q$-partition $S_1,S_2,\ldots S_q$ of $V$, i.e. $\bigsqcup_{k\in[q]}S_k=V$, such that $|S_1|\geq \max_{2\leq k\leq q} |S_k|+\delta n$, we have
    \[
    \big|\{u\in V: \deg_{S_1}(v)\leq \max_{2\leq k\leq q} \deg_{S_k}(v)+\eta d\}\big|\leq \frac{8(q-1)}{d(\delta-\eta)^2}|S_1|\,,
    \]
    where $\deg_S(v)$ denotes the number of vertices in $S$ that are adjacent to $v$.
    \end{lemma}
    \begin{proof}
         Let $A\in \R^{n\times n}$ be the adjacency matrix of $G\sim \mathcal G_d(n)$. As shown in the proof of Lemma~\ref{lem:expansion-edgenumber-facts:alternate}, we have w.h.p. over $G\sim \mathcal G_d(n)$ that $\big\|d^{-1}A-n^{-1}\bone\bone^{\sT}\big\|_{\op}\leq 2d^{-1/2}$. We claim that the desired conclusion holds on this high-probability event. For $S\subset V$, let $\bone_S\in \R^n$ denotes the vector with entries $(\bone_S)_i=\one\{i\in S\}$, and consider $\bigsqcup_{k\in[q]}S_k=V$ such that $|S_1|\geq \max_{2\leq k\leq q}|S_k|+\delta n$. Then, for any $2\leq k \leq q$, we have
        $\big\|\big(d^{-1}A-n^{-1}\bone\bone^{\sT}\big)(\bone_{S_k}-\bone_{S_1})\big\|_2^2\leq 4d^{-1}\big\|\bone_{S_k}-\bone_{S_1}\big\|_2^2\leq 8d^{-1}|S_1|$,
        where the last inequality holds since $|S_k|\leq |S_1|$.
        On the other hand, we have
        \[
        \begin{split}
          \Big\|\Big(\frac{1}{d}A-\frac{1}{n}\bone\bone^{\sT}\Big)(\bone_{S_k}-\bone_{S_1})\Big\|_2^2
          &=\sum_{v\in V}\Big(n^{-1}\big(|S_1|-|S_k|\big)-d^{-1}\big(\deg_{S_1}(v)-\deg_{S_k}(v)\big)\Big)^2\\
          &\geq (\delta-\eta)^2\big|\{v\in V:\deg_{S_1}(v)\leq \deg_{S_k}(v)+\eta d\}\big|\,,
          \end{split}
        \]
        so $\big|\bigcup_{k=2}^{q}\{v\in V:\deg_{S_1}(v)\leq \deg_{S_k}(v)+\eta d\}\big|\leq \frac{8(q-1)}{d(\delta-\eta)^2}|S_1|$ holds by a union bound.
    \end{proof}

\begin{proof}[\textbf{\emph{Proof of Lemma~\ref{lem:from-d-dependent-initialization}}}]
Let $\cS_{i,t} = \cS_i(X_t)$ be the set of sites in $(Y_t)_{t\ge 0}$ in state $i\in \{1,...,q\}$. We bound the drift for $\mathcal S_{1,t}$ and $\mathcal S_{i,t}$ for $i\in \{2,...,q\}$ separately. For an $\eta \in (0,1)$ and $q$-partition $S_1,...,S_q$, let $V^g= V^g(S_1,...,S_q)$ (using $g$ to indicate ``good") be the set of vertices in $V$ with 
\begin{align*}
     \deg_{S_1}(v) \ge \max_{2\le k\le q}\deg_{S_k}(v) + \eta d\,.
\end{align*}
Then, for $i=1,..,q$, if we write $V^g_{t^-} = V^g(\cS_{1,t^-},...,\cS_{q,t^-})$, we can split 
\begin{align*}
    \frac{d}{dt} \mathbb E \big[ |\cS_{i,t}| \mid \cF_{t^-} \big] = - \sum_{v\in \cS_{i,t^-}\cap V^g_{t^-}} c_{v,t} - \sum_{v\in \cS_{i,t^-} \setminus V^g_{t^-}} c_{v,t} + \sum_{v\in \cS_{i,t^-}^c \cap V^g_{t^-}} c_{v,t}^i + \sum_{v\in \cS_{i,t^-}^c \setminus V^g_{t^-}} c_{v,t}^i \,,
\end{align*}
where for a vertex that is in state $i$ in $Y_{t^-}$, we are using $c_{v,t} = \lim_{\epsilon \downarrow 0} \frac{1}{\epsilon}\mathbb P(Y_{t+\epsilon}(v) \ne Y_{t}(v) \mid \cF_{t^-})$ (where abusing notation now $\cF_{t^-}$ is the filtration generated by the Potts Glauber dynamics chain), and for vertices that are not in state $i$ in $Y_{t^-}$, we are using $c_{v,t}^i$ to denote its flip rate into state $i$. 

We aim to lower bound the above drift for $i=1$ and upper bound it for $i\ne 1$. For the lower bound for $i=1$, if $v\in \cS_{1,t} \cap V^g_{t^-}$, we have $c_{v,t} \le (q-1)/(q-1+e^{\betap \eta d}) \le (q-1)e^{ - \betap \eta d}$, and otherwise $c_{v,t}\le 1$ trivially. On the other hand, for the contributions of the positive terms, they have $c_{v,t}^i \ge 1-(q-1) e^{ - \betap \eta d}$ if $v\in \cS_{i,t^-}^c\cap V^g_{t^-}$, and otherwise $c_{v,t}^i\ge 0$ trivially. 
Thus, 
\begin{align*}
    \frac{d}{dt} \mathbb E[|\cS_{1,t}| \cF_{t^-}] \ge  - (q-1)e^{ - \betap \eta d} |\cS_{1,t^-}|   - |V\setminus V^g_{t^-}| + \big(1 - (q-1)e^{ - \betap \eta d}\big) |V^g_{t^-} \setminus \cS_{1,t^-}|\,.
\end{align*}
By analagous bounds for $i = \{2,...,q\}$ one has 
\begin{align*}
    \frac{d}{dt} \mathbb E[ |\cS_{i,t}|\mid \cF_{t^-}] \le - \big( 1-(q-1)e^{ - \betap \eta d}\big)|\cS_{i,t^-} \cap V^g_{t^-}| + |V^g_{t^-}| (q-1) e^{ - \betap \eta d} + |V\setminus V^g_{t^-}|\,.
\end{align*}
Taking the difference of these, we get for every $i\ne 1$, that 
\begin{align}\label{eq:drift-difference-of-state-1-and-other-state}
    \frac{d}{dt} \mathbb E \big[|\cS_{1,t}| - |\cS_{i,t}| \mid \cF_{t^-}]  & \ge \big( 1- (q-1)e^{ - \betap \eta d}\big) \big(  n - (|\cS_{1,t^-}| - |\cS_{i,t^-}|) - 2|V\setminus V^g_{t^-}|  \big) \nonumber \\
    & \qquad - (q-1) e^{ - \betap \eta d} (|\cS_{1,t^-}|  + |V^g_{t^-}|)  - 2|V \setminus V^g_{t^-}|\,.
\end{align}
    We work on the event that $G\sim \mathcal{G}_d(n)$ has for every $q$-partition $S_1,...,S_q$ with $|S_1| \ge \max_{2\le i\le q} |S_i| + \delta n$ for $\delta = \frac{\epsilon}{2}$, for every $\eta \in (0,\delta)$, the number of bad vertices 
    \begin{align*}
        V \setminus V^g(S_1,...,S_q) = \{ v\in V: \deg_{S_1}(v) \le \max_{i\ne 1} \deg_{S_i}(v) + \eta d\}\,,
    \end{align*}
    has size at most $\frac{8 (q-1)}{d(\delta - \eta)^2}|S_1|$.  For every $d\ge 3$, $q\ge 2$, this event holds with probability $1-o(1)$ for $G\sim \mathcal G_{d}(n)$ by Lemma~\ref{lem:expansion-edgenumber-facts:Potts}.
    With these assumptions, we have $|V\setminus V_{t^-}^g| \le \frac{8(q-1)}{d(\delta - \eta)^2} |\cS_{1,t^-}|$. Therefore, 
    \begin{align*}
        \frac{d}{dt} \mathbb E[|\cS_{1,t}| - |\cS_{i,t}| \mid \cF_{t^-}] &  \ge (n - (|\cS_{1,t^-}| - |\cS_{i,t^-}|)) - 4n(q-1)e^{- \betap \eta d} - \frac{16(q-1)}{d (\delta - \eta)^2}|\cS_{1,t^-}| \\ 
        & \ge \Big(1-\frac{16(q-1)}{d (\delta - \eta)^2}\Big)n - (|\cS_{1,t^-}| - |\cS_{i,t^-}|)) - 4n(q-1)e^{- \betap \eta d}
    \end{align*}
    Then, if we let $\eta = \frac{\epsilon}{4}$, and if $t\le \tau_{\frac{\epsilon}{2}}^m \vee \tau_{1-\frac{\gamma_0}{2}}^m$ where we recall that $m_t = \frac{1}{n}(|\cS_{1,t}| - \max_{i= 2,...,q} |\cS_{i,t}|)$, so that $|\cS_{1,t^-}| - |\cS_{i,t^-}| \in  [\frac{\epsilon}{2} n  , (1-\frac{\gamma_0}{4}) n]$, then for each $i=2,...,q$,
    \begin{align*}
        \frac{d}{dt} \mathbb E[\frac{1}{n}(|\cS_{1,t}| -  |\cS_{i,t}| )\mid \cF_{t^-}] \ge  \frac{\gamma_0}{4} -  \frac{256(q-1)}{\epsilon^2 d} - 4 (q-1)e^{ - \betap \epsilon d/4}\,.
    \end{align*}
    If we have $\epsilon  > C_0 d^{-1/2}$ and also have that $\betap \epsilon d > C_0 \log q$ where $C_0$ is a sufficiently large universal constant, then the right-hand side above is at least $\frac{\gamma_0}{5}$, say. 

        Controlling the contribution of the martingale part for each $i=2,...,q$ exactly as in the end of the proof of Lemma~\ref{lem:retain-large-magnetization}, we get that except with probability $e^{ - cn}$ for a sufficiently small constant $c$, after time $T = O(1)$ (e.g., $T= 6\gamma_0^{-1}$ suffices), the magnetization $m_t$ will have hit $\tau_{1-\frac{\gamma_0}{2}}^{m}$. 
\end{proof}

\begin{proof}[\textbf{\emph{Proof of Theorem~\ref{thm:main-Potts}}}]
    Fix any $d\ge 7$, any $\beta > \frac{C_0 \log d}{d}$, and $\betap = 2\beta + \frac{C_0 \log q}{d}$ for $C_0$ a sufficiently large $d$-independent constant. 
    Consider any initialization with $m(Y_0) \ge \epsilon_0$ where $$\epsilon_0(\beta, d) = \min\Big\{1-\gamma_0, \max\{\tfrac{C_0}{\sqrt{d}}, \tfrac{C_0}{\beta d}\}\Big\}\,.$$ Work on the $1-o(1)$ probability event for $G$ that it is $1$-locally-treelike and all expansion properties used in the previous theorems apply.  
    By Lemma~\ref{lem:from-d-dependent-initialization}, the hitting time $\tau_{1-\gamma_0}^{m}$ is $O(1)$ with probability $1-e^{ - \Omega(n)}$. By the strong Markov property, at $\tau_{1-\gamma_0}^m$ we can apply Theorem~\ref{thm:from-d-independent-epsilon} to conclude. 
\end{proof}

\bibliographystyle{plain}
\bibliography{references}

\end{document}